\documentclass{article}
\usepackage[T1]{fontenc}
\usepackage[english]{babel}
\usepackage{textcomp,lmodern}
\usepackage{mathtools,amssymb,amsthm}
\usepackage[babel=true]{microtype}
\usepackage{enumerate}
\usepackage{tikz}
\usetikzlibrary{babel}
\usetikzlibrary{cd}
\newtheorem{fact}{Fact}[section]
\newtheorem{lemma}[fact]{Lemma}
\newtheorem{conjecture}[fact]{Conjecture}
\newtheorem{corollary}[fact]{Corollary}
\newtheorem{proposition}[fact]{Proposition}
\newtheorem{theorem}[fact]{Theorem}
\theoremstyle{definition}
\newtheorem{definition}[fact]{Definition}
\newtheorem{question}[fact]{Question}
\newtheorem{remark}[fact]{Remark}
\newtheorem{example}[fact]{Example}
\DeclareMathOperator\Acc{Acc}
\DeclareMathOperator\Aut{Aut}
\DeclareMathOperator\Cl{cl} 
\DeclareMathOperator\Circ{Circ} 
\DeclareMathOperator\diag{diag}
\DeclareMathOperator\Id{Id}

\DeclareMathOperator\lcm{lcm}
\DeclareMathOperator\NR{NR}
\DeclareMathOperator\rank{rank}
\DeclareMathOperator\Rist{Rist}
\DeclareMathOperator\SRist{SRist}
\DeclareMathOperator\Stab{Stab}
\DeclareMathOperator\SStab{SStab}
\DeclareMathOperator\Sub{Sub}
\DeclareMathOperator\Sym{Sym}
\newcommand*\abs[1]{\lvert#1\rvert}
\newcommand*\Autf{\Aut_{\mathrm{f}}}
\newcommand*\Autfr{\Aut_{\mathrm{fr}}}
\newcommand*\defi[1]{\textbf{#1}}
\newcommand*\gen[1]{\langle#1\rangle}
\newcommand*\GGS{\textrm{GGS}}
\newcommand*\Grig{{\mathcal G}}
\newcommand*\level[1]{\mathcal L_{#1}}
\newcommand*\N{\mathbf{N}}
\newcommand*\portrait{{\mathcal P}}
\newcommand*\presentation[2]{\langle#1\,|\,#2\rangle}
\newcommand*\restr[2]{{#1}_{\mkern 1mu \vrule height 2ex\mkern2mu {#2}}}
\newcommand*\setst[2]{\{#1\,|\,#2\}}
\newcommand*\Subcl{\Sub_{\Cl}}
\newcommand*{\treesection}[2]{{#1}_{\mkern 1mu \vrule height 2ex\mkern2mu {#2}}}

\usepackage[colorlinks,breaklinks,bookmarks,plainpages=false,unicode=true]{hyperref}
\title{Weakly maximal subgroups of branch groups}
\author{Paul-Henry Leemann}
\date{\today}
\hypersetup{pdftitle={Weakly maximal subgroups of branch groups}, pdfauthor={Paul-Henry Leemann}, pdfsubject={Weakly Maximal Subgroups, First Grigorchuk Group, Branch Groups, Profinite Groups, GGS Groups}}
\begin{document}
\maketitle
\begin{abstract}
Let $G$ be a branch group acting by automorphisms on a rooted tree $T$. Stabilizers of infinite rays in $T$ are examples of weakly maximal subgroups of $G$ (subgroups that are maximal among subgroups of infinite index), but in general they are not the only examples.

In this article we describe two families of weakly maximal subgroups of branch groups. We show that, for the first Grigorchuk group as well as for the torsion \GGS{} groups, every weakly maximal subgroup belongs to one of these families.
The first family is a generalization of stabilizers of rays, while the second one consists of weakly maximal subgroups with a block structure.
We obtain different equivalent characterizations of these families in terms of finite generation, the existence of a trivial rigid stabilizer, the number of orbit-closures for the action on the boundary of the tree or by the means of sections.
\end{abstract}
\tableofcontents
%
%
%
%
%
%
%
%
%
%
\section{Introduction and statements of the main results}
Let $T$ be a locally finite spherically regular rooted tree. Among groups that act on $T$ by automorphisms, branch groups are of particular interest (see Section~\ref{Section:Definitions} for all the relevant definitions).
The class of branch groups contains finitely generated groups with interesting properties, such as being infinite torsion, or having intermediate growth.
Branch groups have interesting subgroup structure and some of them, but not all, do not have maximal subgroups of infinite index. This is the case, for example, of the first Grigorchuk group and of the torsion \GGS{} groups.

The next step in understanding the subgroup structure of branch groups is to study \defi{weakly maximal subgroups}, that is the maximal elements among the subgroups of infinite index.
The study of such subgroups began with the following result of Bartholdi and Grigorchuk on \defi{parabolic subgroups}, which are the stabilizers of rays $\Stab_G(\xi)$.
\begin{proposition}[\cite{MR1841750,MR2893544}]\label{Prop:BartholdiGrigorchuk}
If $G\leq\Aut(T)$ is weakly branch, then all the $\Stab_G(\xi)$ for $\xi\in\partial T$ are infinite and pairwise distinct.
Moreover, if $G$ is branch, then all these subgroups are weakly maximal.
\end{proposition}

The study of weakly maximal subgroups continued with an example of Pervova exhibiting a non-parabolic weakly maximal subgroup of the first Grigorchuk group, answering a question of Grigorchuk about the existence of such subgroups.
After that, the author together with Bou-Rabee and Nagnibeda proved the following two results.
\begin{theorem}[\cite{MR3478865}]\label{Thm:BRLN1}
Let $T$ be a regular rooted tree and $G\leq \Aut(T)$ be a finitely generated branch group.\footnote{The result in~\cite{MR3478865} is only stated for regular branch groups, but the proof can easily be adapted for general branch groups.}
Then, for any finite subgroup $F\leq G$ there exists uncountably many weakly maximal subgroups of $G$ containing $F$.
\end{theorem}
\begin{theorem}[\cite{MR3478865}]\label{Thm:BRLN3}
Let $G$ be the first Grigorchuk group or a branch \GGS{} group.
For any vertex $v$, the subgroup $\Stab_G(v)$ contains a weakly maximal subgroup $W$ that does not stabilize any vertex of level greater than the level of $v$.
\end{theorem}
The flavour of Theorem~\ref{Thm:BRLN1} is \emph{small subgroups are contained in many weakly maximal subgroups} while Theorem~\ref{Thm:BRLN3} can be thought of as \emph{big subgroups (vertex stabilizers) contain many weakly maximal subgroup}.
Indeed, Theorem~\ref{Thm:BRLN3} implies that $\Stab_G(v)$ contains a weakly maximal subgroup that is not a parabolic subgroup.
On the other hand, one important corollary of Theorem~\ref{Thm:BRLN1} is the existence for some branch groups of uncountably many weakly maximal subgroups that are not parabolic.
This result hold for any $G$ that admits a unique branch action and such that $G$ contains a finite subgroup $F$ that does not fix any point in $\partial T$; two conditions that hold in a lot of branch groups.

The aim of this paper is twofold: a better understanding of weakly maximal subgroups in general branch groups as well as a full description of weakly maximal subgroups for the particular cases of the first Grigorchuk group and of the \GGS{} torsion groups.

We first have the following general result.
\begin{lemma}\label{Lemma:WMaxInfinite}
Let $G\leq\Aut(T)$ be a group with $\Rist_G(v)$ infinite for all $v$.
Then every weakly maximal subgroup of $G$ is infinite.
\end{lemma}
We then study \defi{generalized parabolic subgroups}.
These are the setwise stabilizers $\SStab_G(C)$ of closed but not clopen subsets $C\subset\partial T$ such that the action $\SStab_G(C)\curvearrowright C$ is minimal.
For this special class of subgroups we are able to prove the following generalization of Bartholdi and Grigorchuk's results.
\begin{theorem}\label{Thm:GeneralizedParabolic}
Let $T$ be a rooted tree and $G\leq\Aut(T)$ be a branch group.
Then all generalized parabolic subgroups of $G$ are weakly maximal and pairwise distinct.
\end{theorem}
Here pairwise distinct is to be understood as $\SStab_G(C_1)\neq\SStab_G(C_2)$ if they are both generalized parabolic subgroups with $C_1\neq C_2$.
As a corollary, we obtain the following stronger versions of the results of~\cite{MR3478865}:
\begin{corollary}\label{Thm:NumberGeneralizedParabolic}
Let $G\leq\Aut(T)$ be a branch group and $F\leq G$ be any subgroup.
Let $\mathcal C_F$ be the set of all non-open orbit-closures of the action $F\curvearrowright\partial T$.
Then the function $\SStab_G(\cdot)\colon\mathcal C_F\to\Sub(G)$ is injective and has values in generalized parabolic subgroups (which are all weakly maximal).
\end{corollary}
Observe that in the above result, and unlike to~\cite{MR3478865}, we do not require $G$ to be finitely generated.
Moreover, the proof of it does not use the axiom of choice.
The same remarks hold for the following two corollaries.
\begin{corollary}\label{Cor:NumberGeneralizedParabolic}
Let $G\leq\Aut(T)$ be a branch group and $F\leq G$ be any subgroup.
Let $v$ be a vertex of $T$.
\begin{enumerate}
\item If all orbit-closures of $\Stab_F(v)\curvearrowright \partial T$ are non-open, then $F$ is contained in uncountably many generalized parabolic subgroup,
\item If all orbits of $\Stab_F(v)\curvearrowright \partial T$ are at most countable and closed, then $F$ is contained in a continuum of generalized parabolic subgroup.
If moreover $\Stab_F(v)$ is non-trivial, then $F$ is contained in a continuum of generalized parabolic subgroups that are not parabolic.
\end{enumerate}
\end{corollary}
Before stating the next corollary, we introduce some equivalence relation on subgroups of a branch group $G$.
We say that two generalized parabolic subgroups $\SStab_G(C_1)$ and $\SStab_G(C_2)$ are \defi{tree-equivalent} if and only if there exists an automorphism $\varphi$ of $\Aut(T)$ such that $\varphi(C_1)=C_2$, see~Definition~\ref{Definition:TreeEquivalence} for a general definition.
The important point to keep in mind is that it is a coarser equivalence relation than conjugation of subgroups and that parabolic subgroups form one class for this relation.
\begin{corollary}\label{Cor:NumberGeneralizedParabolic2}
Let $G\leq\Aut(T)$ be a branch group.
\begin{enumerate}
\item
If $G$ is not torsion-free, it contains a continuum of generalized parabolic subgroups that are not parabolic.
\item
If $G$ has elements of arbitrarily high finite order, there are infinitely many tree-equivalence classes of generalized parabolic subgroups that each contains a continuum of subgroups.
\item
If $G$ is torsion, there is a continuum of tree-equivalence classes of generalized parabolic subgroups that each contains infinitely many subgroups.
\end{enumerate}
\end{corollary}
In contrast with~\cite{MR3478865}, the above corollaries apply to subgroups that are not necessarily finite and we are able to show that finite (and other) subgroups are contained a continuum of weakly maximal subgroups.
For torsion groups we even obtain infinitely many tree-equivalence classes of weakly maximal subgroups.

Observe that, by~\cite[Theorem 6.9]{MR2035113}, if $G$ is a branch group which is torsion, then it satisfies the hypothesis of the second part of Corollary~\ref{Cor:NumberGeneralizedParabolic2}.
In particular, it has both infinitely many tree-equivalence classes of generalized parabolic subgroups that each contains a continuum of subgroups and a continuum of tree-equivalence classes of generalized parabolic subgroups that each contains infinitely many subgroups.

We also provide a version of Theorem~\ref{Thm:BRLN3} that holds in any self-replicating branch group.
\begin{theorem}\label{Thm:CopyOfWeaklyMaxInStab}
Let $G\leq\Aut(T)$ be a self-replicating branch group and let $W$ be a weakly maximal subgroup of $G$.
For any vertex $v$, there exists a weakly maximal subgroup $W^v$ contained in $\Stab_G(v)$ such that the section of $W^v$ at $v$ is equal to~$W$.
\end{theorem}
We then turn our attention on another family of weakly maximal subgroups, the so-called subgroups with a block structure, see Definition~\ref{defn:Block}.
Informally, $H\leq G$ has a block structure if it is, up to finite index, a product of copies of~$G$, some of them embedded diagonally.
If $G$ is finitely generated, the subgroups with a block structure are also finitely generated and therefore, they are at most countably many of them.
In the special case of the first Grigorchuk as well as for the torsion GGS groups, these subgroups coincide with the finitely generated subgroups,~\cite{MR4344374,FL2019}.
Subgroups with a block structure are a source of new examples of weakly maximal subgroups, and once again there are ``as much'' of them as possible.
\begin{proposition}\label{Prop:InfinitelyBlockSubGrps}
Let $G$ be either the first Grigorchuk group, or a torsion \GGS{} group.
Then there exists infinitely many distinct tree-equivalence classes, each of them containing infinitely many weakly maximal subgroups with a block structure.
\end{proposition}
The above results, following~\cite{MR3478865}, are mostly quantitative results.
We now give more qualitative results.
The main idea here is that a weakly maximal subgroup of a branch group should retain some properties of the full group.
There is, a priori, two distinct directions for this.
On one hand, we can ask if  $W$ is ``nearly branch'' and on the other hand if $W$ is ``almost of finite index'' in some sense. 
Recall that a subgroup $G\leq\Aut(T)$ is weakly branch if and only if $G\curvearrowright\partial T$ has exactly one orbit-closure and all the subgroups $\Rist_G(v)$ are infinite.
We first show that a weakly maximal subgroup $W$ of a branch group is always ``nearly branch'' in the following sense.
\begin{proposition}\label{Proposition:Dichotomy1}
Let $G\leq\Aut(T)$ be a branch group and let $W$ be a weakly maximal subgroup of $G$.
Then at least one of the following holds:
\begin{enumerate}
\item All the subgroups $\Rist_W(v)$ are infinite,
\item $W\curvearrowright\partial T$ has a finite number of orbit-closures.
\end{enumerate}
\end{proposition}
On the other hand, a finite index subgroup $H$ of a branch group $G$ has always infinite rigid stabilizers and for any vertex $v$ the section $\pi_v\bigl(\Stab_H(v)\bigr)$ has finite index in $\pi_v\bigl(\Stab_G(v)\bigr)$.
In branch groups, weakly maximal subgroups will always be ``almost of finite index'' in the following sense.
\begin{proposition}\label{Proposition:Dichotomy2}
Let $G\leq\Aut(T)$ be a branch group and let $W$ be a weakly maximal subgroup of $G$.
Then at least one of the following holds:
\begin{enumerate}
\item All the subgroups $\Rist_W(v)$ are infinite,
\item There exists a level $n$ such that for every $v$ in $\level{n}$, the section $\pi_v\bigl(\Stab_W(v)\bigr)$ has finite index in $G=\pi_v\bigl(\Stab_G(v)\bigr)$.
\end{enumerate}
\end{proposition}
For the particular case of the first Grigorchuk group and of torsion \GGS{} groups (which encompass Gupta-Sidki $p$-groups), we even have a full description of the weakly maximal subgroups.
They split into two classes: generalized parabolic subgroups and subgroups with a block structure.
More precisely, we have the following theorem that is summarized in Table~\ref{Table:WMaxSub}.
\begin{theorem}\label{Thm:ClassificationGrigGuptaSidki}
Let $G$ be either the first Grigorchuk group or a torsion \GGS{} group.
Let $W$ be a weakly maximal subgroup of $G$.
Then the following properties are equivalent:
\begin{enumerate}
\item $W$ has a block structure\label{Item:IntroGrig2},
\item There exists $n$ such that $\pi_v\bigl(\Stab_W(v)\bigr)$ has finite index in $G$ for every vertex of level~$n$\label{Item:IntroGrig4},
\item There exists $n$ such that $\pi_v\bigl(\Stab_W(v)\bigr)=G$ for every vertex of level~$n$\label{Item:IntroGrig4prime},
\item There exists a vertex $v$ with $\Rist_W(v)=\{1\}$\label{Item:IntroGrig6},
\item $W\curvearrowright\partial T$ has finitely many orbit-closures\label{Item:IntroGrig3},
\item $W$ is not generalized parabolic\label{Item:IntroGrig5},
\item $W$ is finitely generated.
\end{enumerate}
\end{theorem}
\begin{table}[htbp]
\begin{tabular}{|c|c|}
\hline
generalized parabolic	& 	weakly maximal with a block structure\\
\hline\hline
not finitely generated&	finitely generated\\ \hline
$\forall v:\Rist_W(v)$ is infinite	&	$\exists v:\Rist_W(v)=\{1\}$\\ \hline
$W\curvearrowright\partial T$ has infinitely many orbit-closures	&	$W\curvearrowright\partial T$  has finitely many orbit-closures\\ \hline
$\forall n\exists v\in \level{n}:[G:\pi_v(W)]$ is infinite	&	$\exists n\forall v\in \level{n}:\pi_v(W)=G$\\ \hline
\end{tabular}
\caption{The two classes of weakly maximal subgroups of $G$, where $G$ is either the first Grigorchuk group or a torsion \GGS{} group.}
\label{Table:WMaxSub}
\end{table}

Finally, we investigate further the particular case of the first Grigorchuk group.
Among other things, we gave a particular attention to sections and we prove that parabolic subgroups, as well as some generalized parabolic subgroups, behave well under taking the closure in the profinite topology. More precisely
\begin{proposition}\label{Proposition:ClosureWMaxGrig}
Let $\Grig\curvearrowright T$ be the branch action of the first Grigorchuk group.
For any finite subset $C$ of $\partial T$, we have
\[
	\overline{\Stab_\Grig(C)}=\Stab_{\overline{\Grig}}(C).
\]
If moreover $C$ is contained in one $\Grig$-orbit, then we also have
\[
	\overline{\SStab_\Grig(C)}=\SStab_{\overline{\Grig}}(C).
\]
\end{proposition}
We also give the first example of a weakly maximal subgroup of $\Grig$ that acts level-transitively on $T$.

\paragraph{Organization of the paper}
The next section contains the definitions and some useful reminders as well as preliminary results on weakly maximal subgroups in branch groups.

Section~\ref{Section:GeneralizedParabolicSubgroup} is devoted to the study of generalized parabolic subgroups.
It contains proofs of Theorems~\ref{Thm:GeneralizedParabolic} and of Corollaries~\ref{Thm:NumberGeneralizedParabolic} and~\ref{Cor:NumberGeneralizedParabolic}.

In Section~\ref{Section:NonRigidityTree} we introduce the notion of the \defi{non-rigidity tree} of a (weakly maximal) subgroup and study its properties.
This tool turns out to be of great interest for the study of (weakly maximal) subgroups of branch groups. 
Among other results, we prove there Lemma~\ref{Lemma:WMaxInfinite}, Corollary~\ref{Cor:NumberGeneralizedParabolic2} and give a characterization of generalized parabolic subgroups in terms of the non-rigidity tree.

The next section is about sections and ``lifting'' and contains the proof of Theorem~\ref{Thm:CopyOfWeaklyMaxInStab}.

Section~\ref{Section:BlockSubgroups} concerns the study of subgroups with a block structure.
It contains a general structural result about weakly maximal subgroups of branch groups which encompass Propositions~\ref{Proposition:Dichotomy1} and~\ref{Proposition:Dichotomy2} as well as Theorem~\ref{Thm:ClassificationGrigGuptaSidki} and which implies Proposition~\ref{Prop:InfinitelyBlockSubGrps}.

In Section~\ref{Section:LevelTransitive}, we turn our attention to specific examples of weakly maximal subgroups with block structure, specifically the ones that acts minimally on $\partial T$.

The last section deals in more details with the specific case of the first Grigorchuk group and contains the proof of Proposition~\ref{Proposition:ClosureWMaxGrig}.

\paragraph{Acknowledgment}
The author is grateful to Dominik Francoeur, Rostislav Grigorchuk and  Tatiana Nagnibeda for fruitful discussions and comments on a preliminary version of this article.

The author was partly supported by Swiss NSF grant P2GEP2\_168302 and by RDF-23-01-045-\emph{Groups acting on rooted trees and their subgroups} of XJTLU.
Part of this work was performed within the framework of the LABEX MILYON (ANR-10-LABX-0070) of Universit\'e de Lyon, within the program ``Investissements d'Avenir'' (ANR-11-IDEX-0007) operated by the French National Research Agency (ANR).
%
%
%
%
%
%
%
\section{Definitions and preliminaries}\label{Section:Definitions}
First of all, we insist on the fact that, unless explicitly specified, we will not assume our groups to be finitely generated, or even countable.

Recall that a subgroup $H<G$ is \defi{weakly maximal} if it is maximal among subgroup of infinite index.
In particular, every maximal subgroup of infinite index is weakly maximal, but in general they are far from the only examples.
If $G$ is finitely generated, then, assuming the axiom of choice, every infinite index subgroup is contained in at least one weakly maximal subgroup.

A \defi{rooted tree} $T$ is an unoriented connected graph without cycles and with a distinguished vertex, the \defi{root}.
We will often identify $T$ and its vertex set.
Vertices of a rooted tree can be partitioned into \defi{levels}, where a vertex $v$ is in $\level{n}$ if and only if its distance to the root is $n$.
This naturally endows the vertices of $T$ with a partial order, where $v\leq w$ if the unique path from the root to $w$ passes through $v$.
In this case, we say that $w$ is a \defi{descendant} of $v$.
The terms \defi{parent}, \defi{ancestor}, \defi{child} and \defi{sibling} have their obvious meaning in relation with this order.
For example, every vertex has a unique parent, except for the root which has none.
A rooted tree is \defi{spherically regular} if the degree (equivalently the number of children) of a vertex depends only of its level.
It is \defi{$d$-regular} if every vertex has $d$ children.
If $(m_i)_{i\in \N}$ is a sequence of cardinals greater than $1$, we will denote by $T_{(m_i)}$ an infinite spherically regular rooted tree such that a vertex of level $i$ has exactly $m_i$ children (that is, has degree $m_i+1$, except for the root that has degree $m_0$).
Unless stated otherwise, in the following we will always assume our trees to be rooted and spherically regular.
In general, we will not assume $T$ to be locally finite (i.e. all the $m_i$'s are finite).
Nevertheless, if $G\leq\Aut(T)$ is branch, or more generally, an almost level-transitive rigid group, then $T$ is automatically locally finite, see Definition~\ref{Definition:Branch} and the discussion after it.
We will identify the locally finite rooted tree $T_{(m_i)}$ with its realization as words $x_0x_1x_2\dots x_n$ such that $x_i$ is in $\{0,1,\dots,m_i-1\}$; in particular the root is the empty word and the order on $T_{(m_i)}$ is the lexicographic order.

There is a natural notion of rays (infinite paths emanating from the root, or equivalently right infinite words $x_0x_1x_2\dots$) and of boundary $\partial T$.
The space $\partial T$ is a metric space, where the distance between two distinct rays $x_0x_1x_2\dots$ and $y_0y_1y_2\dots$ is equal to $2^{-i_0}$ where $i_0$ is the lowest index such that $x_{i_0}\neq y_{i_0}$. When $T$ is locally finite, $\partial T$ is homeomorphic to a Cantor space.

For a vertex $v$ of $T$, we will denote by $T_v$ the subtree of $T$ consisting of all descendants of $v$.
It is naturally rooted at $v$.

The group $\Aut(T)$ consists of all graph automorphisms of $T$ sending the root onto itself. That is, an element of $\Aut(T)$ is a bijection of the vertex set that preserves the adjacency relation and fixes the root.
The group $\Aut(T)$ is exactly the group of isometries of $\partial T$.
On the other hand, elements of $\Aut(T)$ naturally preserve the levels.
This can be used to put a metric on $\Aut(T)$, where the distance between two distinct automorphisms $\varphi$ and $\psi$ is equal to $2^{-i_0}$ where $i_0$ is the lowest index such that there is a vertex $v$ of level $i_0$ with $\varphi(v)\neq\psi(v)$.
If~$T$ is locally finite, this induces a compact Hausdorff topology on $\Aut(T)$, turning $\Aut(T)$ into a profinite group.

Let $\{v_0,\dots, v_{n-1}\}$ be the vertices of the first level of $T$.
Since $T$ is spherically regular, all the $T_{v_i}$ are isomorphic and we have a natural isomorphism
\[
	\varphi\colon \Aut(T)\to\Sym(\level{1})\ltimes\Aut(T_{v_0})^{\level{1}}\cong \Aut(T_{v_0})\wr\Sym(\level{1})
\]
where $\Sym(\level{1})$ is the groups of all bijections of vertices of the first level.
That is, for any $g\in\Aut(T)$, there exists a unique permutation $\sigma$ of $\level{1}$ and unique $\treesection{g}{v_i}$'s in $\Aut(T_i)$ such that $\varphi(g)=\sigma(\treesection{g}{v_0},\dots, \treesection{g}{v_{n-1}})$.
The element $\treesection{g}{v}$ is called the \defi{section} of $g$ at $v$ and can be defined inductively for any $v$ in $T$.
In practice, we will often write $g=\sigma(\treesection{g}{v_0},\dots, \treesection{g}{v_{n-1}})$ as a shorthand for $\varphi(g)$.
For any vertex $v$, we have a surjective homomorphism
\begin{align*}
\pi_v\colon\Stab_{\Aut(T)}(v)&\to\Aut(T_v)\\
g&\mapsto\treesection{g}{v}
\end{align*}
The \defi{portrait} $\portrait(g)$ of an automorphism $g$ of $T$ is a labelling of vertices of $T=T_{(m_i)}$ by elements of $\Sym(m_i)$ constructed inductively as following.
Label the root by $\sigma$, where $g=\sigma(\treesection{g}{v_0},\dots, \treesection{g}{v_{m_0-1}})$ and decorated the subtrees $T_{v_i}$ by the portrait of $\treesection{g}{v_i}$.
There is a natural bijection between elements of $\Aut(T)$ and portraits.
An element $g$ is \defi{finitary} if its portrait has only finitely many non-trivial labels.
It is \defi{finitary along rays} if any ray in the portrait of $g$ has only finitely many non-trivial labels.
A subgroup $G$ of $\Aut(T)$ is said to be \defi{finitary}, respectively \defi{finitary along rays}, if all its elements have the desired property.
We have
\[
	\Autf(T)<\Autfr(T)<\Aut(T)
\]
where $\Autf(T)$, respectively $\Autfr(T)$, is the subgroup of all finitary, respectively finitary along rays, automorphisms of $T$.
If $T$ is locally finite, the subgroup $\Autf(T)$ is countable, while the other two have the cardinality of the continuum.
\begin{definition}
Let $T$ be a regular rooted tree.
A subgroup $G$ of $\Aut(T)$ is \defi{self-similar} if for every vertex $v\in T$ we have $\pi_v(\Stab_G(v))\leq G$, or equivalently if all the $\treesection{g}{v}$ belong to $G$.

A self-similar group $G$ is \defi{self-replicating} (or \defi{fractal}) if for every vertex $v\in T$ we have $\pi_v(\Stab_G(v))=G$.
The group $G$ is said to be \defi{strongly self-replicating} if for every vertex $v$ of the first level we have $\pi_v(\Stab_G(\level{1}))=G$, while $G$ is \defi{super strongly self-replicating} if for every $n$ and every vertex $v$ of level $n$ we have $\pi_v(\Stab_G(\level{n}))=G$.\footnote{\GGS{} groups with constant vectors are strongly self-replicating but not super strongly self-replicating~\cite{Uria}.}
\end{definition}
When $T$ is a regular rooted tree, the subgroups $\Autf(T)$, $\Autfr(T)$ and $\Aut(T)$ are all self-replicating.

Given $G$ a subgroup of $\Aut(T)$ it is enlightening to look at stabilizers of various subsets of $T$ or of $\partial T$.
For example it is natural to look at stabilizers of rays $\Stab_G(\xi)$ for $\xi\in\partial T$, also called \defi{parabolic subgroups}, and we will see later that setwise stabilizers of closed subsets of $\partial T$ will also play an important role.
We will also be interested in stabilizers of vertices $\Stab_G(v)$ and in pointwise stabilizers of levels $\Stab_G(\level{n})$,
the latter ones being normal subgroups of finite index.
More generally, if $X$ is any subset of $T\cup\partial T$, $\Stab_G(X)$ denotes the pointwise stabilizer of $X$, that is $\Stab_G(X)=\bigcap_{v\in X}\Stab_G(v)$, while $\SStab_G(X)$ denotes its setwise stabilizer.

We will also look at \defi{rigid stabilizers} of vertices, where $\Rist_G(v)$ is the pointwise stabilizer of $T\setminus T_v$, that is the subgroup of elements acting trivially outside $T_v$.
Finally, \defi{rigid stabilizer of levels}, $\Rist_G(\level{n})$, are defined as the product of all $\Rist_G(v)$ for $v$ of level $n$.
\begin{definition}\label{Definition:Branch}
A subgroup $G$ of $\Aut(T)$ is \defi{level-transitive} if it acts transitively on $\level{n}$ for every $n$, or equivalently if $G\curvearrowright\partial T$ is minimal.
It is \defi{almost level-transitive} if the number of orbits for the actions $G\curvearrowright\level{n}$ is bounded, or equivalently if $G\curvearrowright \partial T$ has a finite number of orbit-closures.

The group $G$ is said to be \defi{micro-supported} if all the $\Rist_G(v)$ are non-trivial (equivalently infinite)\footnote{A group $G\leq\Aut(T)$ is micro-supported if and only if the action $G\curvearrowright\partial T$ is micro-supported. That is, if for every open set $U\subset \partial T$ the subgroup  $\Rist_G(U)$ of elements acting trivially on $\partial T\setminus U$ is non-trivial.} and \defi{rigid} if all the $\Rist_G(\level{n})$ have finite index in~$G$.

The group $G$ is \defi{weakly branch}, respectively \defi{branch}, if it is both level-transitive and micro-supported, respectively rigid.
Finally, we will say that the group $G$ is \defi{almost branch} if it is both almost level-transitive and rigid.
\end{definition}
The reader should have in mind that all the above definitions are useful weakening of the branch properties. A branch subgroup of $\Aut(T)$ is thus altogether almost branch, weakly branch, micro-supported, rigid with an (almost) level-transitive action on the tree. In particular it satisfies all statements that are true for groups possessing one of these properties.

It follows from the definitions that finite subgroups of $\Aut(T)$ are always rigid, while infinite rigid subgroups of $\Aut(T)$ are automatically micro-supported.

Observe that if $G\leq\Aut(T)$ is almost branch, then $T$ is locally finite.
In particular, $T$ is locally finite if and only if the three subgroups $\Autf(T)$, $\Autfr(T)$ and $\Aut(T)$ are all branch.

While in the introduction we stated our results for (weakly) branch groups, most of these statements admit generalizations to groups that are not necessarily level-transitive but only almost level-transitive.
The detailed versions are found in the next sections.
\begin{remark}
We will sometimes say that an abstract group $G$ is (weakly) branch, meaning that it admits a faithful (weakly) branch action.
The existence of such an action can be characterized algebraically from the subgroup lattice of $G$, see~\cite{MR2011117,MR2605182}.
Moreover, some striking examples of branch groups (as the first Grigorchuk group~\cite{MR2011117} or branch generalized multi-edge spinal groups~\cite{MR3834728}) admit a unique branch action.
\end{remark}
Before going further, we recall the following fact that will be of great help for the study of infinite index subgroups of branch groups.
\begin{fact}\label{Fact:NumberOrbitsIndex}
Let $G$ be a group acting transitively on a set $X$.
For every subgroup $H\leq G$, the number of $H$-orbits for the action $H\curvearrowright X$ is bounded above by $[G:H]$.
\end{fact}
While this is stated for transitive action, we will often use it for almost transitive action, that is for actions with a finite number of orbits.
In this case, we have that the number of $H$-orbits is bounded above by the number of $G$-orbits times $[G:H]$.
As a nice consequence of it, we have the following result.
\begin{lemma}\label{Lemma:RistLevelTransitive}
Let $G\leq\Aut(T)$ be almost branch.
Then for any integer $n$, there exists $m\geq n$ such that for every $v$ of level $n$, $\Rist_G(v)\cap\Stab_G(w)$ acts level-transitively on $T_w$ for all $w$ descendant of $v$ of level $m$.
\end{lemma}
\begin{proof}
Let $d$ denotes the maximum number of orbits for the action of $G$ on a level.
Then for every $m\geq n$ the number of orbits for the action $\Rist_G(\level{n})\curvearrowright\level{m}$ is bounded by $[G:\Rist_G(\level{n})]\cdot d$.
Now, let $v$ be any vertex of level $n$.
If $\Rist_G(v)$ does not act level-transitively on $T_v$, there exists some level $n_1> n$ such that $\Rist_G(v)\curvearrowright T_v\cap\level{n_1}$ has at least two orbits.
If $\Rist_G(v)$ does not act level-transitively on $T_w$ for all descendant $w$ of $v$ of level $n_1$, there exists a level such that $\Rist_G(v)\curvearrowright T_v\cap\level{n_2}$ has at least four orbits, and so on.
By the above, this procedure as to stop and there exists $m_v$ such that $\Rist_G(\level{n})\cap\Stab_G(w)$, and hence $\Rist_G(v)\cap\Stab_G(w)$, acts level-transitively on $T_w$ for all $w$ descendant of~$v$ of level $m_v$.
We finish the proof by taking $m$ to be the maximum of all the~$m_v$ for $v$ of level $n$.
\end{proof}
While the proof of the following fact is an easy exercise, it will be of great help to find subgroups that are branch.
\begin{fact}\label{Fact:SubgroupBranch}
Let $G\leq\Aut(T)$ be a branch group and $H\leq G$ be a finite index subgroup.
Then the action of $H$ on $T$ is branch if and only if it is level-transitive.
\end{fact}
Let $T$ be a locally finite rooted tree.
Every subgroup $G\leq\Aut(T)$ comes with the following $3$ natural topologies.
The \defi{profinite topology}, where a basis of neighbourhood at $1$ is given by subgroups of finite index, 
the \defi{congruence topology}, where a basis of neighbourhood at $1$ is given by $\bigl(\Stab_G(n)\bigr)_{n\in\N}$ and finally the \defi{branch topology} where a basis of neighbourhood at $1$ is given by the $\bigl(\Rist_G(\level{n})\bigr)_{n\in\N}$.
The corresponding completions are denoted by $\hat G$ (\defi{profinite completion}), $\overline{G}$ (\defi{congruence completion}, which coincides with the closure of $G$ in the profinite group $\Aut(T)$) and $\tilde G$ (\defi{branch completion}).
Since $\Stab_G(n)$ is always of finite index and contains $\Rist_G(\level{n})$, we have two natural epimorphisms $\hat G\twoheadrightarrow \overline{G}$ and $\tilde G\twoheadrightarrow \overline{G}$.
If moreover $G$ is rigid, then we have $\hat G\twoheadrightarrow \tilde G\twoheadrightarrow \overline{G}$.
Therefore, for branch groups we have three kernels: the \defi{congruence kernel} $\ker(\hat G\twoheadrightarrow \overline{G})$, the \defi{branch kernel} $\ker(\hat G\twoheadrightarrow \tilde G)$ and the \defi{rigid kernel} $\ker(\tilde G\twoheadrightarrow \overline{G})$.
A branch group $G$ is said to have the \defi{congruence subgroup property} if the congruence kernel is trivial; this implies that the branch kernel is also trivial.
See~\cite{MR2891709} for more details on the congruence subgroup property.
An important result in this subject, due to Garrido~\cite{MR3556961}, is the fact that for a branch group $G$ the congruence and branch topology are intrinsic properties of $G$ and do not depend on the chosen branch action of $G$ on a tree.
In particular, the congruence, branch and rigid kernels are intrinsic properties of~$G$.

The congruence subgroup property for branch group has attracted a lot of attention these last years and was used, among other things, to describe the structure of maximal subgroups in many branch groups.
Earlier work from Pervova showed that the first Grigorchuk group,~\cite{MR1841763}, as well as torsion Gupta-Sidki groups,~\cite{MR2197824}, do not have maximal subgroups of infinite index.
This result was later generalized to all torsion multi-edge spinal groups~\cite{MR3834728}.
For super strongly self-replicating, just infinite (see Definition~\ref{Def:JustInf}) groups with the congruence subgroup property, the property to have all maximal subgroups of finite index is preserved when passing to commensurable groups~\cite[Lemma 4]{MR2009443}, in particular this property passes to subgroups of finite index.
One consequence of this result is that if $G$ is a finitely generated, just infinite, super strongly self-replicating branch group with the congruence subgroup property and without maximal subgroups of infinite index, then every weakly maximal subgroup of $G$ is closed in the profinite topology.

Another interesting property that a branch group may possess is the fact to be just infinite.
\begin{definition}\label{Def:JustInf}
A group $G$ is \defi{just infinite} if it is infinite and all of its proper quotient are finite, equivalently if every non-trivial normal subgroup is of finite index.
\end{definition}
Building upon the work of Wilson, Grigorchuk showed in~\cite{MR1765119} that the class of just infinite groups consists of groups of three distinct types that are, roughly speaking, finite powers of a simple group, finite power of an hereditarily just infinite groups (just infinite groups whose finite index subgroups are also just infinite), and the just infinite branch groups.

The following criterion will help to find just infinite subgroups in branch groups.
\begin{lemma}\label{Lemma:JustInfinite}
Let $G\leq\Aut(T)$ be a branch group and $H$ a finite index subgroup.
\begin{itemize}
\item If $H$ is just infinite, it acts level-transitively,
\item If $G$ is just infinite and $H$ acts level-transitively, then $H$ is just infinite.
\end{itemize}
\end{lemma}
\begin{proof}
Observe that for any vertex $v$ of level $n$, the subgroup $L\coloneqq\prod_{w\in H.v}\Rist_H(w)$ is always normal in $H$.
Since $\Rist_H(v)=\Rist_G(v)\cap H$ has finite index in $\Rist_G(v)$, it is not trivial.
If $H.v$ is not equal to $\level{n}$, then $L$ acts trivially on some $T_u$ for $u$ in $\level{n}\setminus H.v$ and is thus of infinite index in $G$ and also in~$H$.
This finishes the proof of the first assertion.

The second assertion is Lemma 8.5 of~\cite{MR2893544}.
\end{proof}
Finally, we record the following fact about centralizer of elements in a branch group.
\begin{lemma}\label{Lemma:InfiniteIndexCentralizer}
Let $G\leq\Aut(T)$ be a branch group.
For any $g\neq 1$ in $\Aut(T)$, the centralizer $C_G(g)$ has infinite index in $G$.
\end{lemma}
\begin{proof}
Since $g$ is non-trivial, it moves a vertex $v$.
Since weakly branch groups do not have abelian rigid stabilizers~\cite[Lemma 4.2]{2020arXiv200608677L}, we can choose $f$ and $h$ two distinct elements of the derived subgroup $\Rist_G(v)'$ and $w$ a descendant of $v$ such that $f(w)\neq h(w)$.
We have $gf(w)\neq gh(w)$ while $hg(w)=fg(w)=g(w)$, hence at most one of $f$ or $h$ belongs to $C_G(g)$.
As a consequence, $C_G(g)$ contains no derived subgroup of rigid stabilizers of levels and is of infinite index by~\cite[Proof of Theorem 4]{MR2893544}.
\end{proof}
%
%
%
%
%
%
%
%
%
%
\subsection{Some examples of branch groups}
The first Grigorchuk group $\Grig$ is probably the most well-known and most studied branch group. This was the first example of a group of intermediate growth,~\cite{MR764305}.
\begin{definition}
The \defi{first Grigorchuk group} $\Grig=\gen{a,b,c,d}$ is the subgroup of $\Aut(T_2)$ generated by $a=(1,1)\varepsilon$, where $\varepsilon$ is the cyclic permutation $(1 2)$ of the first level, and by the three elements $b$, $c$ and $d$ of $\Stab_{\Aut(T)}(\level{1})$ that are recursively defined by
\[
	b=(a,c)\qquad c=(a,d)\qquad d=(1,b).
\]
\end{definition}
This is a $2$-group of rank $3$, all generators have order $2$ and $b$, $c$ and $d$ pairwise commute.
The group $\Grig$ is branch, self-replicating, just infinite, has the congruence subgroup property and possesses a unique branch action in the sense of~\cite{MR2011117}.
Finally, all maximal subgroups of $\Grig$ have finite index,~\cite{MR1841763}.
See~\cite{MR1786869,MR2195454} for references and Section~\ref{Section:GrigorchukGroup} for more details.

Other well-studied examples of branch groups are the Gupta-Sidki groups,~\cite{MR696534}, as well as their generalizations.
\begin{definition}
Let $p$ be a prime, $T$ the $p$-regular tree and let $\mathbf e=(e_0,\dots,e_{p-2})$ be a vector in $\mathbf (F_p)^{p-1}\setminus\{0\}$.
The \GGS{} group $G_{\mathbf e}=\gen{a,b}$ with defining vector $\mathbf e$ is the subgroup of $\Aut(T)$ generated by the two automorphisms 
\begin{align*}
a&=(1,\dots,1)\cdot\varepsilon\\
b&=(a^{e_0},\dots,a^{e_{p-2}},b)
\end{align*}
where $\varepsilon$ is the cyclic permutation $(1 2 \dots p)$.
\end{definition}
The name \GGS{} stands for Grigorchuk-Gupta-Sidki as these groups generalize both the Gupta-Sidki groups (where $\mathbf e=(1,-1,0,\dots, 0)$) and the second Grigorchuk group (where $p=4$ is not prime and $\mathbf e=(1,0,1)$).

A wide generalization of \GGS{} groups is provided by the \defi{generalised multi-edge spinal groups}, where, for $p$ an odd prime and $T$ the $p$-regular rooted tree,
\[
	G=\presentation{\{a\}\cup\{b_i^{(j)}\}}{1\leq j\leq p,1\leq i\leq r_j}\leq\Aut(T_p)
\]
is a subgroup of $\Aut(T)$ that is generated by one automorphism $a=(1,\dots,1)\varepsilon$ and $p$ families $b_1^{(j)},\dots,b_{r_j}^{(j)}$ of directed automorphisms\footnote{An automorphism $b$ of $T$ is \defi{directed} if it fixes a ray $\xi=(v_i)_{i\geq 1}$ and the portrait of $b$ is trivial except maybe for children of the $v_i$'s.}, each family sharing a common
directed path disjoint from the paths of the other families.
See~\cite{MR3834728} for a precise definition.

It is shown in~\cite{MR3834728} that generalized multi-edge spinal groups enjoy a lot of interesting properties.
First of all, they always are residually-(finite $p$) group that are self-replicating.
Moreover, Klopsch and Thillaisundaram showed
\begin{theorem}[\cite{MR3834728}]
Let $G$ be a generalized multi-edge spinal group.
\begin{enumerate}
\item If $G$ is branch, then it admits a unique branch action in the sense of~\cite{MR2011117},
\item If $G$ is torsion, then it is just infinite and branch,
\item If $G$ is torsion, then $G$ does not have a maximal subgroup of infinite index. The same holds for groups commensurable with $G$.
\end{enumerate}
\end{theorem}
In particular, if $G$ is a torsion generalized multi-edge spinal group, then its maximal subgroups are normal and of index $p$.
Since all maximal subgroups are of finite index, then by~\cite{2015arXiv150908090M} the derived subgroup $G'$ of $G$ is contained in the Frattini subgroup (that is, the intersection of all the maximal subgroups)  of $G$.

Some of these groups have the congruence subgroup property, for example torsion \GGS{} groups (Pervova~\cite{MR2308183}),  but this is not always the case,~\cite{MR3834728}.

For the special case of \GGS{} groups, there is an easy criterion, due to Vovkivsky~\cite{MR1754681}, to decide whenever $G$ is torsion.
Indeed, a \GGS{} group with defining vector $\mathbf e$ is torsion if and only if $\sum_{i=0}^{p-2}e_i=0$.

For \GGS{} groups we will also make use of the following facts, see~\cite{MR3152720} for example for a proof.
\begin{proposition}\label{Prop:GGSBasics}
Let $G$ be a \GGS{} group a $G'$ be its derived subgroup. Then,
\begin{enumerate}
\item $\Stab_G(\level{1})=\gen{b}^G=\gen{b,aba^{-1}\dots,a^{p-1}ba^{-(p-1)}}$,
\item $G=\gen{a}\ltimes\Stab_G(\level{1})$,
\item $G/G'=\gen{aG',bG'}\cong C_p\times C_p$,
\item $\Stab_G(\level{2})\leq G'\leq\Stab_G(\level{1})$.
\end{enumerate}
\end{proposition}
Finally, we will use the existence of a normal form for elements in $\Stab_G(\level{1})$.
See~\cite{MR3513107} for the case of Gupta-Sidki groups.
This normal form is best described by the use of the circulant matrix 
\[
	\Circ(\mathbf e,0)=
	\begin{pmatrix}
		e_0		& e_1	& \dots 	& e_{p-2}	& 0\\
		0		& e_0	& e_1		& \dots		& e_{p-2}\\
		e_{p-2}	& 0		& e_1		& \dots		& e_{p-3}\\
		\vdots	&\vdots & \ddots	& \ddots	&\vdots\\
		e_1		& \dots & e_{p-2}	& 0			& e_0
	\end{pmatrix}
\]
where $\mathbf e=(e_0,\dots,e_{p-1})$ is the defining vector of $G$.

The following is a generalization of a similar result for Gupta-Sidki groups that was originally obtained by Garrido in~\cite{MR3513107}.
\begin{lemma}\label{Lemma:GGSNormalForm}
Let $G$ be a GGS group and $g$ be an element of $\Stab_G(\level{1})$.
Then there exist unique $(n_i)_{i=0}^{p-1}$ in $\mathbf F_p$ and $(d_i)_{i=0}^{p-1}$ in $G'$ such that

\[
	g=(a^{\alpha_0}b^{\beta_0}d_0,\dots, a^{\alpha_{p-1}}b^{\beta_{p-1}}d_{p-1})\tag{\dag}\label{Eq:NormalForm}
\]
where
\[
	(\alpha_0,\dots,\alpha_{p-1})=(n_0,\dots,n_{p-1})\Circ(\mathbf e,0)\quad\textnormal{and}\quad (\beta_0,\dots,\beta_{p-1})=(n_0,\dots,n_{p-1})P
\]
with $P$ the matrix representation of the permutation $(12\dots p)$.
\end{lemma}
\begin{proof}
For every integer $i$, let $b_i\coloneqq a^{i}ba^{-i}$ and let $B\coloneqq\Stab_G(\level{1})$.
Hence $b_i=b_j$ if $i\equiv j \pmod p$ and $b_i^p=1$ for every $i$ and
$B=\gen{b_0,\dots,b_{p-1}}$ by Proposition~\ref{Prop:GGSBasics}.
Then the abelian group $B/B'$ admits the following description
\[
	B/B'=\gen{b_0B',\dots,b_{p-1}B'}=\setst{b_0^{r_0}b_1^{r_1}\cdots b_{p-1}^{r_{p-1}}B'}{r_i\in \mathbf F_p}.
\]
Now, for $g\in B$ there exists some $n_i$ in $\mathbf F_p$ and $c$ in $B'$ with
\begin{align*}
	g&=b_0^{n_0}b_1^{n_1}\cdots b_{p-1}^{n_{p-1}}c\\
	&=(a^{n_0e_0}b^{n_1}a^{n_2e_{p-2}}\cdots a^{n_{p-1}e_1},\dots,b^{n_0}a^{n_1e_{p-2}}\cdots a^{n_{p-1}e_{0}})c\\
	&=(a^{\alpha_0}b^{n_1},\dots, a^{\alpha_{p-1}}b^{n_0})(d_0,\dots,d_{p-1})
\end{align*}
where $(\alpha_0,\dots,\alpha_{p-1})=(n_0,\dots,n_{p-1})\Circ(\mathbf e,0)$ and the $d_i$ belong to $G'$.

The $n_i$ as well as $c$ are uniquely determined, and therefore so are the $\alpha_i$, $\beta_i$ and $d_i$.
\end{proof}
%
%
%
%
%
%
%
%
%
%
\section{Generalized parabolic subgroups}\label{Section:GeneralizedParabolicSubgroup}
In this section, we study setwise stabilizers of closed subsets in the setting of weakly maximal subgroups.
The main motivation for this is Theorem~\ref{Thm:GeneralizedParabolic}.
Recall that since the action of $\Aut(T)$ on $\partial T$ is continuous, if $C$ is a closed subset of $\partial T$, then $\SStab_G(C)$ is closed for the congruence topology.
In particular, $\SStab_G(C)$ is closed for the profinite topology and also closed for the branch topology.

The following definition is motivated by the forthcoming Proposition~\ref{Proposition:WMStrongRigidSStab}. 
\begin{definition}
A \defi{generalized parabolic subgroup} of a group $G\leq\Aut(T)$ is the setwise stabilizer $\SStab_G(C)$ of some closed non-open subset $C$ of $\partial T$ such that $\SStab_G(C)$ acts minimally on $C$.
\end{definition}
It directly follows from the definition that if $\SStab_G(C)$ is a generalized parabolic subgroup, then $C$ is closed and nowhere dense.

Before going further, recall that the set of (non necessarily spherically regular) subtrees of $T$ containing the root and without leaf is in bijection with $\mathcal C$ the set of non-empty closed subset of $\partial T$.
This bijection is given by $S\mapsto \partial S$ and $C\mapsto T_C\coloneqq\setst{v\in T}{\exists\xi\in C:v\in\xi}$ that are $G$-equivariant maps.

For every non-empty closed subset $C$ of $\partial T$, the action $\SStab_G(C)\curvearrowright C$ is minimal if and only if $\SStab_G(T_C)=\SStab_G(C)\curvearrowright T_C$ is level-transitive.
In particular, if $\SStab_G(C)$ is a generalized parabolic subgroup, then $T_C\subset T$ is a spherically regular rooted tree.

Generalized parabolic subgroups of weakly branch groups retain some of the branch structure.
More precisely, let $W=\SStab_G(C)$ be a generalized parabolic subgroup.
Since $C$ is nowhere dense, for every vertex $v$ in $T$, there exists $w\geq v$ that is not in $T_C$.
But then $W$ contains $\Rist_G(w)$ and hence $\Rist_W(v)\geq\Rist_W(w)=\Rist_G(w)$ is infinite.
We just proved
\begin{lemma}\label{Lemma:GenParabWeaklyBranch}
Let $W=\SStab_G(C)$ be a generalized parabolic subgroup of a micro-supported group $G$.
Then, for every vertex $v$ the rigid stabilizer $\Rist_W(v)$ is infinite, that is $W$ is micro-supported.
\end{lemma}
\begin{example}\label{Example:GenParAutT}
Let $(m_i)_{i\in N}$ be a sequence of integers greater than $1$ and $T=T_{(m_i)}$ the corresponding spherically regular rooted tree.
Let $G$ be one of the group $\Aut(T)$, $\Autfr(T)$ or $\Autf(T)$.
Let $(n_i)_{i\in N}$ be a sequence such that $1\leq n_i\leq m_i$ for all $i$ and such that $n_i<m_i$ for infinitely many $i$.
For any $S\subseteq T$ spherically regular subtree that is isomorphic to $T_{(n_i)}$, the subgroup $\SStab_G(\partial S)$ is generalized parabolic.
Moreover, all generalized parabolic subgroups of $G$ arise in this form.

Indeed, as we have seen, the condition that $S$ is spherically regular is always necessary for $\SStab_G(S)=\SStab_G(\partial S)$ to acts minimally on $\partial S$.
On the other hand, for $G=\Autf(T)$ (and also for $G=\Autfr(T)$ or $\Aut(T)$), it is also a sufficient condition.
Finally, the condition that infinitely many $n_i$ are strictly less than $m_i$ is equivalent to the fact that $\partial S$ is nowhere dense.
\end{example}
A corollary of this example is that if $G\leq\Aut(T)$ with $T$ locally finite, then~$G$ has at most a continuum of generalized parabolic subgroups, even if $G$ itself has the cardinality of the continuum.

We observe that setwise stabilizers are always infinite.
\begin{lemma}\label{Lemma:GenParInfinite}
Let $G\leq \Aut(T)$ be micro-supported.
Then, for every closed subset $C\leq \partial T$, $\SStab_G(C)$ is infinite.
\end{lemma}
\begin{proof}
On one hand, if $C=\partial T$, then $\SStab_G(C)=G$ is infinite.
On the other hand, if $C\neq \partial T$, then $\partial T\setminus C$ is a non-empty open set an therefore has at least one interior point $\xi=(v_i)_{i\geq 1}$.
Then there is $i$ such that $C$ contains no rays passing through $v_i$.
This implies that $\SStab_G(C)$ contains $\Rist_G(v_i)$ which is infinite.
\end{proof}
While being infinite, setwise stabilizers of non-open closed subsets of $\partial T$ are still ``small'' subgroups in the sense that they are of infinite index and have infinitely many orbit-closures.
\begin{lemma}\label{Lemma:InfOrbitsClosure}
Let $G$ be any subgroup of $\Aut(T)$ and $C$ be a non-open closed subset of $\partial T$.
Then the number of orbit-closures for the action $\SStab_G(C)\curvearrowright \partial T$ is infinite.

Moreover, if $G$ is almost level-transitive, then $\SStab_G(C)$ has infinite index in $G$.
\end{lemma}
\begin{proof}
For $\xi=(v_i)_{i\geq 1}$ in $\partial T$ let $d(\xi,C)$ be the distance in $\partial T$ between $\xi$ and $C$.
If $\xi$ and $\eta$ are in the same $\SStab_G(C)$-orbit, then $d(\xi,C)=d(\eta,C)$.
Therefore, in order to prove the infinity of orbit-closures of $\SStab_G(C)\curvearrowright \partial T$, it is enough to find infinitely many $\xi$ with distinct $d(\xi,C)$.

Since $C$ is not open, there exists $\eta$ in $C$ that is not an interior point.
That is, there exists $(\xi_i)_i$ not in $C$ that converge to $\eta$ in $C$.
Then the sequence $d(\xi_i,C)$ converges to $d(\eta,C)=0$, while $d(\xi_i,C)\neq 0$ since $C$ is closed.
Hence $d(\xi_i,C)$ takes infinitely many values.

Finally, the infinity of orbit-closures on $\partial T$ is equivalent to the fact that the number of orbits of $\SStab_G(C)$ on $\level{n}$ is not bounded.
If $G$ acts almost level-transitively on $T$, Fact~\ref{Fact:NumberOrbitsIndex} implies that $\SStab_G(C)$ has infinite index.
\end{proof}
We are now able to prove that generalized parabolic subgroups in branch groups are pairwise distinct.
\begin{lemma}\label{Lemma:GenParDistinct}
Let $G\leq \Aut(T)$ be an almost branch group.
Let $C_1\neq C_2$ be two  closed nowhere dense subsets of $\partial T$.
Then $\SStab_G(C_1)\neq\SStab_G(C_2)$.
\end{lemma}
\begin{proof}
Let $\xi=(v_i)_{i\geq 1}$ be in $C_1$ but not in $C_2$.
Since $C_2$ is closed, there exists $i$ such that the intersection of $C_2$ with $\partial T_{v_i}$ is empty.
This implies that $\SStab_G(C_2)$ contains $\Rist_G(v_i)$.
On the other hand, since $G$ is almost level-transitive and rigid, then $\SStab_G(C_1)$ does not contain $\Rist_G(v_i)$ as otherwise $C$ would contain a small neighbourhood of $\xi$ by Lemma~\ref{Lemma:RistLevelTransitive}.
\end{proof}
Observe that while $G$ weakly branch is sufficient for showing that parabolic subgroups are distinct (Proposition~\ref{Prop:BartholdiGrigorchuk}), we suppose in Lemma~\ref{Lemma:GenParDistinct} that $G$ is rigid.
However, a similar result holds for weakly branch groups if we restrict a little bit which kind of closed subsets we are allowed to look at.
Recall that $\Acc(C)$ is the set of accumulation points of $C$.
In particular, $\bigr(C_1\setminus\Acc(C_1)\bigl)\neq\bigr(C_2\setminus\Acc(C_2)\bigl)$ if $C_1\neq C_2$ are both finite.
\begin{lemma}\label{Lemma:GenParDistinct2}
Let $G\leq \Aut(T)$ be a weakly branch group.
Let $C_1\neq C_2$ be two closed subsets of $\partial T$ such that $C_1\setminus\Acc(C_1)\neq C_2\setminus\Acc(C_2)$.
Then $\SStab_G(C_1)\neq\SStab_G(C_2)$.
\end{lemma}
\begin{proof}
We have $C_1\setminus\Acc(C_1)\neq C_2\setminus\Acc(C_2)$ if and only if the set $\bigr(C_1\setminus\Acc(C_1)\bigl)\Delta\bigr(C_2\setminus\Acc(C_2)\bigl)=(C_1\Delta C_2)\setminus\bigl(\Acc(C_1)\cup\Acc(C_2)\bigr)$ is non-empty.
Let $\xi$ be in $(C_1\Delta C_2)\setminus\bigl(\Acc(C_1)\cup\Acc(C_2)\bigr)$.
For example, $\xi$ in $C_1$ but not in $C_2$.
Since $C_2$ is closed there exists $v$ a vertex of $\xi$ such that $\SStab_G(C_2)$ contains $\Rist_G(v)$.
On the other hand, $\SStab_G(C_1)$ does not contain $\Rist_G(v)$ as otherwise $C_1$ would contain $\Rist_G(v).\xi$.
But, by (the proof of)~\cite[Lemma 2.3]{MR2893544}, $\xi$ is an accumulation point of $\Rist_G(v).\xi$.
\end{proof}
The following examples show that the hypothesis of Lemmas~\ref{Lemma:GenParDistinct} and~\ref{Lemma:GenParDistinct2} are necessary.
Let $T$ be the $2$-regular rooted tree and $G\leq\Aut(T)$ be any subgroup --- for example one can take $G=\Grig$ which is both almost branch and weakly branch.
Let $C_1=\partial T_0$ be the set of all rays starting with $0$ and $C_2=\partial T_1$. Then $C_1\neq C_2$ is closed but non-open and $C_1\setminus\Acc(C_1)= C_2\setminus\Acc(C_2)=\emptyset$.
On the other hand, an element of $\Aut(T)$ fixes the vertex $0$ if and only if it fixes the vertex $1$. Hence we have $\SStab_G(C_1)=\SStab_G(C_2)$.

Let $T$ be a locally finite rooted tree. Then $\partial T$ and all of its closed subsets are compact. We hence obtain.
\begin{lemma}\label{Lemma:OpenFiniteIndex}
Let $T$ be locally finite, $G\leq \Aut(T)$ be a rigid group and $C$ be a closed subset of  $\partial T$.
If $C$ is open, then $\SStab_G(C)$ has finite index in $G$.
\end{lemma}
\begin{proof}
Since $C$ is both compact and open, it is a finite union of cylinders: $C=\bigcup_{i=1}^d\partial T_{v_i}$.

Let $m$ be the maximal level of the $v_i$'s and $v$ any vertex of level $m$.
Either $v$ is not in $T_C$ and hence $\Rist_G(v)$ is contained in $\SStab_G(C)$, or $v$ is in $T_C$ and hence below some $v_i$.
But in this case $\Rist_G(v_i)$ and hence $\Rist_G(v)$ is also contained in $\SStab_G(C)$.
Therefore, $\SStab_G(C)$ contains the rigid stabilizer of the $m$\textsuperscript{th} level and is of finite index.
\end{proof}
Using Lemmas~\ref{Lemma:InfOrbitsClosure} and~\ref{Lemma:OpenFiniteIndex} we have.
\begin{proposition}\label{Proposition:WMStrongRigidSStab}
Let $G\leq \Aut(T)$ be almost branch and $C$ be a closed subset of $\partial T$.
Then $\SStab_G(C)$ is generalized parabolic if and only if it is weakly maximal.
\end{proposition}
\begin{proof}
If $\SStab_G(C)$ is weakly maximal, it is of infinite index and hence $C$ is not open.
Let $\xi$ be any boundary point of $C$ and let $C'$ be the closure of the $\SStab_G(C)$-orbit of $\xi$.
Then $C'\subseteq C$ is closed but not open, and $\SStab_G(C)\leq\SStab_G(C')$ which is of infinite index.
Therefore they are equal and since $\SStab_G(C')$ acts minimally on $C'$ it is a generalized parabolic subgroup.

On the other hand, let $C$ be a non-open closed subset of $\partial T$ such that the action $\SStab_G(C)\curvearrowright C$ is minimal.
Since $C$ is not open, $\SStab_G(C)$ is an infinite index subgroup of $G$.
Moreover, it contains $\Rist_G(v)$ for all $v$ outside $T_C$.
Now, let us take $g$ not in $\SStab_G(C)$. 
There exists $v\notin T_C$ such that $g.v$ is in $T_C$.
In particular, $\gen{g,\SStab_G(C)}$ contains $\Rist_G(g.v)$ and by transitivity also $\Rist_G(w)$ for all $w$ in $T_C$ of the same level as~$v$.
That is, $\gen{g,\SStab_G(C)}$ contains a rigid stabilizer of a level and is of finite index.
\end{proof}
Together with Lemmas~\ref{Lemma:GenParInfinite} and~\ref{Lemma:GenParDistinct} we obtain the following corollary, which is a strong version of Theorem~\ref{Thm:GeneralizedParabolic}.
\begin{corollary}
Let $G$ be a subgroup of $\Aut(T)$.

If $G$ is micro-supported, then all generalized parabolic subgroups are infinite.

If $G$ is almost branch, then all generalized parabolic subgroups are weakly maximal and pairwise distinct.
\end{corollary}
We now have all the ingredients to prove generalizations of Corollaries~\ref{Thm:NumberGeneralizedParabolic} and~\ref{Cor:NumberGeneralizedParabolic}, as well as the first part of Corollary~\ref{Cor:NumberGeneralizedParabolic2}.
\begin{corollary}\label{Prop:PropCF}
Let $G\leq\Aut(T)$ be an almost branch group and $F\leq G$ be any subgroup.
Let $\mathcal C_F$ be the set of all non-open orbit-closures of the action $F\curvearrowright\partial T$.
Then the function $\SStab_G(C)\colon\mathcal C_F\to\Sub(G)$ is injective and has values in generalized parabolic subgroups, which are all weakly maximal.
\end{corollary}
\begin{proof}
On one hand, for any $C$ in $\mathcal C_F$, $F$ is a subgroup of $\SStab_G(C)$ and acts minimally on $C$.
Therefore, $\SStab_G(C)$ is a generalized parabolic subgroup that is weakly maximal.
On the other hand, by Lemma~\ref{Lemma:GenParDistinct}, if $C_1\neq C_2$, the subgroups $\SStab_G(C_1)$ and $\SStab_G(C_2)$ are distinct.
\end{proof}
\begin{corollary}\label{Cor:ExtendTreeEquivalenceClasses}
Let $G\leq\Aut(T)$ be almost branch and $F\leq G$ be any subgroup.
Let $v$ be a vertex of $T$.
\begin{enumerate}
\item If all orbit-closures of $\Stab_F(v)\curvearrowright \partial T_v$ are non-open, then $F$ is contained in uncountably many generalized parabolic subgroups that are all weakly maximal,\label{Cor:ExtendTreeEquivalenceClasses1}
\item If all orbit-closures of $\Stab_F(v)\curvearrowright \partial T_v$ are at most countable, then $F$ is contained in a continuum of generalized parabolic subgroup.
If moreover $\Stab_F(v)$ is non-trivial, then $F$ is contained in a continuum of generalized parabolic subgroups that are not parabolic.
\end{enumerate}
\end{corollary}
\begin{proof}
For $F$ a subgroup of $G$ and $v$ a vertex of $T$, let $\mathcal C_{F,v}$ be the set of non-open orbit-closures of $\Stab_F(v)\curvearrowright \partial T_v$.
In particular, $\mathcal C_{F,\emptyset}=\mathcal C_{F}$ with the notation of Corollary~\ref{Prop:PropCF}.
We claim that for any $F$ and $v$, the map $\varphi\colon\mathcal C_{F,v}\to\Sub(G)$ sending a $\Stab_F(v)$ orbit-closure $C$ onto $\SStab_G(F.C)$ is injective and has values in generalized parabolic subgroups containing $F$.
Indeed, $\varphi$ is the composition of $\psi\colon \mathcal C_{F,v}\to \mathcal C_{F}$ sending $C$ onto $F.C$ and the map $\SStab_G(\cdot)\colon \mathcal C_{F}\to\Sub(G)$.
By Corollary~\ref{Prop:PropCF}, $\SStab_G(\cdot)$ is injective and has values in generalized parabolic subgroups.
It remains to show that $\psi$ is well-defined (i.e., that $F.C$ is closed, non-open and $F$ acts minimally on it) and injective.
But this follows directly from the topology of $\partial T$.

We now prove the first assertion.
Suppose that all the $\Stab_F(v)$ orbit-closures are not open, and hence have empty-interior.
In particular, for any ray $\xi$ passing through $v$, the subgroup $\Stab_G(\overline{F.\xi})$ is generalized parabolic.
Since $\partial T_v$ is a Cantor space, it is also a Baire space.
That is, the union of countably many closed sets with empty interior has empty interior.
In particular, $\partial T_v$ cannot be covered by a countable union of $\Stab_F(v)$-orbit-closures.

In order to prove te second assertion, observe that if all $\Stab_F(v)$ orbit-closures are countable, they are never open and there must be a continuum of them to cover $\partial T_v$.
Suppose moreover that $\Stab_F(v)$ is not trivial.
Then there is a descendant $w$ of $v$ that is moved by $\Stab_F(v)$.
For any ray $\xi$ going through~$w$, $\overline{F.\xi}$ is at most countable and contains at least $2$ elements.
Therefore, all the $\SStab_G(\overline{F.\xi})$ for $\xi$ passing through $w$ are generalized parabolic subgroups that are not parabolic and there is a continuum of them by the last point.
\end{proof}
By taking $F=\gen{g}$ with $g$ a non-trivial element of finite order, we have
\begin{corollary}\label{Cor:ContinuumGenPar}
Let $G$ be an almost branch group. If $G$ is not torsion-free, it contains a continuum of generalized parabolic subgroups that are not parabolic.
\end{corollary}
Observe that unlike the result of~\cite{MR3478865}, Corollary~\ref{Cor:ContinuumGenPar} does not require the existence of a finite subgroup whith a specific action on the tree, but requires only the existence of a non-trivial element of finite order --- a fact that is independent of the chosen branch action.
Moreover, by Corollary~\ref{Prop:PropCF}, the conclusion of Corollary~\ref{Cor:ContinuumGenPar} holds as soon as $G$ has a subgroup $F$ such that there is a continuum of orbit-closures for the action $F\curvearrowright\partial T$ that are neither open nor reduced to a point.
In fact, if the cardinality of $G$ is strictly less than~$2^{\aleph_0}$, this is also a necessary condition.
In particular, if $G$ is a countable branch group that does not have a continuum of generalized parabolic subgroups that are not parabolic, then for every $1\neq g$ in $G$, almost all (that is except for a countable number) orbit-closures of $\mathbf Z\cong \gen{g}\curvearrowright\partial T$ are either open or reduced to a point.
A fact that seems unlikely.
This led us to ask the following
\begin{question}
Let $G\leq \Aut(T)$ be a countable branch group.
Does $G$ have a continuum of generalized parabolic subgroups that are not parabolic?
\end{question}
Finally, we provide some examples of generalized parabolic subgroups.
By definition, parabolic subgroups, that is stabilizers of rays, are generalized parabolic.
If $\xi$ is a ray of $T$ and $g\in G$ is an element of finite order, then $C=\overline{\gen{g}.\xi}$ is finite and hence $\SStab_G(C)$ is generalized parabolic.
More generally, if $g$ is such that $\Stab_{\gen{g}}(v)$ does not act level-transitively on $T_v$ for every $v$, then $C=\overline{\gen{g}.\xi}$ is nowhere dense (and closed) and thus $\SStab_G(C)$ is generalized parabolic.
In general, more complicated situations may happen.
Let $T$ be the $d$-regular rooted tree. Recall that a subgroup $G\leq\Aut(T)$ is a \defi{regular branch group over $K$} if $G$ is level-transitive, self-similar and $K$ is a finite index subgroup of $G$ such that $K^d$ is contained  in $\Stab_K(\level{1})$ as a finite index subgroup.
\begin{example}
Let $G$ be a regular branch group over $K$.
If $G$ is torsion, then there exists a continuum of generalized parabolic subgroups of $G$ of the form $\SStab_G(C)$ where $C$ is a nowhere dense Cantor subset of $\partial T$.

To construct such a subgroup start with some $k\in K\setminus \Stab_G(0^\infty)$.
Such a $k$ always exists since $K$ has finite index while $\Stab_G(0^\infty)$ has infinite index in~$G$.
Since $G$ is regular branch over $K$, for every vertex $v$ of $T$, there exists an element $k_{@v}$ of $\Rist_G(v)$ that acts on $T_v$ as $k$ on $T$.
Let $(m_n)_{n\in\N}$ be a strictly increasing sequence of integers and $C\coloneqq\overline{\presentation{k_{@0^{m_n}}}{n\in\N}.0^\infty}$.
The sequence of rays $k_{@0^{m_n}}.0^\infty$ is in $C\setminus\{0^\infty\}$ and converges to $0^\infty$, hence $C$ has no isolated points and is a Cantor space.
It remains to show that for some sequence $(m_n)$ the subset $C$ is not open.
Since $k$ has finite order, the orbit $k.0^\infty$ is finite.
Therefore, there exists an integer $i$ and a vertex $v_i=x_0\dots x_i$ such that for every $n\in \N$ the vertex $0^{(n-1)i}v_i$ is not in the $\gen{k}$ orbit of $0^{ni}$.
By taking $(n\cdot i)_{n\in\N}$ for the sequence $(m_n)_{n\in\N}$, we have that none of the vertices $0^{n\cdot (i-1)}x_0\dots x_i$ are on a ray of the $\presentation{k_{@0^{m_n}}}{n\in\N}$ orbit of $0^\infty$.
This proves that $C$ is not open.
If instead of taking the sequence $(n\cdot i)$ we take a subsequence of it, we still end with a $C$ that is not open.
Different subsequences give raise to distinct subsets of $\partial T$ and there is a continuum of such sequences.
\end{example}
Finally, we study the rank of generalized parabolic subgroups.
The aim is to answer the following question.
\begin{question}
Which condition on $G$ ensures that generalized parabolic subgroups are not finitely generated.
\end{question}
A first answer comes from Francoeur's thesis~\cite{FrancoeurThese} where he proved that in many weakly branch self-similar groups, parabolic subgroups are never finitely generated.
\begin{theorem}[\cite{FrancoeurThese}]
Let $T$ be a locally finite regular rooted tree and let $G\leq\Aut(T)$ be a weakly branch group.
If there exists $N\in \mathbf N$ such that for any $v\in T$, $\Rist_{\pi_v(G)}(\level{1})$ acts non-trivially on level $N$, then parabolic subgroups are not finitely generated.
\end{theorem}
For groups with the (weak) subgroup induction property, see Definition~\ref{Definition:SubgroupInduction}, we have a stronger result, which we will prove later in Lemma~\ref{Lemma:GenParabNotFG}.
\begin{proposition}\label{Prop:GenParabNotFG}
Let $G\leq\Aut(T)$ be an almost branch group with the weak subgroup induction property.
Then generalized parabolic subgroups of $G$ are not finitely generated.
\end{proposition}
Finally, in order to decide if generalized parabolic subgroups are finitely generated, it may be useful to look at specific subgroups of them.
Recall that for $C\subset \partial T$ closed, $T_C$ denotes the vertices of $T$ above $C$.
In particular, $T_C\cap\level{n}$ is what is called sometimes the \defi{shadow} of $C$ on level $n$.
\begin{definition}
Let $C$ be a closed subset of $\partial T$.
The \defi{neighbourhood stabiliser} of $C$ is
\[
	\SStab_G(C)^0\coloneqq\setst{g\in\SStab_G(C)}{\exists n\forall v\in\level{n}\cap T_C \textnormal{ the portrait of $g$ is trivial on }T_v}
\]
\end{definition}
This is a normal subgroup of $\SStab_G(C)$.

If $C$ is a ray $\xi$, this coincides with the usual definition of the neighbourhood stabiliser, as the set of elements that fixes a small neighbourhood of $\xi$.
If we fixe $n$ in the above definition, we obtain a subgroup $\SStab_G(C)^0_n$ with the property that $\SStab_G(C)^0$ is the increasing union of the $\SStab_G(C)^0_n$.
If the sequence $\SStab_G(C)^0_n$ is not eventually constant, then $\SStab_G(C)^0$ is not finitely generated.
But this is the case if $W$ is a generalized parabolic subgroup.
Indeed, $\Rist_W(v)$ is never trivial, while $\Rist_{\SStab_G(C)^0_n}(v)=\{1\}$ for $v$ in $\level{n}\cap T_C$.
We just proved
\begin{lemma}
Let $G\leq\Aut(T)$ be micro-supported and let $W=\SStab_G(C)$ be a generalized parabolic subgroup.
If $\SStab_G(C)/\SStab_G(C)^0$ is finite, then $W$ is not finitely generated.
\end{lemma}
When $C=\{\xi\}$, the group $\SStab_G(C)/\SStab_G(C)^0$ is known as the \defi{group of germs} of $G$ at $\xi$.
A ray $\xi$ is said to be \defi{regular} if the group of germs is trivial and singular otherwise.
This definition extends mutatis mutandis to closed subset.
We hence have
\begin{corollary}
Let $G\leq\Aut(T)$ be micro-supported and let $W=\SStab_G(C)$ be a generalized parabolic subgroup with $C$ regular.
Then $W$ is not finitely generated.
\end{corollary}
%
%
%
%
%
%
%
%
%
%
%
%
%
%
%
\section{Non-rigidity tree}\label{Section:NonRigidityTree}
An important tool to study a (weakly maximal) subgroup $H$ of a branch group $G$ acting on $T$ is the knowledge of the vertices $v$ of $T$ such that a finite index subgroup of $\Rist_G(v)$ is contained in $H$.
This motivates the following definition.
\begin{definition}
Let $G$ be a subgroup of $\Aut(T)$ and $H$ be a subgroup of $G$.
The \defi{non-rigidity tree of $H$}, written $\NR(H)$, is the subgraph of $T$ generated by all $v$ such that $\Rist_H(v)$ has infinite index in $\Rist_G(v)$.
\end{definition}
It directly follows from the definition that a vertex $v$ is in $\NR(H)$ if and only if $H$ does not contain a finite index subgroup of $\Rist_G(v)$.

The following useful lemma describes the behaviour of subgroups under the addition of some rigid stabilizer.
\begin{lemma}\label{Lemma:RistFiniteIndex}
Let $T$ be a locally finite rooted tree.
Let $G\leq\Aut(T)$ be a rigid group and let $H$ be a subgroup of $G$.
Let $v$ be a vertex of $T$ with children $\{w_1,\dots,w_d\}$.
Suppose that for each $i\in\{1,\dots,d\}$, $H$ contains some finite index subgroup $H_{w_i}$ of $\Rist_G(w_i)$.
Then $H$ has finite index in $\gen{H,\Rist_G(v)}$.
\end{lemma}
\begin{proof}
Let $H'=\gen{H,\Rist_G(v)}$.
We have $H=\gen{H_{w_1}\times\dots\times H_{w_d}}^H\cdot H$ while $H'=\gen{H,\Rist_G(v)}=\gen{\Rist_G(v)}^H\cdot H$.
In particular, the index of $H$ in $H'$ is bounded by the index of $\gen{H_{w_1}\times\dots\times H_{w_d}}^H=\prod_{w\in H.\{w_1,\dots,w_d\}}H_w$ in $\gen{\Rist_G(v)}^H=\prod_{u\in H.v}\Rist_G(u)$, where for each $w$ in $H.\{w_1,\dots,w_d\}$, $H_w$ is a finite index subgroup of $\Rist_G(w)$.
This latter index  is itself bounded by $[\Rist_G(v):H_{w_1}\times\dots\times H_{w_d}]^{\abs{H.v}}$ which is finite.
\end{proof}
\begin{corollary}\label{Cor:RistFiniteIndex}
Let $T$ be a locally finite rooted tree.
Let $G\leq\Aut(T)$ be a rigid group and let $W\leq G$ be a weakly maximal subgroup.
Let $v$ be a vertex of $T$ with children $\{w_1,\dots,w_d\}$.
If $W$ contains a finite index subgroup of $\Rist_G(w_i)$ for each $w_i$, then $W$ contains $\Rist_G(v)$.
\end{corollary}
A small variation of Lemma~\ref{Lemma:RistFiniteIndex} allows us to obtain an alternative description of the non-rigidity tree of weakly maximal subgroups.
\begin{lemma}\label{Lemma:TrivialBranchKernel}
Let $T$ be a locally finite rooted tree.
Let $G\leq\Aut(T)$ be a rigid group and let $W\leq G$ be a weakly maximal subgroup.
Let $v$ be a vertex of $T$.
Then $W$ contains $\Rist_G(v)$ if and only if $v$ is not in $\NR(W)$.
\end{lemma}
\begin{proof}
The only if direction is trivial.
For the other direction, suppose that $\Rist_W(v)$ has finite index in $\Rist_G(v)$ and let $W'\coloneqq \gen{W,\Rist_G(v)}$.
We will prove that $W$ has finite index in $W'$, and hence that they are equal.
We have $W=\gen{\Rist_W(v)}^W\cdot W$, while $W'=\gen{\Rist_G(v)}^W\cdot W$. It follows that the index of $W$ in $W'$ is bounded above by $[Rist_G(v):\Rist:W(v)]^{\abs{W.v}}$ which is finite.
\end{proof}
The following lemma justifies the name non-rigidity trees and shows some elementary properties of them.
\begin{lemma}\label{Lemma:NRTree}
Let $G\leq\Aut(T)$ and let $H$ be a subgroup of $G$.
The non-rigidity tree of $H$ is a tree and it is non-empty if and only if $H$ has infinite index in $G$. Moreover, we have
\begin{enumerate}
\item If $H$ has infinite index in $G$, then $\NR(H)$ is a rooted subtree of $T$, rooted at $\emptyset$ (the root of $T$),
\item If $H\leq K$, then $\NR(K)\subseteq \NR(H)$, with equality if $H$ has finite index in $K$,
\item Suppose that $T$ is locally finite and that $G$ is rigid. If $W$ is a weakly maximal subgroup of $G$, then $\NR(W)$ has no leaf.\label{NoLeaf}
\end{enumerate}
\end{lemma}
\begin{proof}
For the first assertion, let $v$ be a vertex outside $\NR(H)$ and $w$ be one of its descendant.
By hypothesis, $\Rist_H(v)$ has finite index in $\Rist_G(v)$ which implies that $\Rist_H(w)$ has finite index in $\Rist_G(w)$ and hence that $w$ is not in $\NR(H)$.
That is, $\NR(H)$ is always a tree.
It is empty if and only if it does not contain the root, that is if $H$ has finite index in $G$.

The second assertion is a direct consequence of the definition while the third assertion follows from Corollary~\ref{Cor:RistFiniteIndex}.
\end{proof}
For every subgroup $H$ of $G$, we have $H\leq\SStab_G\bigl(\NR(H)\bigr)$.
Moreover, if $\NR(H)$ has no leaves (for example if $H$ is weakly maximal), then $\SStab_G\bigl(\NR(H)\bigr)=\SStab_G\bigl(\partial\NR(H)\bigr)$, with $\partial\NR(H)$ being a closed subset of $\partial T$.
As an application of Proposition~\ref{Proposition:WMStrongRigidSStab} we obtain
\begin{lemma}\label{Lemma:InverseBijections}
Let $G\leq\Aut(T)$ be almost branch and let $W\leq G$ be a weakly maximal subgroup.
Then $W=\SStab_G\bigl(\NR(W)\bigr)$ if and only if $W$ is generalized parabolic.

More precisely, the maps $W\mapsto \partial\NR(W)$ and $C\mapsto \SStab_G(C)$ are inverse bijections from the set of generalized parabolic subgroups of $G$ and the set of closed non-open subset of $\partial T$ such that the action $\SStab_G(C)\curvearrowright C$ is minimal.
\end{lemma}
\begin{proof}
Let $C$ be a closed non-open subset of $\partial T$ such that the action $\SStab_G(C)\curvearrowright C$ is minimal, and let $W=\SStab_G(C)$.
By Proposition~\ref{Proposition:WMStrongRigidSStab}, $W$ is weakly maximal.
On one hand, let $\xi$ be in $C$ and $v$ be a vertex of $\xi$.
Since $G$ is almost branch, the orbit $\Rist_G(v).\xi$ contains a neighbourhood of $\xi$.
Using that $C$ is not open, this implies that $W$ does not contain $\Rist_G(v)$, and by Lemma~\ref{Lemma:TrivialBranchKernel} that $C\subseteq\partial\NR(W)$.
On the other hand, let $\eta$ be in the open set $\partial T\setminus C$.
There exists some $v$ on $\eta$ such that $\partial T_v$ is contained in $\partial T\setminus C$.
We conclude that $\Rist_G(v)$ is contained in $W$ and therefore that $\eta$ is not in $\partial\NR(W)$. This prove the inequality $\partial\NR(W)\subseteq C$.
We hence have proved that $C=\partial\NR\bigl(\SStab_G(C)\bigr)$.

Now, let $W$ be a generalized parabolic subgroup, and let $C$ be the nowhere dense closed subset of $\partial T$ witnessing it: $W=\SStab_G(C)$.
By the above $C=\partial\NR(W)$.
This implies both that $\SStab_G\bigl(\partial\NR(W)\bigr)=\SStab_G(C)=W$ and that $\partial\NR(W)$ is a closed non-open subset of $\partial T$ such that the action $\SStab_G\bigl(\partial\NR(W)\bigr)\curvearrowright \partial\NR(W)$ is minimal.
\end{proof}
We obtain the following characterization of generalized parabolic subgroups.
\begin{corollary}\label{Cor:Equivalences}
Let $G$ be almost branch and let $W\leq G$ be a weakly maximal subgroup.
Then, the followings are equivalent
\begin{enumerate}
\item
$W$ is a generalized parabolic subgroup,
\item\label{Cor:Property:3}
$\SStab_G\bigl(\NR(W)\bigr)$ has infinite index in $G$,
\item\label{Cor:Property:4}
$\partial \NR(W)$ is not open,
\item
The number of orbit-closures of the action $W\curvearrowright \partial T$ is infinite.
\end{enumerate}
\end{corollary}
\begin{proof}
Observe that $\SStab_G\bigl(\NR(W)\bigr)=\SStab_G\bigl(\partial\NR(W)\bigr)$.
Lemmas~\ref{Lemma:InfOrbitsClosure} and~\ref{Lemma:OpenFiniteIndex} imply the equivalence of Properties~\ref{Cor:Property:3} and~\ref{Cor:Property:4}.

Lemma~\ref{Lemma:InverseBijections} and the fact that $W\leq\SStab_G\bigl(\NR(W)\bigr)$ imply the equivalence of the first two properties.

It remains to show that a weakly maximal subgroup is generalized parabolic if and only if the number of orbit-closures of the action $W\curvearrowright \partial T$ is infinite.
The left-to-right implication is Lemma~\ref{Lemma:InfOrbitsClosure}.
Finally, suppose that $W$ is not generalized parabolic.
We claim that in this case every orbit-closure $C$ of the action $W\curvearrowright \partial T$ is open.
Indeed, by definition $W$ is contained in $\SStab_G(C)$ and acts minimally on $C$.
If $C$ was not open, then by Lemma~\ref{Lemma:InfOrbitsClosure}, $\SStab_G(C)$ would have infinite index in $G$ and hence be equal to $W$, which is absurd.
Since all orbit-closures are open, by compacity of $\partial T$ there are only a finite number of them.
\end{proof}
By Lemma~\ref{Lemma:InverseBijections}, generalized parabolic subgroups act minimally on the boundary of their non-rigidity tree.
We conjecture that this fact generalizes to every weakly maximal subgroup.
\begin{conjecture}\label{Conj:MinimalActionOnNRTree}
Let $G$ be a branch group and let $W\leq G$ be a weakly maximal subgroup.
Then the action $W\curvearrowright\partial\NR(W)$ is minimal.
\end{conjecture}
\begin{remark}
As a consequence of Lemma~\ref{Lemma:InverseBijections}, distinct generalized parabolic subgroups have distinct non-rigidity trees.
This does not generalize to distinct weakly maximal subgroups.
Indeed, if $W$ is a weakly maximal subgroup then for every $g$ in $\SStab_G\bigl(\partial \NR(W)\bigr)$, the non-rigidity tree of $W^g$ coincides with the one for $W$.
But if $W$ is not generalized parabolic, then $\SStab_G\bigl(\partial \NR(W)\bigr)$ is a finite index subgroup of $G$, while if $G$ is torsion $W$ is self-normalizing by~\cite{MR3478865}.
That is, in a torsion branch group, weakly maximal subgroups that are not generalized parabolic admit infinitely many distinct conjugates with the same non-rigidity tree.
\end{remark}
This remark raises the following question.
\begin{question}
Let $G$ be a branch group and $W$ and $M$ two weakly maximal subgroups with the same non-rigidity tree $\NR(W)=\NR(M)$.
Does this necessarily imply that $W$ and $M$ are conjugated?
\end{question}
Observe that if $W$ is level-transitive, then it contains no rigid stabilizer and hence $\NR(W)=T$.
It is unclear if the converse holds.
We hence have the following weaker form of Conjecture~\ref{Conj:MinimalActionOnNRTree}
\begin{question}\label{Question:LevelTransitive}
Let $G$ be a branch group and $W$ be a weakly maximal subgroup.
Is it true that $W$ is level-transitive if and only if $\NR(W)=T$?
\end{question}
The following lemma shows that possible counterexamples to Conjecture~\ref{Conj:MinimalActionOnNRTree} are very constrained.
\begin{lemma}\label{Lemma:sub-directProductProjects}
Let $G$ be any subgroup of $\Aut(T)$ and let $W\leq G$ be a weakly maximal subgroup.
Let $n$ be an integer such that $W$ does not act transitively on $\NR(W)\cap \level{n}$. Then for every vertex $v$ of level $n$, the subgroup $\pi_v\bigl(\Stab_W(v)\bigr)$ has finite index in $\pi_v\bigl(\Stab_G(v)\bigr)$.
\end{lemma}
\begin{proof}
Let $u$ be any vertex of $\NR(W)\cap\level{n}$ and $W'\coloneqq\gen{W,\Rist_G(u)}$.
This is a finite index subgroup of $G$, and hence for every $v$ of level $n$ the subgroup $H_v=\pi_v\bigl(\Stab_{W'}(v)\bigr)$ has finite index in $\pi_v\bigl(\Stab_{G}(v)\bigr)$.
Now, let $\{u_\alpha\}_{\alpha \in I}$ be the $W$-orbit of $u$.
We have $W'=W\cdot\prod_{\alpha\in I}\Rist_G(u_\alpha)$.
In particular, for every vertex $v$ of level $n$ and for every $g$ in $\Stab_{W'}(v)$ there exists $w\in W$ and $g_\alpha$ in $\Rist_G(u_\alpha)$ (with all the $g_\alpha$ trivial except for a finite number) such that $g=w\prod_{\alpha\in I}g_\alpha$.
Equivalently, we have $w=g\prod_{\alpha\in I}g_\alpha^{-1}$.
If $v$ is not in the $W$-orbit of $u$, then $\treesection{w}{v}=\treesection{g}{v}$ which implies that $\pi_v\bigl(\Stab_{W}(v)\bigr)=H_v$.

Finally, if the action of $W$ on $\NR(W)\cap\level{n}$ is not transitive, then for every $v$ of level $n$ there exists $u$ in $\NR(W)\cap\level{n}$, with $v$ not in the $W$-orbit of $u$ and we are done.
\end{proof}
As a corollary, we prove Lemma~\ref{Lemma:WMaxInfinite}.
\begin{corollary}\label{Cor:WMaxInfinite}
Let $T$ be a non necessarily locally finite tree and let $G\leq\Aut(T)$ be micro-supported.
Then every weakly maximal subgroup of $G$ is infinite.
\end{corollary}
\begin{proof}
If $\NR(W)$ is not equal to $T$, then $W$ contains a finite index subgroup of some $\Rist_G(v)$ and is hence infinite.
If $\NR(W)=T$ and $W$ acts level-transitively on it, then it is infinite.
Finally, the last case is $\NR(W)=T$ but $W$ does not act level-transitively on it.
By the last lemma, $W$ contains a subgroup that projects onto some infinite group and we are done.
\end{proof}
In Corollary~\ref{Cor:ContinuumGenPar} we showed that a branch group $G$ with a torsion element has a continuum of generalized parabolic subgroups that are not parabolic.
If~$G$ is countable, then there is still a continuum of such subgroups, even up to conjugation or up to $\Aut(G)$.
Equivalence of subgroups up to conjugation, or up to $\Aut(G)$ were already considered in~\cite{MR3478865}.
The problem with these two equivalence relations is that each class contains at most countably many elements and hence parabolic subgroups split into many distinct classes.
This motivates the following definition.
\begin{definition}\label{Definition:TreeEquivalence}
Let $G$ be a subgroup of $\Aut(T)$.
Two weakly maximal subgroups $W_1$ and $W_2$ of $G$ are said to be \defi{tree-equivalent} if there is an automorphism of $T$ sending $\NR(W_1)$ to $\NR(W_2)$.
\end{definition}
Observe that this notion a priori depends on the chosen action.

It is possible to consider other equivalence relations on the set of weakly maximal subgroups, for example the fact that $\NR(W_1)$ and $\NR(W_2)$ are conjugated by an element of the profinite completion of $G$, or by an element of the normalizer of $G$ in $\Aut(T)$.
Among these possibilities, being tree-equivalent is the coarser equivalence relation, that is the one with the biggest classes.

It follows from Lemma~\ref{Lemma:InfOrbitsClosure} that
\begin{lemma}
Let $G\leq\Aut(T)$ be an almost level-transitive group and $W$ be a weakly maximal subgroup of $G$.
Then $W$ is tree-equivalent to a parabolic, respectively generalized parabolic, subgroup if and only if $W$ is parabolic, respectively generalized parabolic.
\end{lemma}
\begin{example}
Let $(m_i)_{i\in N}$ be a sequence of integers greater than $1$ and $T=T_{(m_i)}$ the corresponding spherically regular rooted tree.
Let $G$ be one of the group $\Aut(T)$, $\Autfr(T)$ or $\Autf(T)$.
Then tree-equivalence classes of generalized parabolic subgroups of $G$ are in bijection with sequences $(n_i)_{i\in N}$ of integers such that $1\leq n_i\leq m_i$ for all $i$ and such that $n_i<m_i$ for infinitely many $i$.
In particular, there is a continuum of such classes, each containing a continuum of distinct subgroups.

Indeed, this is the last lemma applied to Example~\ref{Example:GenParAutT}.
Lemma~\ref{Lemma:GenParDistinct} ensures that each equivalence class contains a continuum of subgroups.
\end{example}
The above example shows that for $\Aut(T)$, $\Autfr(T)$ and $\Autf(T)$ we have ``a lot'' of ``big'' classes of generalized parabolic subgroups.
But these groups are not finitely generated.
Corollary~\ref{Cor:NumberGeneralizedParabolic2} asserts that we still have this kind of result for more general groups if we allow either ``a lot'' or ``big'' to be weakened a little and mean infinitely many instead of a continuum.

In order to finish the proof of Corollary~\ref{Cor:NumberGeneralizedParabolic2} we will use the fact that in torsion groups weakly maximal subgroups are self-normalizing,~\cite{MR3478865}; that is, they are equal to their normalizer.

We now show this slight generalization of Corollary~\ref{Cor:NumberGeneralizedParabolic2}.
\begin{proposition}
Let $G\leq\Aut(T)$ be almost branch.
\begin{enumerate}
\item
If $G$ is not torsion-free, it contains a continuum of generalized parabolic subgroups that are not parabolic.\label{Prop:Torsion1}
\item
If $G$ has elements of arbitrarily high finite order, there are infinitely many tree-equivalence classes of generalized parabolic subgroups that each contains a continuum of subgroups.\label{Prop:Torsion2}
\item
If $G$ is torsion, there is a continuum of tree-equivalence classes of generalized parabolic subgroups that each contains infinitely many subgroups.\label{Prop:Torsion3}
\end{enumerate}
\end{proposition}
\begin{proof}\label{Proof:InfinitelyTreeClasses}
\textit{Proof of~\ref{Prop:Torsion1}.} This is Corollary~\ref{Cor:ContinuumGenPar}.

\textit{Proof of~\ref{Prop:Torsion2}.} Let $(g_i)_{i=1}^n$ be a sequence of elements of $G$ with finite but increasing order.
Let $C_i$ be an orbit of maximal size for $\gen{g_i}\curvearrowright \partial T$.
Up to extracting a subsequence, $(\abs{C_i})_i$ is a strictly increasing sequence of finite numbers.
All the $\SStab_G(C_i)$ are generalized parabolic subgroups and since the $C_i$ have pairwise distinct cardinalities, the $\SStab_G(C_i)$ are in distinct tree-equivalence classes.

On the other hand, since $C_i$ is finite, there exists a vertex $v_i$ of $T$ such that $\gen{g_i}.\{v_i\}$ contains $\abs{C_i}$ elements.
Let $\{v_{i,j}\}_{j=1}^{\abs{C_i}}$ be these elements.
By maximality, for every ray $\xi$ passing through $v_i$, its orbit under $\gen{g_i}$ consists of exactly one ray under each $v_{i,j}$ for $1\leq j\leq C_i$.
That is, for every $\xi$ passing through $v_i$, the subgroup $\SStab_G(\gen{g_i}.\{\xi\})$ is a generalized parabolic subgroup that is tree-equivalent to $\SStab_G(C_i)$.
There is a continuum of such $\xi$ giving raise to pairwise distinct subgroups by Lemma~\ref{Lemma:GenParDistinct}.
This finishes the proof of the second part of the proposition.

\textit{Proof of~\ref{Prop:Torsion3}.}
Suppose now that $G$ is torsion.
Then weakly maximal subgroups are self-normalizing.
In consequence, any weakly maximal subgroup of $G$ has infinitely many conjugates, which are all tree-equivalent.

Now, let $\xi=(v_i)_i$ be a ray in $\partial T$.
By Lemma~\ref{Lemma:RistLevelTransitive}, there exists an element $1\neq g_1$ in $\Rist_G(v_1)$ that moves $\xi$.
Since this element is torsion, the orbit $\gen{g_1}. \xi$ is finite.
In particular, there exists $i_2$ such that $g_1.v_{i_2}\neq v_{i_2}$ and $\Stab_{\gen{g_1}}(v)$ acts trivially on $T_{v}$ for every $v$ of level $i_2$ in the $\gen{g_1}$-orbit of $\xi$.
Let $1\neq g_2$ be any element of $\Rist_G(v_{i_2+1})$ moving $\xi$.
Again, $g_2$ being torsion, there exists $i_3$ such that $g_2.v_{i_3}\neq v_{i_3}$ and $\Stab_{\gen{g_2}}(v)$ acts trivially on $T_{v}$ for every $v$ of level $i_3$ in the $\gen{g_2}$-orbit of $\xi$.
Since $\gen{g_1,g_2}=\gen{g_1}\cdot\gen{g_2}^{\gen{g_1}}$ and the fact that $g_2$ is in $\Rist_G(v_{i_2+1})$, we obtain that $\Stab_{\gen{g_1,g_2}}(v)$ acts trivially on $T_{v}$ for every $v$ of level $i_3$ in the $\gen{g_1,g_2}$-orbit of $\xi$.
By induction, we obtain a sequence of integers $(i_j)_{j\geq 1}$, $i_1=1$, and a sequence of elements $(g_j)$ of $G$ such that $1\neq g_j$ belongs to $\Rist_G(v_{i_j+1})$, moves $v_{i_{j+1}}$ and such that the subgroup $\Stab_{\gen{g_1,\dots, g_j}}(v)$ acts trivially on $T_v$ for every $v$ of level $i_{j+1}$ in the $\gen{g_1,\dots,g_j}$-orbit of $\xi$.
Let $\mathbf b=(b_j)_{j\geq 1}$ be a binary sequence.
We associate to it the subset $C_{\mathbf b}\coloneqq\overline{\gen{g_j^{b_j},j\geq 1}.\xi}$  of $\partial T$ and the subgroup $W_{\mathbf b}\coloneqq\SStab_G(C_{\mathbf b})$.
We claim that all the $W_{\mathbf b}$ are generalized parabolic subgroups of $G$ and that if $\mathbf b\neq\mathbf b'$, then $W_{\mathbf b}$ and $W_{\mathbf b'}$ are not tree-equivalent.
Since there is a continuum of binary sequences, the claim implies the last part of the proposition.

By definition, $C_{\mathbf b}$ is closed and $W_{\mathbf b}$ acts minimally on it.
In order to prove that $W_{\mathbf b}$ is generalized parabolic, it remains to show that $C_{\mathbf b}$ is not open.
For every $v_{i_j+1}$ and every sibling $w$ of $v_{i_j+1}$, no element of $\partial T_w$ is in $\gen{g_j,j\geq 1}.\xi$.
This implies that $\gen{g_j^{b_j},j\geq 1}.\xi$ does not contain $\partial T_{v_{i_j}}$ for every $j$ and so neither does $C_{\mathbf b}$.
Since $\xi$ is in $C_{\mathbf b}$, this proves that this subset of $\partial T$ is not open.

Finally, $W_{\mathbf b}$ and $W_{\mathbf b'}$ are not tree-equivalent if and only if there is no automorphism of $T$ sending $C_{\mathbf b}$ onto $C_{\mathbf b'}$.
Let $k$ be the smallest integer such that $b_k\neq b_k	'$.
One can suppose that $b_k=0$ while $b_k'=1$ and hence $A_k\coloneqq\gen{g_1^{b_1},\dots,g_k^{b_k}}=\gen{g_1^{b_1},\dots,g_{k-1}^{b_{k-1}}}$ while $B_k\coloneqq\gen{g_1^{b_1'},\dots,g_k^{b_k'}}=\gen{g_1^{b_1},\dots,g_{k-1}^{b_{k-1}},g_k}$.
By construction, these two subgroups have the same action on $\level{i_k}$, as well as on $T_v$ for $v\in \level{i_k}$ not on the $A_k$-orbit of $\xi$.
But, $\Stab_{A_k}(v)$ acts trivially on $T_v$ for every $v$ of level $i_{j+1}$ in the $A_k$-orbit of $\xi$, while $\Stab_{B_k}(v)$ contains $g_{k}$ and hence moves $v_{i_{k+1}}$.
We conclude that $\abs{\level{i_{k+1}}\cap A_k.\xi}<\abs{\level{i_{k+1}}\cap B_k.\xi}$.
Since the cardinality of $\level{i_{k+1}}\cap C_{\mathbf b}.\xi$ is equal to the cardinality of $\level{i_{k+1}}\cap \gen{g_1^{b_1},\dots,g_k^{b_k}}.\xi$ we have $\abs{\level{i_{k+1}}\cap C_{\mathbf b}.\xi}\neq\abs{\level{i_{k+1}}\cap C_{\mathbf b'}.\xi}$, which implies that $C_{\mathbf b}$ cannot be sent by an automorphism of $T$ onto $C_{\mathbf b'}$.
\end{proof}
%
%
%
%
%
%
%
%
%
%
%
%
%
%
%
%
\section{Sections of weakly maximal subgroups}\label{Section:Sections}
The aim of this section is to better understand sections of weakly maximal subgroups.
We will also introduce an operation of ``lifting'' subgroups of $\pi_v(G)$ to subgroups of $G$ and show that, for weakly maximal subgroup, it is the inverse of the section operation.

Before going further, we insist on the fact that when we write $\pi_v(G)$, this will always be a shorthand for $\pi_v\bigl(\Stab_G(v)\bigr)$.

First of all, we remark that sections behave nicely with respect to the properties of being branch and just infinite.
\begin{lemma}\label{Lemma:SectionJustInfinite}
Let $P$ be any property in the following list :
\begin{center}
$\{$almost level-transitive, level-transitive, micro-supported, rigid, [just infinite and branch]$\}$.
\end{center}
Let $G$ be a subgroup of $\Aut(T)$. If $G$ has $P$, then for every vertex $v$ the group $\pi_v(G)\leq\Aut(T_v)$ has also $P$.\end{lemma}
\begin{proof}
It is obvious that sections of (almost) level-transitive groups are (almost) level-transitive.
Similarly, sections of micro-supported (respectively rigid) groups are also micro-supported (respectively rigid) and hence sections of branch groups are branch.

By~\cite{MR1765119}, a branch group $G$ is just infinite if and only if for every $n$, the derived subgroup $\Rist_G(\level{n})'$ has finite index in $G$.
Suppose now that $G$ is branch and just infinite. Then for every integer $n$ the subgroup $\Rist_{G}(\level{n+\abs{v}})'$ has finite index in $G$.
It follows that $\pi_v\bigl(\Rist_{G}(\level{n+\abs{v}})'\bigr)$ has finite index in $\pi_v(G)$.
It follows from the definition of $\pi_v$ that $\pi_v\bigl(\Rist_{G}(\level{n+\abs{v}})'\bigr)=\Rist_{\pi_v(G)}(\level{n})'$, which concludes the proof.
\end{proof}
\begin{lemma}\label{WMC:lemmaSections}
Let $G$ be a rigid subgroup of $\Aut(T)$.
Let $W$ be a weakly maximal subgroup of $G$ that is contained in the first level stabilizer.
Then at most one of the first level sections of $W$ is of infinite index.

More precisely, let $\{v_\alpha\}_{\alpha\in I}$ be the collection of the first level vertices.
If $[\pi_{v_\alpha}(G):\pi_{v_\alpha}(W)]$ is of infinite index for some $\alpha$, then $\pi_{v_\alpha}(W)$ is a weakly maximal subgroup of $\pi_{v_\alpha}\bigl(G)$ and for $\beta\neq \alpha$ the section $\pi_{v_\beta}(W)$ contains the section $\pi_{v_\beta}\bigl(\Rist_G(v_\beta)\bigr)$.
\end{lemma}
\begin{proof}
Suppose that $\pi_{\alpha}(W)\coloneqq\pi_{v_\alpha}(W)$ is of infinite index in $\pi_\alpha(G)$.
If it were not weakly maximal, then we would have an infinite index subgroup $L$ of $\pi_\alpha(G)$ with $\pi_\alpha(W)< L$.
For $h\in L\setminus \pi_\alpha(W)$, there exists $g$ in $\Stab_G(v_\alpha)$ such that $g$ projects onto $h$.
Then we have $W<W'\coloneqq\gen{W,g}\leq \Stab_G(v_\alpha)$.
On the other hand, $\pi_\alpha(W')\leq L$ is of infinite index in $\pi_{\alpha}(G)$ which implies that $W'$ is itself of infinite index in $G$, which is absurd.

On the other hand, let $h_\beta$ be in $\pi_{v_\beta}\bigl(\Rist_G(v_\beta)\bigr)$ and let $g_\beta$ be the only element of $\Rist_G(v_\beta)$ projecting to $h_\beta$.
Let $W'\coloneqq \gen{W,g_\beta}$.
We claim that $W'$ is equal to $W$.
Since $W$ is weakly maximal, it is enough to show that $W'$ is of infinite index.
But the section of $W'$ at $v_\alpha$ is equal to the section of $W$, and hence of infinite index, which finishes the proof.
\end{proof}
The following lemma will be useful in the next section.
\begin{lemma}\label{WMC:lemmaSectionsNotTriv}
Let $G$ be a rigid subgroup of $\Aut(T)$ and let $W$ be a weakly maximal subgroup of $G$.
Then no section of $W$ is trivial.
\end{lemma}
\begin{proof}
The proof is done by contradiction.
Let $v$ be a vertex such that $\pi_v(W)$ is trivial.
By~\cite[Lemma 4.2]{2020arXiv200608677L}, $\pi_v(\Rist_G(v))$ is not virtually cyclic, in particular, $\pi_v(W)=\{1\}$ is not weakly maximal in it.
Hence there exists $h$ in $\pi_v(\Rist_G(v))$ such that $\gen{h}=\gen{\pi_v(W),h}$ is of infinite index in $\pi_v(G)$.
Let $g$ be an element of $\Rist_G(v)$ projecting to $h$ and let $W'\coloneqq\gen{W,g}>W$.

Let $f$ be an element of $\Stab_{W'}(v)$. Then $f=w_1g^{\alpha_1}w_2g^{\alpha_2}\cdots g^{\alpha_n}w_{n+1}$ for some $w_i$ in $W$ with $w_1\cdots w_{n+1}\in\Stab_W(v)$.
We can rewrite $f$ as
\[
	f=w_1g^{\alpha_1}w_1^{-1}
	\cdot \dots\cdot(w_1\cdots w_n)g^{\alpha_n}(w_1\cdots w_n)^{-1}\cdot (w_1\cdots w_{n+1}).
\]
In particular, the section $\pi_v(f)$ is the product of the sections
\[
	\pi_v\bigl((w_1\cdots w_r)g^{\alpha_r}(w_1\cdots w_r)^{-1}\bigr)
\]
for $1\leq r\leq n$.
We claim that $\pi_v\bigl((w_1\cdots w_r)g^{\alpha_r}(w_1\cdots w_r)^{-1}\bigr)=\pi_v(g^{\alpha_r})=h^{\alpha_r}$.
Indeed, if $w_1\cdots w_r$ stabilizes $v$ then its section is trivial and the claim follows. While, if  $w_1\cdots w_r$ moves $v$ the claim follows from the fact that $g$ belongs to the rigid stabilizer of $v$.
We just proved that $\pi_v(W')=\gen{h}$ is of infinite index in $\pi_v(G)$, and hence $W'$ is itself of infinite index in $G$, a contradiction.
\end{proof}
For generalized parabolic subgroups, we even have a characterization of the non-rigidity tree in terms of sections.
\begin{lemma}\label{Lemma:ProjGenParab}
Let $G\leq\Aut(T)$ be almost branch and let
$W\leq G$ be a generalized parabolic subgroup.
Then $\pi_v(W)$ has infinite index in $\pi_v(G)$ if and only if $v$ is in $\NR(W)$. Moreover, if $v$ is in $\NR(W)$, then $\pi_v(W)$ is a weakly maximal subgroup of $\pi_v(G)$.
\end{lemma}
\begin{proof}
On one hand, if $v$ is not in $\NR(W)$, then $W$ contains $\Rist_G(v)$ by Lemma~\ref{Lemma:TrivialBranchKernel}.
But, by rigidity of $G$, the subgroup $\pi_v\bigl(\Rist_G(v)\bigr)$ has finite index in $\pi_v(G)$, and so has $\pi_v(W)$.

On the other hand, if $v$ is in $\NR(W)$, then $\pi_v(W)$ is contained in $\SStab_{\pi_v(G)}(C\cap\partial T_v)$.
Since $G$ is almost level-transitive and rigid, so is $\pi_v(G)$, while $C$ closed and nowhere dense implies the same properties for $C\cap \partial T_v$ and $\pi_v(W)$ acts minimally on $C\cap \partial T_v$ since $W$ does so on $C$.
By Proposition~\ref{Proposition:WMStrongRigidSStab} this implies that $\SStab_{\pi_v(G)}(C\cap\partial T_v)$ is weakly maximal in $\pi_v(G)$ and hence of infinite index, and so is $\pi_v(W)$.
It remains to prove that $\pi_v(W)$ is weakly maximal in $\pi_v(G)$. This is done similarly to Lemma~\ref{WMC:lemmaSections}, using that $W$ stabilizes $v$.
\end{proof}
The first Grigorchuk group as well as the torsion \GGS{} groups behave particularly well with respect to sections. This is the content of Lemma~\ref{Lemma:Francoeur} as well as of Corollary~\ref{Lemma:GuptaSidki}.

The following lemma is due to Dominik Francoeur and plays a key role in the proof of the assertion~\ref{Thm:StructItem5} of Theorem~\ref{Thm:StructureWMax}.
We are grateful for his help.
This proof is based on the description made by Grigorchuk and Wilson in~\cite{MR2009443} of subgroups of the first Grigorchuk group $\Grig$.
\begin{lemma}[Francoeur]\label{Lemma:Francoeur}
Let $H\leq \Grig$ be a subgroup acting level-transitively on~$T$.
Then there exists a vertex $v$ such that $\pi_v(H) = \Grig$.
\end{lemma}
\begin{proof}
Suppose that this is not the case.
Then, for all $v\in T$, we have $\pi_v(H)\neq \Grig$.
In~\cite{MR1841763}, Pervova proved that the maximal subgroups of $\Grig$ all have index $2$. It follows that the derived subgroup $\Grig'$ is contained in the Frattini subgroup of $\Grig$ (they are in fact equal, see~\cite{MR2195454}).
This implies that $\pi_v(H)\Grig'\neq\Grig$.
It follows from Lemma 6 of~\cite{MR2009443} that if $\pi_v(H)\Grig'\leq \gen{a,x}\Grig'$ for some $x\in\{b,c,d\}$, then there exists some $w\geq v$ such that $\pi_w(H)\leq\Stab_\Grig(\level{1})$.
This is of course absurd, since $H$ acts level-transitively on $T$.
We conclude that $\pi_v(H)\Grig'$ must be either $\gen{b, ad}\Grig'$, $\gen{c, ab}\Grig'$ or $\gen{d, ac}\Grig'$.
Then, by Lemma 5 of~\cite{MR2009443}, we have that $\pi_v(H)$ is actually equal to $\gen{b, ad}$, $\gen{c, ab}$ or $\gen{d, ac}$ respectively.
It is easy to check that the sections of the stabilisers of each of these groups is eventually $\Grig$, which contradicts our hypothesis.
\end{proof}
The following lemma is a generalization to all torsion \GGS{} groups of a result that was originally obtained by Garrido,~\cite{MR3513107}, for Gupta-Sidki groups.
\begin{lemma}\label{Lemma:GGSSections}
Let $G$ be a torsion \GGS{} group and $H$ be a subgroup of $G$ that is not contained in $\Stab_G(\level{1})$.
Then either all first level sections of $H$ are equal to $G$, or they are all contained in $\Stab_G(\level{1})$ so that $\Stab_H(\level{1})=\Stab_H(\level{2})$.
\end{lemma}
\begin{proof}
Denote by $H_0=\pi_0(H),\dots,H_{p-1}$ the first level sections of $H$.
Since $H$ does not fix pointwise the first level, they are all conjugated in $G$ and we may thus assume that no $H_i$ is equal to $G$.
Being conjugated in $G$, all the $H_i$ have the same image modulo $G'$, the derived subgroup of $G$.
The possibilities for these are $\{1\}$, $\gen{b}$, $\gen{a}\gen{b}$ or $\gen{ab^k}$ for $k\in \mathbf F_p$.

If $H_iG'=G'$ or $H_iG'=\gen{b}G'$, then $H_i\leq H_iG'\leq\Stab_G(\level{1})$ and $\Stab_H(\level{1})=\Stab_H(\level{2})$ as desired.
We will show that the remaining two cases cannot happen.

If $H_iG'=\gen{a}\gen{b}G'$. Recall that we supposed $H_i\neq G$ and hence $H_i$ is contained in some maximal subgroup $M<G$.
By~\cite{MR2197824}, all maximal subgroups of $G$ have finite index which by~\cite{2015arXiv150908090M} implies that $G'$ is contained in the Frattini subgroup of $G$.
But then $M\geq G'$ and $G=\gen{a}\gen{b} G'=H_iG'\leq MG'=M$, a contradiction.

Finally, suppose that there exists $k\in \mathbf F_p$ such that $H_iG'\leq\gen{ab^k}G'$ for all~$i$.
Let $h=(h_0,\dots,h_{p-1})$ be an element of $\Stab_H(\level{1})$.
Then there exists some $r_0,\dots,r_{p-1}$ in $\mathbf F_p$ such that
\[
	(h_0G',\dots,h_{p-1}G')=\bigl((ab^k)^{r_0}G',\dots,(ab^k)^{r_{p-1}}G'\bigr).
\]
By Lemma~\ref{Lemma:GGSNormalForm} there exists $n_i$ in $\mathbf F_p$ such that 
\[
	(n_0,\dots,n_{p-1})C= (r_0,\dots,r_{p-1}) \quad\textnormal{and}\quad (n_0,\dots,n_{p-1})P=k(r_0,\dots,r_{p-1})
\]
where $P$ is the permutation matrix associated to $(1 2 \dots p)$ and $C$ is the circulant matrix $\Circ(e_0,\dots,e_{p-2},0)$.
These equations are equivalent to 
\[
	(n_0,\dots,n_{p-1})(kCP^{-1}-\Id)=(0,\dots,0)
\]
By~\cite{MR3152720}, if $D=\Circ(a_0,\dots,a_{p-1})$ is a circulant matrix with entries in $\mathbf F_p$, then $D$ is invertible as soon as $\sum_{i=0}^{p-1}a_i\neq 0$.
Since $G$ is of torsion we have $\sum_ie_i=0$ which implies that the matrix $kCP^{-1}-\Id=\Circ(ke_1-1,ke_2,\dots,ke_{p-2},0,ke_0)$ is invertible by the above criterion.
Therefore, $(n_0,\dots,n_{p-1})=(0,\dots,0)$ is the only solution to our equation and all the $h_i$ are in $G'$. That is, $H_iG'=G'$, which finishes the proof.
\end{proof}
This directly implies
\begin{corollary}\label{Lemma:GuptaSidki} Let $G$ be a torsion \GGS{} group and $H\leq G$ be a subgroup acting level-transitively on $T$.
Then, for every vertex $v$ of the first level, $\pi_v(H)=G$.
\end{corollary}
Corollary~\ref{Thm:NumberGeneralizedParabolic} is a result about weakly maximal subgroups containing a given ``small'' subgroup of $G$.
We now turn our attention at weakly maximal subgroups contained in a given ``big'' subgroup.
\begin{definition}\label{Definition:Wv}
Let $G$ be a subgroup of $\Aut(T)$, $v$ a vertex of $T$ and $H$ a subgroup of $\pi_v(G)$.
The subgroup $H^{v}$ is defined by
\[
	H^v\coloneqq\setst{g\in\Stab_G(v)}{\treesection{g}{v}\in H}.
\]
\end{definition}
This is the biggest subgroup of $\Stab_G(v)$ such that $\pi_v(H^v)=H$.

Observe that $H^v$ contains $\Rist_G(u)$ for every $u$ not in $T_v$ and 
that the rank of $H$ is at most the rank of $H^v$.

Since $\pi_{g.v}(G)=\pi_v(G)^g$, if $v$ and $w$ are of the same level and $G$ is level-transitive, the subgroups $\pi_v(G)$ and $\pi_w(G)$ are conjugated and hence abstractly isomorphic.
We can define a partial map 
\begin{align*}
	\eta\colon T\times\setst{H\leq \pi_v(G)}{v\in T}&\to\{H\leq G\}\\
\end{align*}
that is defined for couples $(v,H)$ with $H\leq \pi_v(G)$, and in this case $\eta(v,H)=H^v$.
Observe that if $G$ is self-replicating (that is, $\pi_v(G)=G$ for every vertex), then the domain of definition of $\eta$ is exactly $T\times\{H\leq G\}$.

The main application of the following proposition is for self-replicating branch group in which case it implies Theorem~\ref{Thm:CopyOfWeaklyMaxInStab}.
Nevertheless it can also be useful in the more general context of self-replicating families of groups.
\begin{proposition}\label{Prop:NewWmax}
Let $G\leq\Aut(T)$ be a rigid group, $v$ a vertex of $T$ of level $n$ and $H$ a subgroup of $\pi_v(G)$.
Then 
\begin{enumerate}
\item $[G:H^v]$ is finite if and only if $[\pi_v(G):H]$ is finite. More precisely, $[\pi_v(G):H]\leq[G:H^v]\leq [\pi_v(G):H]\cdot[G:\Rist_G(n)]$,\label{Prop:NewWmax1}
\item $H^v$ is weakly maximal if and only if $H$ is weakly maximal,\label{Prop:NewWmax2}
\item If $G$ is finitely generated, then $H^v$ is finitely generated if and only if $H$ is finitely generated,\label{Prop:NewWmax3}
\item For every $v\in T$, the function $\eta(v,\cdot)$ is injective on $\{H\leq\pi_v(G)\}$,\label{Prop:NewWmax4}
\item Let $v\neq u$ be two vertices with $T_v\cong T_u$ so we might identify $\Aut(T_u)\cong\Aut(T_v)$. Let $H$ be a subgroup of $\pi_v(G)\cap\pi_u(G)$ that has infinite index in both $\pi_v(G)$ and $\pi_u(G)$\footnote{In fact $H$ such that $\pi_v(\Rist_G(v))\not\leq H$ and $\pi_u(\Rist_G(u))\not\leq H$ is enough.}. If $H$ has no fixed point except for the root, then $H^v\neq H^u$, \label{Prop:NewWmax5}
\item If $G$ is self-replicating and $H\leq G$ has infinite index and no fixed point except for the root, then $\eta(\cdot, H)$ is injective,\label{Prop:NewWmax5b}
\item $\pi_v(H^v)=H$,\label{Prop:NewWmax7}
\item For $L\leq G$, we have $\bigl(\pi_v(L)\bigr)^v\geq \Stab_L(v)$. Moreover, if $L$ is weakly maximal and $v$ is the unique vertex of level $n$ in $\NR(L)$, then $\bigl(\pi_v(L)\bigr)^v=L$.\label{Prop:NewWmax8}
\end{enumerate}
\end{proposition}
\begin{proof}
\textit{Proof of~\ref{Prop:NewWmax1}.}
We begin by proving a claim that directly implies the first assertion of the proposition, but will also be useful for the rest of the proof.
Let $K$ be a subgroup of $G$ such that $\pi_v(K)=H$ and $K$ contains $\Rist_G(w)$ for every $w$ in $\level{n}\setminus\{v\}$.
We claim that $[G:K]$ is finite if and only if $[\pi_v(G):H]$ is finite and that in this case $[\pi_v(G):H]\leq[G:K]\leq [\pi_v(G):H]\cdot[G:\Rist_G(n)]$.

Since $\pi_v(K)=H$, we have $[G:K]\geq [\pi_v(G):H]$.
On the other hand, let $d=[G:\Rist_G(n)]$ and $d'=[K:K\cap\Rist_G(n)]$.
Since $G$ is rigid, both $d$ and $d'$ are finite and $[G:K]=\frac{d}{d'}[\Rist_G(n):K\cap \Rist_G(n)]$.
The assumption on $K$ implies that $[\Rist_G(n):K\cap \Rist_G(n)]=[\pi_v\bigl(\Rist_G(n)\bigr):\pi_v\bigl(K\cap \Rist_G(n)\bigr)]\leq[\pi_v(G):\pi_v\bigl(K\cap \Rist_G(n)\bigr)]$.
Finally, $[\pi_v(K):\pi_v\bigl(K\cap \Rist_G(n)\bigr)]\leq d'$.
Altogether, this gives us
\[
	[G:K]\leq d\frac{[\pi_v(G):\pi_v\bigl(K\cap \Rist_G(n)\bigr)]}{[\pi_v(K):\pi_v\bigl(K\cap \Rist_G(n)\bigr)]} = d\cdot [\pi_v(G):\pi_v(K)]=d\cdot [\pi_v(G):H].
\]

\textit{Proof of~\ref{Prop:NewWmax2}.}
We now prove that $H^v$ is weakly maximal if $H$ is.
Let $g$ be an element of $G$ that is not in $H^v$.
If $g$ is not in $\Stab_G(v)$, then $\gen{g,H^v}$ contains $\Rist_G(\level{n})$ and is of finite index.
If $g$ is in $\Stab_G(v)$, the subgroup $\gen{g,H^v}$ stabilizes $v$ and thus $\pi_v(\gen{g,H^v})=\gen{\treesection{g}{v},H}$ is a finite index subgroup of $\pi_v(G)$.
By the above claim, this implies that $\gen{g,H^v}$ is a finite index subgroup of $G$, that is that $H^v$ is weakly maximal.

Suppose now that $H^v$ is a weakly maximal subgroup of $G$ and let $g$ be in $\pi_v(G)\setminus H$.
Then there exists $f\in \Stab_G(v)\setminus H^v$ such that $\treesection{f}{v}=g$.
Therefore, the subgroup $\gen{H^v,f}$ is of finite index in $G$ and so is $\pi_v(\gen{H^v,f})=\gen{H,g}$.

\textit{Proof of~\ref{Prop:NewWmax3}.}
We already know that $\rank(H)\leq\rank(H^v)$.
It remains to show that if $H$ is finitely generated, then so is $H^v$.
For that, we turn our attention to the homomorphism
\[
	\pi_v\colon\Stab_G(v)\to \pi_v(G)\qquad g\mapsto \treesection{g}{v}
\]
and more particularly to its kernel.
By definition, $\ker(\pi_v)$ is the set of all elements $g$ in $\Stab_G(v)$ acting trivially on $T_v$.
Let  $R\coloneqq\prod_{v\neq w\in\level{n}}\Rist_G(w)$.
Then $\ker(\pi_v)\cap\Rist_G(\level{n})=R$.
Since $R$ is a quotient of $\Rist_G(\level{n})$, a finite index subgroup of $G$, $R$ is finitely generated if $G$ is.
On the other hand, since $\Rist_G(\level{n})$ is a finite index subgroup of $G$, $R$ has finite index in $\ker(\pi_v)$.
Therefore, $\ker(\pi_v)$ is finitely generated.
We now look at the restriction of $\pi_v$ to $H^v$.
We have $\restr{\pi_v}{H^v}\colon H^v\twoheadrightarrow H$ and $\ker(\restr{\pi_v}{H^v})=\ker(\pi_v)\cap H^v$.
But $\ker(\pi_v)$ is contained in $H^v$, and therefore $\ker(\restr{\pi_v}{H^v})=\ker(\pi_v)$ is finitely generated.
Since $\restr{\pi_v}{H^v}$ is onto $H$, this implies that if $H$ is finitely generated, so is $H^v$.

\textit{Proof of~\ref{Prop:NewWmax4}.}
It follows from $\pi_v(H^v)=H$, that $H^v=K^v$ if and only if $H=K$.

\textit{Proof of~\ref{Prop:NewWmax5}.}
Let $H\leq\pi_v(G)\cap\pi_u(G)$ be a subgroup  with no fixed points (except for the root) that has infinite index in both $\pi_v(G)$ and $\pi_u(G)$ and let $M\coloneqq\gen{H^v,H^u}$.
Then $M$ contains every $\Rist_G(w)$ for $w$ not in $T_v$, but also for every $w$ not in $T_u$.
That is, if $u$ and $v$ are incomparable in $T$ ($v$ is not a descendant of $u$ nor $u$ is a descendant of $v$), $M$ contains a rigid stabilizer of a level and is of finite index in $G$.
In particular, $H^v\neq H^u$.
It remains to treat the case where $u\in T_v$ (the case $v\in T_u$ is done by symmetry).
But then, $M$ contains all $\Rist_G(w)$ for $w\neq u$ of the same level of $u$, and also $\Rist_G(u)$ since the action of $H^v$ on $T_v$ is the same as the action of $H$ and hence moves $u$.

\textit{Proof of~\ref{Prop:NewWmax5b}.} This directly follows from~\ref{Prop:NewWmax5}.

\textit{Proof of~\ref{Prop:NewWmax7}.}
This directly follows from the definitions.

\textit{Proof of~\ref{Prop:NewWmax8}.}
The inequality $\Stab_L(v)\leq\bigl(\pi_v(L)\bigr)^v$ follows from the definitions.
Suppose now that $L$ is weakly maximal and that $v$ is the unique vertex of level $n$ in $\NR(L)$.
Then $L$ stabilizes $v$ and we have $L\leq\bigl(\pi_v(L)\bigr)^v$.
Since $L$ projects onto $\pi_v(L)$, we may use the claim of the beginning of the proof.
Applying this to the pairs $\bigl(L,\pi_v(L)\bigr)$ and $\bigl((\pi_v(L))^v,\pi_v(L)\bigr)$ we obtain that $\bigl(\pi_v(L)\bigr)^v$ is an infinite index subgroup of $G$ containing the weakly maximal subgroup $L$.
They are thus equal.
\end{proof}
As a direct corollary, we obtain this generalization of Theorem~\ref{Thm:CopyOfWeaklyMaxInStab}
\begin{corollary}\label{Corollary:ProjInSelfReplicating}
Let $G$ be a self-replicating rigid group.
Then for any weakly maximal subgroup $W$ of $G$ and any vertex $v$ of $G$, there exists a weakly maximal subgroup $W^v$ of~$G$ that is contained in $\Stab_G(v)$ and with $\pi_v(W^v)=W$.
\end{corollary}
More generally, one useful application of Proposition~\ref{Prop:NewWmax} is the possibility to take a weakly maximal subgroup $W$ with non-rigidity tree $S$ and to cut or add some trunk to $S$ to obtain a new subtree $\tilde S$ in $T$ and to automatically have a weakly maximal subgroup $\tilde W$ with $\NR(\tilde W)=\tilde S$. 

The procedure applies to a family of self-replicating groups, but it is easier to describe it for a single self-replicating rigid group $G$.
The \defi{trunk} of a rooted subtree $S$ of $T$ is the maximal subtree $B$ of $S$ containing the root and such that every vertex of $B$ has exactly~$1$ child.
The trunk is equal to the whole subtree $S$ if and only if $S$ is a ray, otherwise it is finite (possibly empty).
The \defi{crown} $C$ of the subtree $S$ is $S$ minus its trunk.
Now, let $W$ be any weakly maximal subgroup of $G$, and let $B$ and $C$ be respectively the trunk and the crown of $\NR(W)$.
Since adding or removing an initial segment to a ray still produces a ray and that parabolic subgroups are always weakly maximal, we may suppose that $C$ is not empty.
Let $v$ be the vertex of $C$ of smallest level.
Then $\pi_v(W)$ is a weakly maximal subgroup of $\pi_v(G)=G$ with empty trunk and with crown isomorphic to $C$ (via the identification of $T$ and $T_v$).
Now, for every vertex $w$ in $T$, the subgroup $\bigl(\pi_v(W)\bigr)^w$ is a weakly maximal subgroup of $G$ with trunk the unique segment from the root to $w$ and with crown isomorphic to $C$ (via the identification of $T_w$ and $T_v$).
%
%
%
%
%
%
%
%
%
\section{Block subgroups}\label{Section:BlockSubgroups}
Other interesting subgroups of groups acting on rooted trees are subgroups with block structure.
It follows from~\cite{MR4344374,FL2019,2024arXiv240215496F} that for the first Grigorchuk group and torsion \GGS{} groups they coincide with the finitely generated subgroups.
Recall that two vertices $u$ and $v$ are incomparable if both $u\nleq v$ and $v\nleq u$, or equivalently if the subtrees $T_u$ and $T_v$ do not intersect.
Block subgroups are subgroups of automorphisms of $T$ that possess a special form. Abstractly, block subgroups are simply direct products, but some of the direct factors may be embedded ``diagonally''. To make this precise, let us first introduce the notion of a diagonal block subgroup. In order to do that, we will need the notion of the \defi{stabilized rigid stabilizer} of a subset $V\subset T$ of pairwise incomparable vertices:
\[
	\SRist_G(V)\coloneqq G\cap\prod_{v\in V}\Rist_{\Aut(T)}(v).
\]
In other words, $\SRist_G(V)$ consists of all elements of $G$ that act trivially outside of $\bigcup_{v\in V}T_v$. Observe that $\SRist_G(V)$ fixes $V$ pointwise.
\begin{definition}\label{defn:diag}
Let $G\leq \Aut(T)$ be a group of automorphisms of $T$ and let $V\subseteq T$ be a finite set of pairwise incomparable vertices of $T$. A subgroup $H\leq G$ is said to be a \defi{diagonal block subgroup with supporting vertex set $V$} if
\begin{enumerate}
\item $H \leq \SRist_G(V)$,
\item $\pi_v(H)$ is of finite index in $\pi_v(\Stab_G(v))$ for all $v\in V$,
\item for all $v\in V$, $\varphi_v$ is injective on $H$.
\end{enumerate}
\end{definition}
In other words, a diagonal block subgroup with supporting vertex set $V$ is a subgroup $H$ that is abstractly isomorphic to a finite index subgroup of $\pi_v(\Stab_G(v))$ for some $v\in V$, but which is embedded in $\SRist_G(V)$ ``diagonally'', in the sense that for every $v,w\in V$ and every $h\in H$, the ``$w$ component of $h$'', $\pi_w(h)$, is uniquely determined by its ``$v$ component'' $\pi_v(h)$. To illustrate, let us see a couple of important examples of diagonal block subgroups.
\begin{example}
Let $G\leq \Aut(T)$ be a branch group and let $v\in T$ be any vertex. Then, the diagonal block subgroups with supporting vertex set $V=\{v\}$ are exactly the finite index subgroups of $\Rist_G(v)$.
\end{example}
\begin{example}
Let $G\leq \Aut(T)$ be a self-replicating regular branch group over a subgroup $K$. Then, for every vertex  $v$, there exists a subgroup $K_{@v}\leq \Rist_K(v)$ such that $K_{@v} \cong K$. Let $f_v\colon K \rightarrow K@v$ be one such isomorphism, and let $v,w\in T$ be two incomparable elements. Then, the subgroup $H=\setst{f_v(k)f_w(k)}{ k\in K}$ is a diagonal block subgroup with supporting vertex set $V=\{v,w\}$. This example can be generalised to bigger sets of incomparable vertices in an obvious way, see Figure~\ref{fig:diag} for an example with supporting vertex set of cardinality $3$.
\end{example}
\begin{figure}[htbp]
\centering
\includegraphics{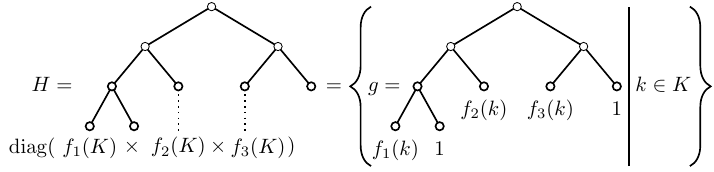}
\caption{A diagonal block subgroup over $K$ with  supporting vertex set of cardinality $3$.}
\label{fig:diag}
\end{figure}
Block subgroups are then simply direct products of diagonal block subgroups over pairwise incomparable supporting vertex sets.

\begin{definition}\label{defn:Block}
Let $G \leq \Aut(T)$ be a group of automorphisms of $T$ and let $H\leq G$ be a subgroup. Let $V=\bigsqcup_{i=1}^{b}V_i\subseteq T$ be a finite set of pairwise incomparable elements of $T$ partitioned into $b$ non-empty subsets, and let us denote by $P=\{V_i\}_{i=1}^{b}$ this partition. We say that $H$ is a \defi{block subgroup with supporting partition $P$} if
\[
	H=H_1\cdots H_b
\]
where each $H_i$ is a diagonal block subgroup with supporting vertex set $V_i$. See Figure~\ref{fig:block} for an example.

We say that a subgroup $H$ of $G$ has a \defi{block structure} if it is virtually a block subgroup.
\end{definition}
\begin{figure}[htbp]
\centering
\includegraphics{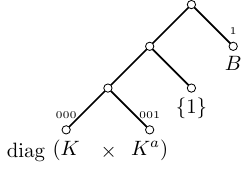}
\caption{A block subgroup with supporting partition $\{\{000,001\},1\}$. Here $H=H_1\cdot H_2$ with $H_1$ a block subgroup over $K$ with supporting partition $\{000,001\}$ and $H_2$ a block subgroup over $B$ with supporting partition $\{B\}$.}
\label{fig:block}
\end{figure}
\begin{remark}
Notice that in the previous definition, each $H_i$ is a subgroup of $\SRist_G(V_i)$. By assumption, when $i\ne j$, any element of $V_i$ is incomparable with any element of $V_j$. Therefore, the subgroups $H_i$ pairwise commute. Thus, a block subgroup is truly a group, and it is isomorphic to a direct product of diagonal block subgroups.
\end{remark}
%
%
The following is a direct consequence of the definitions.
\begin{lemma}\label{Lemma:ImplicationsBlock}
Let $G$ be a rigid group and let $H$ be a block subgroup of infinite index.
Then there exists $v$ such that $\Rist_H(v)$ is trivial
\end{lemma}
\begin{corollary}\label{Cor:NotBlock}
Let $G$ be a rigid group. Then generalized parabolic subgroups of $G$ do not have a block structure.
\end{corollary}
\begin{proof}
By Lemma~\ref{Lemma:GenParabWeaklyBranch}, the rigid stabilizers of a generalized parabolic subgroup $W$ are all infinite, and therefore $W$ cannot have a block structure.
\end{proof}
\begin{corollary}\label{Corollary:ImplicationsBlock}
Let $G$ be almost branch and let $W$ be a weakly maximal subgroup of $G$ with block structure.
Then there exists $n$ such that for every vertex $v$ of level $n$, the section $\pi_v(W)$ has finite index in $\pi_v(G)$.
\end{corollary}
\begin{proof}
We know that $W$ is not generalized parabolic and hence acts on $\partial T$ with finitely many orbit-closures by Corollary~\ref{Cor:Equivalences}.
The same remains true for every $\Stab_W(v)$.
This can happens only if none of the $\pi_v(W)$ are trivial.
By the definition of a block subgroup, this implies that there exists a transversal $X$ such that for every $v\in X$, the section $\pi_v(W)$ has finite index in $\pi_v(G)$.
By taking $n$ to be the maximal level of elements of $X$ we obtain the desired property.
\end{proof}
It follows from the definition that if $T$ is locally finite and $G\leq\Aut(T)$ is finitely generated, then every $H\leq G$ with a block structure is finitely generated.
It is natural to ask whether the converse is true. This is the case if $G$ has the so called subgroup induction property, which we introduce now.
\begin{definition}\label{Definition:StrongSubgroupInduction} Let $G\leq \Aut(T)$ be a self-similar group.
A class $\mathcal{X}$ of subgroups of $G$ is said to be \defi{inductive} if
\begin{enumerate}[(A)]
\item\label{Item:DefSubgroupInduction1}
Both $\{1\}$ and $G$ belong to $\mathcal{X}$,
\item\label{Item:DefSubgroupInduction2}If $H\leq L$ are two subgroups of $G$ with $[L:H]$ finite, then $L$ is in $\mathcal X$ if $H$ is in $\mathcal X$.
\item\label{Item:DefSubgroupInduction3}
If $H$ is a finitely generated subgroup of $\Stab_G(\level 1)$ and all first-level sections of $H$ are in $\mathcal{X}$, then $H\in \mathcal{X}$.
\end{enumerate}

The group $G$ has the \defi{subgroup induction property} if for any inductive class of subgroups $\mathcal X$, every finitely generated subgroup of $G$ is contained in $\mathcal X$.
\end{definition}
A slight modification of the Definition~\ref{Definition:StrongSubgroupInduction} leads to the a priori weaker notion of weak subgroup induction property.
We say that a class $\mathcal{X}$ of subgroups of $G$ is \defi{strongly inductive} if it satisfies~\ref{Item:DefSubgroupInduction1},~\ref{Item:DefSubgroupInduction3} and 
\begin{enumerate}\renewcommand{\theenumi}{B'}
\renewcommand{\labelenumi}{(\theenumi)}
\item\label{Item:StrongInductive}
If $H\leq L$ are two subgroups of $G$ with $[L:H]$ finite, then $L$ is in $\mathcal X$ if and only if $H$ is in $\mathcal X$,
\end{enumerate}
\begin{definition}\label{Definition:SubgroupInduction}
The group $G$ has the \defi{weak subgroup induction property} if for any strongly inductive class of subgroups $\mathcal X$, every finitely generated subgroup of $G$ is contained in $\mathcal X$.
\end{definition}
For super strongly self-replicating groups with the congruence subgroup property, the subgroup induction property and the weak subgroup induction property coincide~\cite{FL2019}, but this is not know to be true in general.
As for now, the only known examples of groups with the subgroup induction property are the first Grigorchuk group~\cite{MR2009443} and the torsion \GGS{} groups~\cite{MR3513107,FL2019}.
The subgroup properties as many interesting consequences and we refer the reader to~\cite{MR4344374,FL2019,2024arXiv240215496F} for a list those as below we will only name the few we will need.
Firstly, a self-similar group with the subgroup induction property is necessarily self-replicating~\cite{FL2019}[Proposition 3.2] and any branch group with the weak subgroup induction property is torsion and just infinite~\cite{FL2019}[Theorem A].
The subgroup induction property has also consequences on the profinite topology on $G$, as well as on maximal and weakly maximal subgroups of $G$
\begin{theorem}[{~\cite{FL2019}[Theorem C]}]\label{Thm:Separable}
Let $G\leq Aut(T)$ be a finitely generated self-similar branch group with the subgroup induction property. Then, finitely generated subgroups of $G$ are closed in the profinite topology. Suppose moreover the intersection of all finite index maximal subgroups of $G$ is not trivial (e.g. if $G$ is a $p$-group). Then, all maximal subgroups of $G$ are of finite index and all of its weakly maximal subgroups are closed in the profinite topology.
\end{theorem}
Then, we have two results about sections.
\begin{proposition}[{\cite{MR4344374}[Proposition 4.3]}]\label{Prop:SubSect1}
Let $G\leq \Aut(T)$ be a self-similar group. Then the following assertions are equivalent.
\begin{enumerate}
\item $G$ possesses the weak subgroup induction property,
\item  For every finitely generated subgroup $H\leq G$, there is a transversal $X$ of $T$ such that for every $v\in X$ the section $\pi_v(\Stab_H(X))$ is either trivial or of finite index in $\pi_v(\Stab_H(X))$.
\end{enumerate}
\end{proposition}
\begin{proposition}[{\cite{2024arXiv240215496F}[Proposition 5.15]}]\label{Prop:SubSect2}
Let $G\leq \Aut(T)$ be a finitely generated self-similar group with the subgroup induction property.
Then, for every finitely generated subgroup $H\leq G$, there exists a transversal $X$ of $T$ such that for every  $v\in X$ the section $\pi_v(\Stab_H(v))$ is either finite or equal to $G$.
\end{proposition}
Observe that Proposition~\ref{Prop:SubSect1} is about sections of $\Stab_H(X)$, while Proposition~\ref{Prop:SubSect1} is about sections of $\Stab_H(v)$. In both cases, the transversal $X$ can be chosen to be equal to $\level n$ for some $n$.

Using the above characterization, we can finally prove Proposition~\ref{Prop:GenParabNotFG}.
Observe that until now, we did not use this proposition.
Therefore the use of Corollary~\ref{Cor:WMaxInfinite} and Lemma~\ref{Lemma:ProjGenParab} in its proof does not create a circular argument.
\begin{lemma}\label{Lemma:GenParabNotFG}
Let $G\leq\Aut(T)$ be an almost branch group with the weak subgroup induction property.
Then generalized parabolic subgroups of $G$ are not finitely generated.
\end{lemma}
\begin{proof}
If $W$ is a generalized parabolic subgroup, any transversal $X$ contains at least one vertex $v$ of $\NR(W)$.
By Lemma~\ref{Lemma:ProjGenParab}, $\pi_v(W)$ is a weakly maximal subgroup of $\pi_v(G)$, and hence an infinite, by Corollary~\ref{Cor:WMaxInfinite}, subgroup of infinite index.
In particular, $\pi_v\bigl(\Stab_W(v)\bigr)$ is neither trivial, nor of finite index in~$\pi_v(G)$.
\end{proof}
We need one last definition before stating the main result of~\cite{2024arXiv240215496F}.
\begin{definition}[\cite{2024arXiv240215496F}]\label{defn:treeprim}
Let $G\leq \Aut(T)$ be a group of automorphisms of $T$.
We say that the action of $G$ on $T$ is \defi{tree-primitive} if for every $n\in \N$, the only $G$-invariant partitions of $\level n$ are the partitions $\{\level{n-m}(T_v)\}_{v\in \level m}$ for $0\leq m \leq n$.
\end{definition}
Roughly, an action of a group of automorphisms of a rooted tree is tree-primitive if it is as primitive as it can be; given that an action on a rooted tree can never be primitive, as it must preserve the partition of the tree into levels; see Definition~\ref{defn:treeprim}.
Examples of branch group with tree-primitive actions include the first Grigorchuk group~\cite{2024arXiv240215496F}[Corollary 4.6] as well as torsion \GGS{} groups~\cite{2024arXiv240215496F}[Corollary 4.7].
\begin{theorem}[\cite{MR4344374,2024arXiv240215496F}]\label{Thm:SubgroupInductionW}
Let $G\leq \Aut(T)$ be a finitely generated self-similar branch group with the subgroup induction property acting tree-primitively on~$T$, and let $H\leq G$ be a finitely generated subgroup of $G$. Then, $H$ has a block structure.
\end{theorem}
We are now able to prove a structural theorem for weakly maximal subgroups of branch groups.
This theorem encompasses Propositions~\ref{Proposition:Dichotomy1} and~\ref{Proposition:Dichotomy2}, as well as Theorem~\ref{Thm:ClassificationGrigGuptaSidki}.
\begin{theorem}\label{Thm:StructureWMax}
Let $G\leq\Aut(T)$ be almost branch and let $W\leq G$ be a weakly maximal subgroup.
For the following properties that $W$ may have,
\begin{enumerate}
\renewcommand\labelenumi{(\alph{enumi})}
\renewcommand\theenumi{\alph{enumi}}
\item $W$ has a block structure\label{Item:GrigA},
\item There exists $n$ such that $\pi_v(W)$ has finite index in $\pi_v(G)$ for every vertex of level~$n$\label{Item:GrigB},
\item There exists $n$ such that $\pi_v(W)=\pi_v(G)$ for every vertex of level~$n$\label{Item:GrigBprime},
\item $W$ is not micro-supported\label{Item:GrigC},
\item $W$ is almost level-transitive\label{Item:GrigD},
\item $W$ is not generalized parabolic\label{Item:GrigE},
\item $W$ is finitely generated\label{Item:GrigF}.
\end{enumerate}
we have the implications

\begin{enumerate}
\item In general\label{Thm:StructItem1}
\[ \begin{tikzcd}[arrows=Rightarrow]
	\eqref{Item:GrigA} \arrow[r] \arrow[d] 	& \eqref{Item:GrigB} \arrow[d]	&\arrow[l]\eqref{Item:GrigBprime} \\%
	\eqref{Item:GrigC} \arrow[r]				& \eqref{Item:GrigD} \arrow[r, Leftrightarrow]&\eqref{Item:GrigE},
\end{tikzcd}\]
\item If $G$ is finitely generated, then $\eqref{Item:GrigA}\implies \eqref{Item:GrigF}$,\label{Thm:StructItem2}
\item
	If $G$ is a finitely generated branch group with the weak subgroup induction property, then $\eqref{Item:GrigF}\implies \eqref{Item:GrigB}$. If moreover $G$ has the subgroup induction property, then $\eqref{Item:GrigF}\implies \eqref{Item:GrigBprime}$,\label{Thm:StructItem3}
\item
	If $G\leq \Aut(T)$ is a finitely generated self-similar branch group with the subgroup induction property acting tree-primitively on $T$, then $\eqref{Item:GrigF}\iff\eqref{Item:GrigBprime}\iff \eqref{Item:GrigA}$,\label{Thm:StructItem4}
\item 
	If $G$ is either the first Grigorchuk group, or a torsion \GGS{} group, then properties \eqref{Item:GrigA} to \eqref{Item:GrigF} are equivalent.\label{Thm:StructItem5}
\end{enumerate}
\end{theorem}
\begin{proof}
\textit{Proof of~\ref{Thm:StructItem1}.}
In every almost branch group, \eqref{Item:GrigD} and \eqref{Item:GrigE} are equivalent by Corollary~\ref{Cor:Equivalences} and they are both implied by \eqref{Item:GrigC} since generalized parabolic subgroups are micro-supported.
It is also clear that \eqref{Item:GrigB} always implies \eqref{Item:GrigD}.
On the other hand, \eqref{Item:GrigA} implies both \eqref{Item:GrigC} and \eqref{Item:GrigB} by Lemma~\ref{Lemma:ImplicationsBlock}  and Corollary~\ref{Corollary:ImplicationsBlock}.
Finally, the implication $\eqref{Item:GrigBprime}\implies\eqref{Item:GrigB}$ is trivial.

\textit{Proof of~\ref{Thm:StructItem2}.}
If $G$ is finitely generated, then \eqref{Item:GrigA} also implies \eqref{Item:GrigF} by the remark before Definition~\ref{Definition:StrongSubgroupInduction}.

\textit{Proof of~\ref{Thm:StructItem3}.}
Since $W$ is weakly maximal, none of its section is trivial by Lemma~\ref{WMC:lemmaSectionsNotTriv}.
We conclude using Propositions~\ref{Prop:SubSect1} or~\ref{Prop:SubSect2}.

\textit{Proof of~\ref{Thm:StructItem4}.}
The implication $\eqref{Item:GrigF}\implies \eqref{Item:GrigA}$ is Theorem~\ref{Thm:SubgroupInductionW}.
But the proof of Theorem~\ref{Thm:SubgroupInductionW} given~\cite{2024arXiv240215496F} is done in two steps. The first step is Proposition~\ref{Prop:SubSect2}. The second and more difficult step consists of showing that if $H$ is a subgroup with a transversal $X$ such that $\pi_v(\Stab_H(v)$ is either finite or equal to $G$ for all $v$ in $X$, then $H$ has a block structure. In other words, the proof~\cite{2024arXiv240215496F}[Theorem 6.13] already gives $\eqref{Item:GrigBprime}\implies \eqref{Item:GrigA}$.

\textit{Proof of~\ref{Thm:StructItem5}.}
Suppose that \eqref{Item:GrigD} holds, that is that $W$ is almost level-transitive.
Then there exists a transversal $X$ such that $\pi_v(W)$ acts level-transitively on $T_v$ for every $v$ in~$X$.
By Lemma~\ref{Lemma:Francoeur} for the first Grigorchuk group or Corollary~\ref{Lemma:GuptaSidki} for the torsion \GGS{} groups, we have another transversal $X'$ such that every section is equal to $G$, which implies \eqref{Item:GrigBprime}.
\end{proof}
Let $G$ be either the first Grigorchuk group, or a torsion \GGS{} group.
We already know that $G$ has a continuum of weakly maximal subgroups.
On the other hand, $G$ is finitely generated and thus has at most countably many weakly maximal subgroups that are finitely generated.
As a corollary of Theorem~\ref{Thm:StructureWMax}, we prove Proposition~\ref{Prop:InfinitelyBlockSubGrps} that says that, in some sense, $G$ has also as many finitely generated weakly maximal subgroups as it is possible.
\begin{proof}[Proof of Proposition~\ref{Prop:InfinitelyBlockSubGrps}.]
Let $\{v_1,\dots,v_p\}$ be the vertices of the first level and let $H$
be the diagonal subgroup $\diag\bigl(\Rist_G(v_1)\times\dots\times\dots\times\Rist_G(v_p)\bigr)$.
This is an infinite index subgroup and it is hence contained in a weakly maximal subgroup $W$.
The subgroup $H$, and hence also $W$, acts on $\partial T$ with finitely many orbit-closures.
In particular, $W$ is not generalized parabolic.

For $v$ a vertex in $T$, let $W^v$ be the subgroup of $G$ acting like $W$ on $T_v$, see Definition~\ref{Definition:Wv}.
By Proposition~\ref{Prop:NewWmax}, all the subgroups $W^v$ are weakly maximal and none of them is generalized parabolic.
Moreover, it follows from the description of $\NR(W)$ and from the discussion after Corollary~\ref{Corollary:ProjInSelfReplicating}, that if $v$ and $w$ are vertices of different level, then $W^v$ and $W^w$ are not tree equivalent.
Finally, since the $W^v$ are weakly maximal, they are self-normalizing by~\cite{MR3478865} and hence have infinitely many conjugate.
Every conjugate of $W^v$ is still a non generalized parabolic weakly maximal subgroup that is tree-equivalent to $W^v$ since $\NR(gW^vg^{-1})=\NR(W^v)^g$.
Since none of these subgroups are generalized parabolic, they all have a block structure.
\end{proof}
It is natural to ask if all properties listed in Theorem~\ref{Thm:StructureWMax} are always equivalent.
First of all, Properties~\eqref{Item:GrigB} and~\eqref{Item:GrigBprime} are very similar and are equivalent if and only id finite index subgroups of $G$ have sections (of a deep enough level) equal to~$G$.
This happens for example whenever $G$ is a super strongly self-replicating group with the congruence subgroup property.
In~\cite{MR3886188} Francoeur and Garrido classified maximal subgroups of non-torsion \v Suni\'c groups acting on the binary rooted tree and that are different from the infinite dihedral group.
All groups in this family are finitely generated, have trivial congruence kernel (and hence trivial branch kernel), are just infinite, branch and self-replicating.
Francoeur and Garrido exhibited maximal subgroups of infinite index that are finitely generated, micro-supported (that is, do not have property \eqref{Item:GrigC}) and do not have property \eqref{Item:GrigB} of Theorem~\ref{Thm:StructureWMax}.
Such subgroups are not closed in the profinite topology and hence are not generalized parabolic (that is, have property \eqref{Item:GrigE}).
To summarize this, even when restricted to finitely generated branch group that are just infinite, self-replicating and have the congruence subgroup property, \eqref{Item:GrigD} does not imply \eqref{Item:GrigB} or \eqref{Item:GrigC}, nor does \eqref{Item:GrigF} imply~\eqref{Item:GrigA}.

We call a weakly maximal subgroup that is both micro-supported and almost level-transitive (that is, with property \eqref{Item:GrigD} but not property \eqref{Item:GrigC} of Theorem~\ref{Thm:StructureWMax}) \defi{exotic}.
Every maximal subgroup of infinite index is exotic, and to our knowledge, all known examples of branch groups $G$ that have an exotic weakly maximal subgroup also have a maximal subgroup of infinite index.
\begin{remark}\label{Remark:Trichotomy}
If $G$ is such that all weakly maximal subgroups that are not micro-supported have a block structure, then we have a trichotomy for weakly maximal subgroups of $G$.
They are either generalized parabolic (if and only if they are micro-supported but not almost level-transitive), or have a block structure (if and only if they are almost level-transitive but not micro-supported) or they are exotic (both micro-supported and almost level-transitive).
\end{remark}

On the other hand, we are inclined to believe that for a large class of groups, properties \eqref{Item:GrigB} and \eqref{Item:GrigC} of Theorem~\ref{Thm:StructureWMax} are equivalent.
The heuristic behind this is that a subgroup with \eqref{Item:GrigB} but not \eqref{Item:GrigC} should be ``too big'' (that is, of finite index), while a subgroup with \eqref{Item:GrigC} but not \eqref{Item:GrigB} should be ``too small'' (that is, not weakly maximal).

In view of the above discussion we formulate the following questions.
\begin{question}\label{Question:4and6}
Is it true that in any branch group, properties \eqref{Item:GrigB} and \eqref{Item:GrigC} of Theorem~\ref{Thm:StructureWMax} are equivalent?
\end{question}
\begin{question}
Let $G$ be a branch group.
Suppose that $G$ does not have maximal subgroup of infinite index.
Does this imply that $G$ does not have exotic (both almost level-transitive and micro-supported) weakly maximal subgroup?
\end{question}
For both questions, if the answer happens to be negative, it would be interesting to both have a counterexample and to find a ``large'' class of branch groups in which the answer is true.

If Question~\ref{Question:4and6} admits a positive answer, then Remark~\ref{Remark:Trichotomy} apply to all just infinite branch groups.
%
%
%
%
%
%
%
%
%
%
%
%
%
%
%
%
\section{Level-transitive weakly maximal subgroups}\label{Section:LevelTransitive}
The aim of this section is to exhibit examples of weakly maximal subgroups that are level-transitive.
%
%
%
%
%
%
%
%
%
%
%
%
%
%
%
Recall that a subgroup $A$ in a branch group $G$ is dense for the congruence topology if and only if $\bigl(A\Stab_G(n)\bigr)/\Stab_G(n)=G/\Stab_G(n)$ for every~$n$.
On the other hand, $A$ is level-transitive if and only if, for every $n$ the subgroup $\bigl(A\Stab_G(n)\bigr)/\Stab_G(n)\leq G/\Stab_G(n)$ acts transitively on the $n$\textsuperscript{th} level of the tree.
This directly implies
\begin{fact}
In a branch group, every maximal subgroup of infinite index is a level-transitive weakly maximal subgroup.
\end{fact}
As we will see, some branch groups have level-transitive weakly maximal subgroups that are not maximal.
We will first outline a method to search for level-transitive weakly maximal subgroups in branch groups acting on the $2$-regular rooted tree.
We then illustrate this method in the particular example of the Grigrochuk group.
Finally, for every branch group $G$ we will construct some extension $\check G$ which will retains a lot of the properties of $G$ and show that $\check G$ always has a level-transitive weakly maximal subgroup.

All these constructions are variations of the following simple example.
\begin{example}
Let $T$ be a locally finite tree.
Let $G$ be a subgroup of $\Aut(T)$ such that $\Rist_G(\level{1})=G\times \dots\times G$ and such that the subgroup $R$ of rooted automorphisms of $G$ (elements of $G$ that rigidly permutes the trees rooted at the first-level vertices) acts transitively on the first level.
For example, $G$ is one of $\Aut(T)$, $\Autfr(T)$ or $\Autf(T)$, or is Bondarenko's example from~\cite{MR2727305}.
Then $\gen{R,\diag(G\times\dots\times G)}$ is an infinite index subgroup of $G$ acting level-transitively on $G$.
\end{example}
The condition on $\Rist_G(\level{1})$ in the above example is pretty restrictive.
Nevertheless, while the groups $\Aut(T)$, $\Autfr(T)$ and $\Autf(T)$ are never finitely generated, Bondarenko's example is finitely generated.
However, it is possible to adapt a little bit the construction in order to deal with other branch groups.
\begin{lemma}\label{Lemma:DiagonalTransitive}
Let $G$ be a weakly branch group acting on a $2$-regular tree.
Suppose that $G$ contains $a=(01)$ the automorphism of $T$ permuting rigidly $T_0$ and $T_1$.
Suppose moreover that there exists a subgroup $A$ of $G$ and an automorphism $\varphi$ of $G$ such that:
\begin{enumerate}
\item $\varphi^2(A)=A$ and the restriction of $\varphi^2$ on $A$ is the identity,
\item $\diag\bigl(A\times \varphi(A)\bigr)=\gen{\bigl(g,\varphi(g)\bigr)}_{g\in A}$ is a subgroup of $G$,
\item $\gen{A,\varphi(A)}$ acts level-transitively on $T$.
\end{enumerate}
Then the subgroup $L\coloneqq\gen{a,\diag\bigl(A\times \varphi(A)\bigr)}$ is an infinite index subgroup of $G$ that acts level-transitively on $T$.
\end{lemma}
\begin{proof}
The second condition implies that $L$ is a subgroup of $G$.
The third condition and the fact that $a$ acts transitively on the first level imply that $L$ acts level-transitively.
Finally, by definition of $a$, the rigid stabilizer in $L$ of the vertex $0$ consists of elements of the form
\[
	(g_1\varphi(h_1)\cdots g_n\varphi(h_n),\varphi(g_1)h_1\cdots \varphi(g_n)h_n)
\]
where the $g_i$ and $h_i$ belong to $A$ and $\varphi(g_1)h_1\cdots \varphi(g_n)h_n=1$.
Since $\varphi^2$ is the identity on $A$, this implies that $\Rist_L(0)=\{1\}$ and therefore that $L$ is of infinite index in $G$ as soon as $G$ is weakly branch.
\end{proof}
We now give the first example of a level-transitive weakly maximal subgroup for the particular case of the first Grigorchuk group.
\begin{lemma}\label{Lemma:LevelTransGrigorchuk}
The subgroup
\[
	W_T\coloneqq\gen{a,\diag(\gen{b,ac}\times \gen{b,ac}^a)}=\gen{a,bab,cadab}
\]
is a level-transitive weakly maximal subgroup of the first Grigorchuk group $\Grig$.
\end{lemma}
\begin{proof}
We first show that $W_T$ is an infinite index subgroup that is level-transitive as a consequence of the last lemma.
Hence, it is sufficient to verify the hypotheses of the said lemma.
The conjugation by $a$ is an involution, which directly implies the first hypothesis.
It is enough to verify the second condition on the generators of $\gen{b,ac}$, that is to check that $(b,aba)$ and $(ac,aaca)=(ac,ca)$ are elements of $\Grig$.
We have indeed $acadaba=(b,aba)$ and $baba=(ca,ac)$.
This also gives the equality between the $2$ subgroups of the lemma.
Finally, $\gen{b,ac}$ is a normal subgroup of index $2$ ($J_{0,5}$ in the notation of~\cite{MR1872621}) that contains $K=\gen{(ab)^2,(bd^a)^2,(b^ad)^2}$, the subgroup of index $16$ on which $\Grig$ is branch.
The subgroup $K$ itself is not branch, but contains $K\times K$ and acts transitively on the first level of $T_0$ (by $(ab)^2=(ca,ac)$ for example).
Since $\gen{b,ac}$ acts transitively on the first level of $T$ and contains $K$, it acts level-transitively.
Hence $W_T$ is of infinite index and level-transitive. In particular, it is contained in a weakly maximal subgroup that is level-transitive.

The proof of the weak maximality of $W_T$ is done in the next section, in Lemma~\ref{Lemma:WLWMax}.
\end{proof}
We now give a general construction which, given a branch group $G$, produces a branch group $\check G$ in which $G$ embeds diagonally and that has an infinite index subgroup that is level-transitive.
Moreover, if $G$ is finitely generated, so is $\check G$, which implies that $\check G$ contains a level-transitive weakly maximal subgroup.

Let $T=T_{(m_i)}$ be a locally finite rooted tree.
Let $G$ be any subgroup of $\Aut(T)$ and let $A$ be the quotient $G/\Stab_G(\level{1})$.
Let $m_0'=m_0$ and $m_i'=m_{i-1}$ for $i\geq 1$.
Let $\check G$ be the permutational wreath product $G\wr_{\level{1}}A$, viewed has a subgroup of the automorphism group of $T'=T_{(m'_i)}$.
\begin{proposition}
For every property $P$ in the following list, $\check G$ has $P$ if and only if $G$ has $P$.
\begin{enumerate}
\item Finitely generated,\label{Prop:HatG1}
\item Almost level-transitive, respectively level-transitive,\label{Prop:HatG2}
\item Micro-supported, respectively rigid,\label{Prop:HatG3}
\item The rigid (respectively branch, respectively congruence) kernel is trivial,\label{Prop:HatG4}
\item Branch and just infinite,\label{Prop:HatG5}
\item Polynomial (respectively intermediate, respectively exponential) growth,\label{Prop:HatG6}
\item Is torsion,\label{Prop:HatG7}
\item Is a $p$-group ($p$ prime).\label{Prop:HatG8}
\end{enumerate}
\end{proposition}
\begin{proof}
\textit{Proof of~\ref{Prop:HatG1}.}
On one hand, the rank of $\check G$ is at most the rank of $A$ (a finite group) plus the rank of $G$.
On the other hand, $G$ is the section of the first level stabilizer of $\check G$. In particular, $\rank(G)$ is finite if $\rank(\check G)$ is finite.

\textit{Proof of~\ref{Prop:HatG2}.}
It directly follows from the definition of $A$ and the fact that $\pi_v(\check G)=G$ for any first level vertex that $G$ is (almost) level-transitive if and only if $\check G$ is.

\textit{Proof of~\ref{Prop:HatG3}.}
We have $\Rist_{\check G}(\emptyset)=\Rist_{\check G}(\level{0})=\Stab_{\check G}(\level{0})=\check G$.
On the other hand, for any vertex $v=u_1u_2\dots u_n$ we have
\[
	\Rist_{\check G}(v)=\{1\}\times \dots\times \Rist_G(u_2\dots u_n)\times\dots\times \{1\},
\]
where the non-trivial factor appears in position $u_1$.
Finally, for any $n\geq 1$ we have
\begin{align*}
	\Rist_{\check G}(\level{n})&=\Rist_{G}(\level{n-1})\times \dots\times \Rist_{G}(\level{n-1})\\
	\Stab_{\check G}(\level{n})&=\Stab_{G}(\level{n-1})\times \dots\times \Stab_{G}(\level{n-1}).
\end{align*}
This directly implies the assumption for the micro-supported property.
Since $G\times \dots\times G$ has finite index in $\check G$, $\Rist_{\check G}(\level{n})$ has finite index in $\check G$ if and only if $\Rist_{G}(\level{n-1})$ has finite index in $G$, which finishes the proof for the rigid property.

\textit{Proof of~\ref{Prop:HatG4}.}
Suppose that $G$ has a trivial branch kernel and let $H$ be a finite index subgroup of $\check G$.
Then the index of $H\cap (G\times \{1\}\times\dots\times\{1\})$ in $G\times \{1\}\times\dots\times\{1\}\cong G$ is also finite.
By assumption, there exists $n_1$ such that $\Rist_G(\level{n_1})\leq H\cap (G\times \{1\}\times\dots\times\{1\})$.
Repeating this argument on other coordinates, we obtain $\Rist_G(\level{n_1})\times\dots\times\Rist_G(\level{n_{m_0}})\leq H$ and therefore $\Rist_{\check G}(\level{\max(n_j)+1})\leq H$.
A similar  argument shows that if $G$ has a trivial rigid kernel, so does $\check G$.
On the other hand, suppose that $\check G$ has trivial branch kernel and let $H$ be a finite index subgroup of $G$.
Then $H\times G\times\dots\times G$ is a finite index subgroup of $\check G$ and hence contains $\Rist_{\check G}(\level{n})=\Rist_G(\level{n-1})\times\dots\times \Rist_G(\level{n-1})$ for some $n$. In particular $H$ contains $\Rist_G(\level{n-1})$, and we have proved the triviality of the branch kernel for $G$.
As before, a similar argument takes care of the rigid kernel.
Finally, the triviality of the congruence kernel is equivalent to the simultaneous triviality of both the branch and the rigid kernels.

\textit{Proof of~\ref{Prop:HatG5}.}
By~\cite{MR1765119}, a branch group acting on a locally finite tree is just infinite if and only if for every vertex $v$, the derived subgroup $\Rist_G(v)'$ has finite index in $\Rist_G(v)$.
The description of rigid stabilizers of $\check G$ directly implies that if $\check G$ is just infinite, so is $G$.
For the other direction, we only need to check that the derived subgroup of $\check G$ has finite index in $\check G$.
But the derived subgroup of $\check G$ contains $\Rist_{\check G}(\level{1})'=G'\times\dots\times G'$, which has by assumption finite index in $\Rist_{\check G}(\level{1})$ and hence finite index in $\check G$.

\textit{Proof of~\ref{Prop:HatG6}.}
The group $\check G$ is the semi-direct product of $A$ (a finite group) and of $G^{m_0}$ (a finite product of copies of $G$).
In particular, $G$ and $\check G$ have the same type (polynomial, intermediate or exponential) of growth rate.

\textit{Proof of~\ref{Prop:HatG7} and~\ref{Prop:HatG8}.}
Since $G$ embeds into $\check G$, if $\check G$ is torsion or a $p$-group so is $G$.
On the other hand, every element $g$ of $\check G$ is of the form $(g_1,\dots,g_{m_0})a$ for some $a\in A$ and $g_i$'s in $G$.
In particular, $g^{\abs A}=(h_1,\dots,h_{m_0})$ for some $h_i$'s in $G$.
If all the $h_i$'s have finite order, denoted $o_i$, then the order of $g$ divides $\abs A\cdot\lcm(o_i)$. If $G$ is a $p$-group, so is $A$ and hence $\abs A$ as well as all the $o_i$ are powers of $p$, which finishes the proof.
\end{proof}
\begin{lemma}
If $T$ is $d$-regular and $G$ is self-similar, so is $\check G$.
\end{lemma}
\begin{proof}
Suppose that $G$ is self-similar and let $g=(g_1,\dots,g_{d})a$ be an element of~$\check G$.
Since $g_1$ is in $G$, we have $g_1=(h_1,\dots,h_{d})b$ with $b$ in $A$ and the $h_i$ in $G$ by self-similarity.
Therefore, $g_1$ belongs to $\check G$ and the same is true for the other $g_i$.
\end{proof}
\begin{question}
Suppose that $\check G$ is self-similar. Does this implies that $G$ is also self-similar?

What can be said about the self-replicacity of $G$ and $\check G$?
\end{question}
The above results shows that the groups $\check G$ and $G$ look alike for a lot of purposes.
The main difference is that the construction of $\check G$ allows to easily find a level-transitive subgroup of infinite index.
\begin{lemma}
If $G$ is weakly branch, then $\gen{A,\diag(G\times\dots\times G)}$ is an infinite index level-transitive subgroup of $\check G$.
\end{lemma}
\begin{proof}
The proof is similar to the proof of Lemma~\ref{Lemma:DiagonalTransitive}, with the automorphism $\varphi$ being the identity.
\end{proof}
%
%
%
%
%
%
%
%
%
%
%
%
%
%
\section{Weakly maximal subgroups of the first Grigorchuk group}
\label{Section:GrigorchukGroup}
We now take a special look at what is probably the most studied branch group: the first Grigorchuk group $\Grig$.

This group is branch over $K\coloneqq\gen{(ab)^2}^\Grig$ and we have $K<_2 B\coloneqq\gen{b}^\Grig$, where both $K$ and $B$ are normal in $\Grig$.
There are exactly $10$ subgroups of $\Grig$ containing~$B$,~\cite{LeemannThese}: the group $\Grig$ itself, 3 subgroups of index $2$, $5$ subgroups of index $4$ (on which $J_{1,5}=\gen{B,(ad)^2}$ is the only normal subgroup) and $B$ itself, see Figure~\ref{WMC:Figure:Subgroups} and Table~\ref{WMC:TableB}.

Finally, let us recall that $H=\Stab_\Grig(\level{1})=\gen{b,c,d,aba,aca,ada}$ is the stabilizer of the first level and that rigid stabilizers have the following description
\[
	\Rist_{\Grig}(\level{1})=B\times B\qquad\textnormal{and for }n\geq2: \Rist_{\Grig}(\level{n})=\underbrace{K\times K\times\dots\times K}_{2^n \textnormal{ factors}}.
\]
\begin{figure}[htb]
\centering
\includegraphics{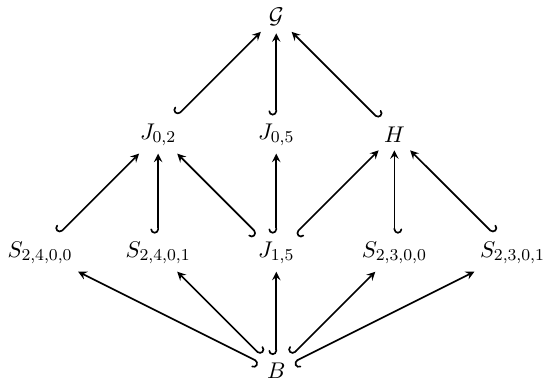}
\caption{The lattice of subgroups of $\Grig$ containing $B$. Each subgroup has index $2$ in the one above it.}
\label{WMC:Figure:Subgroups}
\end{figure}
\begin{table}[htb]
\[
\begin{array}{|c|c|c|}
\hline
\textnormal{Subgroup} & \textnormal{Generators}&\textnormal{Conjugates}\\
\hline
\hline
\Grig & a,b,c&\Grig\\
\hline
J_{0,2} &b,a, a^c &J_{0,2} \\
\hline
J_{0,5} &b,ac &J_{0,5} \\
\hline
H &c,c^a,d,d^a &H \\
\hline
J_{1,5}& b,b^a,dd^a&J_{1,5}\\
\hline
S_{2,3,0,0}& b,c,b^a,b^{aca},c^{aca}&S_{2,3,0,0},(S_{2,3,0,0})^a=S_{2,3,0,1}\\
\hline
S_{2,4,0,0}& a,b,b^{da},a^{dad}&S_{2,4,0,0},(S_{2,4,0,0})^d=S_{2,4,0,1}\\
\hline
B &b,b^a, b^{ada} & B\\
\hline
\end{array}
\]
\caption{The $10$ subgroups of $\Grig$ containing $B$ and their generators; $6$ of them are normal.}
\label{WMC:TableB}
\end{table}
%
%
%
%
%
%
%
%
%
%
%
%
%
%
%
%
%
%
%
%
%
%
%
%
%
%
%
\subsection{More examples of weakly maximal subgroups of \texorpdfstring{$\Grig$}{the first Grigorchuk group}}
\label{Subsection:ExemplesOfWeaklyMaximal}
In this subsection we investigate two explicit examples of weakly maximal subgroups that have a block structure.
The first example, due to Pervova and firstly mentioned in~\cite{MR2893544}, is
\[
	W_P\coloneqq\gen{a,\diag( J_{1,5}\times  J_{1,5}),\{1\}\times K\times\{1\}\times K}.
\]

The second example is the subgroup
\begin{align*}
	W_T	&=\gen{a,\diag(J_{0,5}\times J_{0,5}^a)}=\gen{a,bab,cadab}=\gen{a,bab,cac}
\end{align*}
of Lemma~\ref{Lemma:LevelTransGrigorchuk}.
\paragraph{The subgroup $W_P$}
The first thing we need to do is to show that the group $W_P$ is a subgroup of $\Grig$.
The element $a$ belongs to $\Grig$ and $\{1\}\times K\times\{1\}\times K$ is a subgroup of $\Grig$.
Therefore, it remains to check that $\diag(J_{1,5}\times  J_{1,5})$ is also a subgroup of $\Grig$.
But we have $J_{1,5}=\gen{b,aba,dada}$ and an easy verification gives us $dd^a=(b,b)$, $d^{ac}d^{aca}=(aba,aba)$ and $(ac)^4=(dada,adad)=(dada,dada)$.

The non-rigidity tree $\NR(W_P)$ of $W_P$ is the subtree $S$ of $T$ generated by $T_{00}$ and $T_{10}$ and $W_p$ acts level-transitively on it.
Indeed, since $W_P$ contains $\{1\}\times K\times\{1\}\times K$ we have $\NR(W)\subseteq S$.
On the other hand, $W_P$ contains $a$, which sends $T_{00}$ to $T_{10}$, and $\diag( J_{1,5}\times  J_{1,5})$.
Since $J_{1,5}=\gen{b,aba,dada}$, its left section contains $\gen{a,c,b}=\Grig$.
Therefore, $\diag( J_{1,5}\times  J_{1,5})$ acts level-transitively on $T_{00}$ and $W_P$ acts level-transitively on $S$ which implies $\NR(W)=S$.

We will now prove several results in order to better understand the structure of $W_p$ and of $\overline{W}_p$, its closure in $\hat\Grig$ the profinite completion of $\Grig$.
For the weak maximality of $W_P$, we will mainly give the proof of~\cite{MR2893544} but with more details.
\begin{lemma}[\cite{MR2893544}]\label{WMC:Lemma:index}
Let $A$ be a subgroup of $\Grig$ containing $J_{1,5}$.
Let $1\neq x\in\Stab_\Grig(1)$. Then $\gen{x}^{A}$ has infinite index in $\Grig$ if and only if $\treesection{x}{i}=1$ for some $i\in\{0,1\}$.
\end{lemma}
\begin{proof}
If $\treesection{x}{0}=1$, then $\gen{x}^{A}\leq \{1\}\times \Grig$ is of infinite index by Fact~\ref{Fact:NumberOrbitsIndex}. The case $\treesection{x}{1}=1$ is similar.

On the other hand, suppose that $x=(x_0,x_1)$ with $x_i\neq 1$ for $i\in\{0,1\}$.
In this case, the centralizer $C_\Grig(x_i)$ has infinite index in $\Grig$ by Lemma~\ref{Lemma:InfiniteIndexCentralizer}.
This implies that for $i\in\{0,1\}$ there exists $y_i\in K\setminus C_\Grig(x_i)$.
We then have $1\neq[x_i,y_i]$ belongs to $K$. 
Since $A$ contains $J_{1,5}$, it contains $K\times K$ and $(1,[x_1,y_1])=[x,(1,y_1)]=x\cdot (1,y_1)x^{-1}(1,y_1)^{-1}$ is in $\gen{x}^A$.
On the other hand, $\pi_i(A)\geq\pi_i(J_{1,5})=\Grig$.
Altogether, we have $\gen{x}^{A}\geq \gen{[x_0,y_0]}^\Grig\times \gen{[x_1,y_1]}^\Grig$.
Both $\gen{[x_i,y_i]}^\Grig$ are non-trivial normal subgroups of $\Grig$ and therefore of finite index since $\Grig$ is just infinite.
This shows that $\gen{x}^{A}$ itself is of finite index in $\Grig$.
\end{proof}
\begin{lemma}[\cite{MR2893544}]\label{LemmaWPWMax}
$W_P$ is a finitely generated weakly maximal subgroup of~$\Grig$.
\end{lemma}
\begin{proof}
The subgroups $K$ and $J_{1,5}$ being finitely generated, so is $W_P$.
Now, if $g$ is an element of $W_P\cap (K\times\{1\}\times\{1\}\times\{1\})$, we have $g=(k,1,1,1)$ and also $g=(g_0,g_1,g_0,g_3)$ which implies $g=1$.
We have shown that $W_P\cap (K\times\{1\}\times\{1\}\times\{1\})=\{1\}$ and therefore that $W_P$ is of infinite index.

We now want to prove that $W_P$ is weakly maximal.
That is, for all $x\in\Grig\setminus W_P$, the subgroup $\widetilde{W}\coloneqq \gen{W_P,x}$ is of finite index in $\Grig$.
Since $a$ belongs to $W_P$, we can assume that $x$ belongs to $H$ and $x=(x_0,x_1)$.
We have $\Grig/B=\{1,a,d,ad,ada,\dots,(ad)^3a\}\cong D_{2\cdot 4}$, the dihedral group of order $8$, and hence $\Grig/ J_{1,5}=\{1,a,d,ad\}$.
By factorizing the first coordinate by $\pi_0(W_P)\geq  J_{1,5}$ we can assume that $x_0$ belongs to $\{1,a,d,ad\}$ which leave us with four cases to check.
If $x_0$ is not in $H$, then $\widetilde{W}$ contains $(\{1\}\times K\times \{1\}\times \{1\})^{x_0}=K\times \{1\}\times \{1\}\times \{1\}$ and $a$.
In this case, $\widetilde{W}$ contains $K\times K\times K\times K$ and is of finite index.
We can hence suppose that $x_0$ is in $H$, and by symmetry, that $x_1$ is  also in $H$.
It thus remains to check two cases: $(1,x_1)$ and $(d,x_1)$ with $x_1$ in $H$.

If $x_0=1$, then $x_1\neq 1$.
In this case, $\widetilde{W}$ contains both $\{1\}\times\gen{x_1}^{J_{1,5}}$ and $a$, and hence also $\gen{x_1}^{J_{1,5}}\times\gen{x_1}^{J_{1,5}}$.
If the subgroup $\gen{x_1}^{J_{1,5}}$ has finite index in $\Grig$, then $\widetilde{W}$ has also finite index in $\Grig$.
We can therefore assume that $\gen{x_1}^{J_{1,5}}$ has infinite index in $\Grig$, which implies by Lemma~\ref{WMC:Lemma:index} that $x_1=(1,z)$ or $x_1=(z,1)$, $z\neq 1$.
In both cases, we have $x=(1,x_1)$ is an element of $\Rist_\Grig(2)=K\times K\times K\times K$ which implies that $z$ is in $K$.
This rules out the case $x_1=(1,z)$, since in this case we would have $x=(1,1,1,z)\in \{1\}\times K\times \{1\}\times K\leq W$.
On the other hand, if $x_1=(z,1)$, the subgroup $\widetilde W$ contains $\gen{(1,1,z,1)}^{\diag(J_{1,5}\times J_{1,5})}=\{1\}\times\{1\}\times\gen{z}^\Grig\times\{1\}$ since $\pi_0(J_{1,5})=\Grig$.
The group $\Grig$ being just infinite and $z\neq 1$, the subgroup $\gen{z}^\Grig$ has finite index in $\Grig$ and $\gen{z}^\Grig\cap K$ has finite index in $K$.
Therefore, $\tilde W$ contains $(\gen{z}^\Grig\cap K)\times K\times(\gen{z}^\Grig\cap K)\times K$ which has finite index in $K\times K\times K\times K=\Rist_\Grig(\level{2})$ and thus in $\Grig$.

We will now show that the case $x_0=d$ cannot happen if $x_i$ is in $H$.
Indeed, $(d,x_1)$ belongs to $\Grig$ if and only if $(1,ax_1)=c^a\cdot(d,x_1)$ is in $\Grig$.
But in this case, $(1,ax_1)$ belongs to $\Rist_\Grig(1)=B\times B$ and so $ax_1$ is in $B\leq H$, which is impossible if $x_1\in H$.
\end{proof}
\begin{lemma}\label{WMC:Lemma:WcapH}
The subgroup $W_P\cap H$ is a weakly maximal subgroup of $H$.
\end{lemma}
\begin{proof}
Let $x$ be in $H\setminus W_P$ and look at $\widetilde W\coloneqq\gen{x,W_P\cap H}$.
Then $x=(x_0,x_1)$ and factorizing the first factor by $J_{1,5}$, we can assume that $x_0$ is either $1$ or $d$ and $x_1$ is in $H$.
If $x_0=1$, then $x_1\neq 1$.
In this case, $\widetilde{W}$ contains both $\{1\}\times\gen{x_1}^{J_{1,5}}$ and $\diag(J_{1,5}\times J_{1,5})$.
In particular, it contains $C\times C$ where $C=\gen{x_1}^{J_{1,5}}\cap J_{1,5}$.
Observe that $C$ is of finite index in $\Grig$ if and only if $\gen{x_1}^{J_{1,5}}$ is so.
The rest of the proof follows the proof of the weak maximality of $W$.
\end{proof}
Let $G$ be a topological group. A subgroup $W$ of $G$ is said to be \defi{weakly maximal closed} if it is maximal among all closed subgroups of infinite index.
\begin{lemma}
\label{GrigSubgroup:Conjecture:BarWIsOk}\hspace{1em}
\begin{enumerate}
\item
$\overline W_P$ is a topologically finitely generated weakly maximal closed subgroup of $\hat \Grig$;\label{WMC:ConjBarW1}
\item
There are uncountably many distinct conjugates of $\overline W_P$ in $\hat \Grig$.\label{WMC:ConjBarW2}
\end{enumerate}
\end{lemma}
\begin{proof}
Let $1\neq x\in\Stab_{\hat\Grig}(1)$.
We claim that $\overline{\gen{x}^{J_{1,5}}}$ has infinite index in $\hat\Grig$ if and only if $\treesection{x}{i}=1$ for some $i\in\{0,1\}$.
The proof of this claim is the same as the proof of Lemma~\ref{WMC:Lemma:index} and we only need to show that for every $1\neq x\in\hat\Grig$, there exists $y\in\overline{K}$ such that $[x,y]\neq 1$.
But this follows from the fact that if $1\neq x$ is any element of $\hat \Grig$ then $[\Grig: C_\Grig(x)]=\infty$, which is Lemma~\ref{Lemma:InfiniteIndexCentralizer}.
Indeed, $[\hat\Grig: C_{\hat\Grig}(x)]=\infty$ if and only if $[\Grig: C_{\hat\Grig}(x)\cap\Grig]=[\Grig: C_\Grig(x)]=\infty$.
In this case, the finite index subgroup $\overline{K}$ cannot be contained in $C_{\hat\Grig}(x)$.

The application $\overline{\phantom{H}}\colon\Subcl(G)\to\Subcl(\hat\Grig)$ that sends a closed (in the profinite topology) subgroup of $\Grig$ to its closure in $\hat\Grig$ sends infinite index subgroups to infinite index subgroups. See~\cite{LeemannThese} for more details.
Since $W_P$ is weakly maximal in $\Grig$, it is closed and we have that $\overline W_P$ is an infinite index closed subgroup of $\hat\Grig$.
Indeed, by~\cite{MR1841763} all maximal subgroups of $\Grig$ are of finite index. The group $\Grig$ being a super strongly self-replicating just infinite group with the congruence subgroup property, the property of having all maximal subgroups of finite index is preserved when passing to finite index subgroup by~\cite[Lemma 4]{MR2009443}. Hence $W=\bigcap_{i\geq 0}M_i$ where $M_0=\Grig$ and $M_{i+1}$ is a maximal, and therefore finite index, subgroup of $M_i$ containing $W$.
The element $a$ normalizes both $\diag( J_{1,5}\times J_{1,5})$ and $\{1\}\times K\times\{1\}\times K$ and since $K$ is normal and $J_{1,5}\leq H$, the subgroup $\{1\}\times K\times\{1\}\times K$ is normalized by $\diag( J_{1,5}\times J_{1,5})$
Therefore, $W_P=\{1,a\}\cdot\diag( J_{1,5}\times J_{1,5})\cdot(\{1\}\times K\times\{1\}\times K)$ and $\overline W_P=\overline{\{1,a\}}\cdot\overline{\diag( J_{1,5}\times J_{1,5})}\cdot\overline{\{1\}\times K\times\{1\}\times K}$.
The subgroup $\{1,a\}$ is closed and we have $\overline{\{1\}\times K\times\{1\}\times K}=\{1\}\times \overline{K}\times\{1\}\times \overline{K}$ and $\overline{\diag( J_{1,5}\times J_{1,5})}=\diag(\overline{J_{1,5}}\times\overline{J_{1,5}})$.
Altogether, we have 
\[
	\overline W_P=\gen{a,\diag(\overline{J_{1,5}}\times\overline{J_{1,5}}),\{1\}\times \overline{K}\times\{1\}\times \overline{K}}
\]
is a (topologically) finitely generated subgroup of $\hat\Grig$.

For all $n$ we have $\Stab_\Grig(n)=\Stab_{\hat\Grig}(n)\cap\Grig$ and $\Rist_\Grig(n)=\Rist_{\hat\Grig}(n)\cap\Grig$.
Since these subgroups are of finite index, we have 
\[
	\overline{\Stab_\Grig(n)}=\Stab_{\hat\Grig}(n)\qquad\overline{\Rist_\Grig(n)}=\Rist_{\hat\Grig}(n)
\]
In particular, we have
\[
	\Rist_{\hat\Grig}(\level{1})=\overline{B}\times\overline{B}\qquad \Rist_{\hat\Grig}(\level{2})=\overline{K}\times\overline{K}\times\overline{K}\times\overline{K}
\]
On the other hand, the closure preserves transversals for finite index subgroups.
In particular, $\hat\Grig/\overline{H}=\{1,a\}$ and $\hat\Grig/\overline{J_{1,5}}=\{1,a,ad,d\}$.
Let $x$ be an element from $\Grig\setminus \overline W_P$  and look at $\widetilde{\overline W}\coloneqq\overline{\gen{\overline W_P,x}}$.
The proof that $\widetilde{\overline{W}}$ is of finite index is the same as the one for $\widetilde{W}$, where Lemma~\ref{WMC:Lemma:index} is replaced by the claim at the beginning of this proof.

Since $\overline{W}_P$ is weakly maximal closed, we have $g\overline{W}_Pg^{-1}\neq \overline{W}_P$ for every torsion element in $\hat\Grig\setminus\overline{W}_P$, adapting the proof of \cite[Proposition 2.3]{MR3478865}.
Using that $\Grig$ is torsion and that $\overline{W}_P\cap\Grig=W_P$ we obtain that the normalizer of $\overline{W}_P$ has infinite index in $\hat\Grig$.
On the other hand, $\hat \Grig$ is topologically just infinite and hence contains no subgroup with countable index,~\cite{MR3592603}.
In particular, $\overline{W}_P$ has uncountably many distinct conjugates.
\end{proof}
Using more convoluted methods, it is showed in~\cite{LeemannThese} that $\overline W_P$ has a continuum of conjugates in $\hat \Grig$.
\paragraph{The subgroup $W_T$}
We have seen in the last section that $W_T=\gen{a,bab,cac}$ is an infinite index subgroup of $G$ that is finitely generated and acts level-transitively.
It remains to show that it is weakly maximal in order to finish the proof of Lemma~\ref{Lemma:LevelTransGrigorchuk}.
\begin{lemma}\label{Lemma:WLWMax}
The subgroup $W_T$ is weakly maximal.
\end{lemma}
\begin{proof}
Let $x$ be an element of $G\setminus W_T$ and $\widetilde W_T\coloneqq\gen{W_T,x}$.
Since $a$ belongs to $W_T$, we may suppose that $x=(x_0,x_1)$ is in $H$.
On the other hand, $\Grig=J_{0,5}\sqcup a J_{0,5}$ and $W_T$ contains $\diag(J_{0,5}\times J_{0,5}^a)$ and thus we may suppose that $x=(x_0,1)$ or $x=(x_0,a)$.

Firstly suppose that $x=(x_0,1)$.
In particular, $x$ belongs to $\Rist_\Grig(0)=B\times \{1\}$ and therefore $x_0$ is in $B$ and thus in $H$ and $x_0=(z,t)$.
If both $z$ and $t$ are not trivial, then by Lemma~\ref{WMC:Lemma:index}, the subgroup $\gen{x_0}^{J_{0,5}}$ has finite index in $\Grig$.
But $\widetilde W_T$ contains $\gen{x_0}^{J_{0,5}}\times \{1\}$ and $a$, and thus it contains $\gen{x_0}^{J_{0,5}}\times \gen{x_0}^{J_{0,5}}$, which is a finite index subgroup of $\Grig$ and we are done.
On the other hand, suppose that at least one of $z$ or $t$ is trivial, say $t$. Then $x=(z,1,1,1)$ with $z\neq 1$.
In this case, $\widetilde W_T$ contains 
\begin{align*}
\gen{x}^{\diag(J_{0,5}\times J_{0,5}^a)}&=\gen{(z,1)}^{J_{0,5}}\times\{1\}\\
&=\gen{z}^\Grig\times \{1\}\times\{1\}\times\{1\}
\end{align*}
Once again, $\Grig$ being just infinite and $z$ non-trivial implies that $\gen{z}^\Grig$ is a finite index subgroup $N$ of $\Grig$ and since $\widetilde W_T$ acts transitively on the second level it contains $N\times N\times N\times N$ which is a finite index subgroup of $\Grig$.

Let us now look at the case $x=(x_0,a)\notin W_T$.
Observe that $\pi_1(\widetilde W_T)=\gen{\pi_1(W_T),\treesection{x}{1}}=\gen{J_{0,5},a}=\Grig$.
Since $x$ belongs to $\Grig$, so does $xaba=(x_0c,1)$ which implies that $x_0=yc$ with $y\in B$.
Both $x^{-1}=(cy^{-1},a)$, and $(y,aya)\in\diag(J_{0,5}\times J_{0,5}^a)$ belong to $\widetilde W_T$.
In particular, $\widetilde W_T$ contains $\bigl(x^{-1}\cdot(y,aya)\bigr)^2=(1,yaya)$ and therefore $\widetilde W_T$ contains $\{1\}\times\gen{yaya}^\Grig$.
If $yaya$ is not trivial, then the same argument as before shows that $\widetilde W_T$ is of finite index.
Since $y$ is in $B$, so is $yb$ and therefore $(yb,ayba)$ belongs to $W_T$.
In particular, $\bigl(x^{-1}\cdot(yb,ayba)\bigr)^2=(1,ybayba)$ is in $\widetilde W_T$.
As before, if $ybayba$ is not trivial, $\widetilde W_T$ is of finite index.
All we have to do is to show that $yaya$ and $ybayba$ cannot be both trivial.
Since $y$ is in $B$ it is equal to $(t_1,t_2)$.
Then $yaya=(t_1t_2,t_2t_1)$ and $ybayba=(t_1at_2c,t_2ct_1a)$.
If they are both trivial, then $t_1=t_2^{-1}$ and $t_2^{-1}at_2c=1$.
But $t_2^{-1}at_2c$ is never in $H$ regardless of the value of $t_2$ and we have the desired contradiction.
\end{proof}
%
%
%
%
%
%
%
%
%
%
%
%
%
%
%
%
\subsection{Sections of weakly maximal subgroups of \texorpdfstring{$\Grig$}{the first Grigorchuk group}}
\label{WMC:SecClassification}
While general results about sections of weakly maximal subgroups are given in Section~\ref{Section:Sections}, we give here more details for the particular case of the first Grigorchuk group.

Let $W<\Grig$ be a weakly maximal subgroup.
If $W$ is contained in the stabilizer of the first level, then by Lemma~\ref{WMC:lemmaSections}, if one of the first-level sections is of infinite index, say $\pi_0(W)$, then $\pi_0(W)$ is a weakly maximal subgroup of $\Grig$ and the other section $\pi_1(W)$ contains $B$.
It is natural to ask if this result can be extended to weakly maximal subgroups that do not stabilize the first level.
In fact, the same proof shows that if $W\cap H$ is weakly maximal in $H$ and one of the first-level sections of $W\cap H$ is of infinite index, then it is a weakly maximal subgroup of $\Grig$ and the other section contains $B$.
This remark and Lemma~\ref{Lemma:ProjGenParab} imply the following.
\begin{corollary}\label{Cor:NotWMInH}
If $W$ is a generalized parabolic subgroup of $\Grig$ not contained in~$H$, then $W\cap H$ is not weakly maximal in $H$.
\end{corollary}
Corollary~\ref{Cor:NotWMInH} can't be extended to all weakly maximal subgroups of $\Grig$, as shown by the counter-example of $W_P$, see Lemmas~\ref{LemmaWPWMax} and~\ref{WMC:Lemma:WcapH}.

Let $\xi\in\partial T$ be any ray.
Then $W=\SStab_\Grig(\{\xi,a.\xi\})$ satisfies the hypothesis of the above corollary which shows the existence of weakly maximal subgroup of $\Grig$ such that $W\cap H$ is not a weakly maximal subgroup of $H$.
This answers positively question 6.5.6 of~\cite{LeemannThese}.

By Proposition~\ref{Prop:NewWmax}, for every weakly maximal subgroup $L$ of $\Grig$, there exists a weakly maximal subgroup $W$ of $\Grig$ that stabilizes the first level and such that $\pi_1(W)=L$.
In particular, by taking $L$ any weakly maximal subgroup that is not (generalized) parabolic, we obtain infinitely many examples of weakly maximal subgroups that are not (generalized) parabolic, but still contained in $H$, hence answering question 6.5.3 of~\cite{LeemannThese}.

In~\cite{LeemannThese}, the author also asked the following question about sections of weakly maximal subgroups.
\begin{question}\label{Question:ProjWM}
Let $W$ be a weakly maximal subgroup of $\Grig$ contained in $H$ and $\pi_0(W)$ and $\pi_1(W)$ its left and right sections.
If $\pi_1(W)$ is of infinite index, which of the $10$ subgroups of $\Grig$ containing $B$ could appear as $\pi_0(W)$?
\end{question}
We will show that both $\Grig$ (Lemma~\ref{Lemma:SectionG}) and $J_{0,2}$ (Proposition~\ref{Proposition:SectionsParabolic}) can be obtained as $\pi_0(W)$.
\begin{lemma}\label{Lemma:SectionG}
There exists a continuum of generalized parabolic subgroup $W$ of $\Grig$ such that $W$ is contained in $H$, $\pi_1(W)$ is of infinite index in $\Grig$ and $\pi_0(W)=\Grig$.
\end{lemma}
\begin{proof}
Let $F=\gen{c,a}$.
This is a finite subgroup of $\Grig$ and hence contained in a continuum of pairwise distinct generalized parabolic subgroup $W$, Corollary~\ref{Cor:NumberGeneralizedParabolic}.
For each of these $W$, let $W^1=\setst{g\in\Stab_G(1)}{\treesection{g}{1}\in W}$.
By Proposition~\ref{Prop:NewWmax}, the $W^1$ are weakly maximal subgroup of $\Grig$ contained in $\Stab_\Grig(1)=H$, with $\pi_1(W^1)=W$, and they are pairwise distinct.
Moreover, since $W$ is generalized parabolic, so is $W^1$.
Finally, since $W$ contains $c$ and $a$, both $b=(a,c)$ and $aba=(c,a)$ are in $W^1$.
Thus, $\pi_0(W^1)$ contains $\gen{B,a,c}=\Grig$.
\end{proof}
%
%
%
%
%
%
%
%
%
%
%
%
%
%
%
%
%
%
%
%
%
%
\subsection{Generalized parabolic subgroups  of \texorpdfstring{$\Grig$}{the first Grigorchuk group}}
In this subsection, we show that parabolic subgroups of $\Grig$ behave nicely under the closure in the $\Aut(T)$ topology and give a full description  of their sections.

In order to study parabolic subgroups, it is of great help to have a precise description of vertex stabilizer and of one particular parabolic subgroup.
This was done in~\cite{MR1899368}, where $\Stab_\Grig(1^n)$ is described as an iterated semi-direct product, as depicted in Figure~\ref{WMC:Stab1n}.
\begin{figure}[htbp]
\centering
\includegraphics{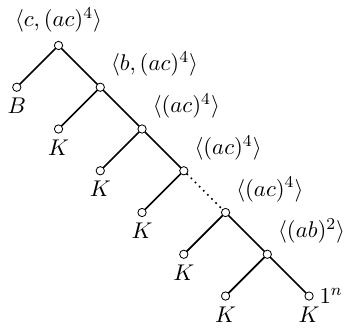}
\caption{The stabilizer in $\Grig$ of the vertex $1^n$.}
\label{WMC:Stab1n}
\end{figure}
In particular, for $P\coloneqq\Stab_\Grig(1^\infty)$ we have
\[
	P=\Bigl(B\times\bigl((K\times((K\times\dots)\rtimes\gen{(ac)^4})\bigr) \rtimes\gen{b,(ac)^4}\bigr)\Bigr)\rtimes\gen{c,(ac)^4}.
\]
As already said, $\Grig$ can be endowed with the congruence topology coming from $\Aut(T)$, which is the same as the profinite topology since $\Grig$ has trivial congruence kernel.
Recall that for a subgroup $G$ of $\Aut(T)$, we denote by $\overline{G}$ its closure in $\Aut(T)$.
Then we naturally have $\overline{\Stab_G(\xi)}\leq\Stab_{\overline{G}}(\xi)$ for any ray in $\partial T$.
The first Grigorchuk group satisfies the interesting property that this inequality is in fact always an equality:
\begin{proposition}\label{Proposition:ClosureOfStabs}
Let $\Grig\curvearrowright T$ be the branch action of the first Grigorchuk group.
For any finite subset $C$ of $\partial T$, we have
\[
	\overline{\Stab_\Grig(C)}=\Stab_{\overline{\Grig}}(C).
\]
If moreover $C$ is contained in one $\Grig$-orbit, then we also have
\[
	\overline{\SStab_\Grig(C)}=\SStab_{\overline{\Grig}}(C).
\]
\end{proposition}
Before giving the proof of this proposition, let us take a look at some of its consequences.
For $C$ a closed and nowhere dense subset of $\partial T$, we know that $\SStab_{\Grig}(C)$ is weakly maximal if and only if it acts minimally on $C$.
An obvious necessary condition for the action $\SStab_{\Grig}(C)\curvearrowright C$ to be minimal is that both the action  $\SStab_{\overline{\Grig}}(C)\curvearrowright C$ and $\Grig\curvearrowright C$ are minimal.
Proposition~\ref{Proposition:ClosureOfStabs} implies that for finite $C$, this is also a sufficient condition:
\begin{corollary}
Let $C$ be a finite subset of $\partial T$.
Then $\SStab_{\Grig}(C)$ is a generalized parabolic subgroup if and only if $C$ is contained in one ${\Grig}$-orbit and in one $\SStab_{\overline{\Grig}}(C)$-orbit.
\end{corollary}
The first Grigorchuk group, as well as other groups with all maximal subgroups of finite index, has the nice property that every weakly maximal subgroup is closed in the profinite topology.
In~\cite{LeemannThese}, the author studied closed subgroups of finitely generated branch groups, and more particularly the maps between the set of closed subgroups of $G$ and of the set of closed subgroups of its profinite completion.
\begin{align*}
	\Theta\colon\Subcl(G)&\to \Subcl(\hat G)&\Psi\colon\Subcl(\hat G)&\to \Subcl(G)\\
	 H&\mapsto\overline{H}					&	M&\mapsto M\cap G
\end{align*}
When restricted to finite index subgroups, these maps are lattice-isomorphisms (that also preserve normality and the index) and $\Theta$ is the inverse of $\Psi$.
The author thus asked if the image by one of these maps of a weakly maximal subgroup was still a weakly maximal subgroup.
This is not the case as shown by the following corollary of Propositions~\ref{Proposition:ClosureOfStabs} and~\ref{Proposition:WMStrongRigidSStab}.
\begin{corollary}
Let $\Theta$ and $\Psi$ be the above maps for $G=\Grig$ the first Grigorchuk group.
Let $\mathcal C$ be the set of finite subsets $C$ of $\partial T$ such that $\SStab_{\Grig}(C)$ acts minimally (i.e. transitively) on $C$.
When restricted to generalized parabolic subgroups of the form $\SStab(C)$ with $C$ in $\mathcal C$, the maps $\Theta$ and $\Psi$ are bijective and $\Theta$ is the inverse of $\Psi$.

On the other hand, for all $C=\{\xi,\eta\}$ in $\partial T$ of cardinality $2$, the subgroup $W_{\xi,\eta}\coloneqq\SStab_{\hat\Grig}(\{\xi,\eta\})$ is weakly maximal and closed, while for a fixed $\xi$, the subgroup $\Psi(W_{\xi,\eta})=\SStab_{\Grig}(\{\xi,\eta\})$ is weakly maximal only for countably many~$\eta$.
\end{corollary}
Let $T_{\lfloor 3\rfloor}$ be the finite subtree of $T$ consisting of all vertices of level at most~$3$.
A labelling of vertices of $T_{\lfloor 3\rfloor}$ by elements of $\Sym(2)$ is called an \defi{allowed pattern} if it occurs as the top of the portrait $\portrait(g)$ of some $g$ in $\Grig$.
Otherwise it is a \defi{forbidden pattern}.
A portrait $\portrait(g)$ of an element $g\in \Aut(T)$ is said to contain a forbidden pattern if there is a labelled subtree in it that is a forbidden pattern.
Observe that looking at labellings of vertices of $T_{\lfloor 3\rfloor}$ is equivalent to looking at labellings of inner vertices of $T_{\lfloor 4\rfloor}$, that is at automorphisms of $T_{\lfloor 4\rfloor}$.
The following result about portrait will be one of the two ingredients in the proof of Proposition~\ref{Proposition:ClosureOfStabs}.
\begin{proposition}[{\cite[Proposition 7.2, Corollary 7.2]{MR2195454}}]\label{Proposition:Portrait}
Let $T$ be the $2$-regular rooted tree and $g$ an element of $\Aut(T)$.
Then
\begin{enumerate}
\item
$g$ is in $\overline{\Grig}=\hat \Grig$ if and only if $\portrait(g)$ contains no forbidden patterns,
\item
$g$ is in $\Grig$ if and only if $\portrait(g)$ contains no forbidden patterns and there is a transversal $X$ of $T$ such that for every $v$ in $X$, the portrait of $g$ below $T_v$ is the portrait of one element in $\{1,a,b,c,d\}$.
\end{enumerate}
\end{proposition}
The following technical lemma about stabilizers of vertices of level $4$ is the second main ingredient of the proof of Proposition~\ref{Proposition:ClosureOfStabs}.
\begin{lemma}\label{Lemma:StabilizerinT4}
Let $v$ be a vertex of level $4$.
For any $g\in\Stab_\Grig(v)$, there exists $h$ in $\Stab_\Grig(T_v)$, the pointwise stabilizer of $T_v$, such that $\portrait(g)$ and $\portrait(h)$ coincide on $T_{\lfloor 3\rfloor}$.
\end{lemma}
\begin{proof}
Up to conjugating by an element of $\Grig$, it is enough to prove the lemma for $v=1^4$.
On the other hand, it is also sufficient to prove the lemma for elements of some generating set of $\Stab_\Grig(1^4)$.
Recall that $\Stab_\Grig(1^4)$ is the iterated semi-direct product depicted in Figure~\ref{WMC:Stab1n} for $n=4$.
That is, $\Stab_\Grig(1^4)$ is generated by
\[
	S=B_{@0}\cup K_{@10}\cup K_{@110}\cup K_{@1110}\cup K_{@1111}\cup \{c, (ac)^4, b_{@1},(ac)^4_{@1},(ac)^4_{@11},(ab)^2_{@111}\}
\]
where $B_{@0}\coloneqq \setst{g\in \Rist(0)}{\treesection{g}{0}\in B}$ and similarly for the other elements.
By~\cite[Theorem 7.9]{MR2035113} we have $B_{@0}=\Rist_\Grig(0)$ and $B_{@v}=\Rist_\Grig(v)$ for any $v$ of level at least $2$, and all these are subgroups of $\Grig$.
The element $b$ is in $B$, and hence $b_{@1}$ is in $B_{@1}=\Rist_\Grig(1)$.
Similarly, both $(ac)^4=(b,b,b,b)$  and $(ab)^2=(ca,ac)$ belong to $K=\gen{(ab)^2}^\Grig\leq\Rist_\Grig(v)$.
Hence every element of $S$ is indeed in $\Grig$. 

It now remains to construct a function $\psi\colon S\to\Stab_\Grig(1^4)=\gen{S}$ such that $s$ and $\psi(s)$ agree on $T_{\lfloor 3\rfloor}$ and $\psi(s)$ belongs to $\Stab_\Grig(T_{1^4})$.
If $s$ belongs to $B_{@0}\cup K_{@10}\cup K_{@110}\cup K_{@1110}$, it fixes pointwise $T_{1^4}$ and we can take $\psi(s)=s$.
If $s$ belongs to $K_{@1111}\cup\{(ac)^4_{@1},(ac)^4_{@11},(ab)^2_{@111}\}$ then $\portrait(s)$ is trivial on $T_{\lfloor 3\rfloor}$ and we can take $\psi(s)=1$.
Finally, we define $\psi(c)\coloneqq c(ac)^4_{@1}(ac)^4_{@11}$, $\psi(b_{@1})\coloneqq b_{@1}(ac)^4_{@11}$ and  $\psi\bigl((ac)^4\bigr)\coloneqq(ac)^4(ac)^4_{@1}(ac)^4_{@11}$.
Since the portrait of $(ac)^4_{@1}$ and $(ac)^4_{@11}$ are trivial on $T_{\lfloor 3\rfloor}$ we each time have that the portrait of $s$ and of $\psi(s)$ coincide on $T_{\lfloor 3\rfloor}$ and that $\treesection{\psi(c)}{1^4}=\treesection{c}{1^4}\treesection{(ac)^4_{@1}}{1^4}\treesection{(ac)^4_{@11}}{1^4}$ and similarly for $\psi(b_{@1})$ and $\psi\bigl((ac)^4\bigr)$.
We have
\begin{gather*}
	\treesection{c}{1^4}=d \qquad \treesection{{b_{@1}}}{1^4}=b\qquad \treesection{(ac)^4}{1^4}=d\\
	 \treesection{(ac)^4_{@1}}{1^4}=c \qquad \treesection{{(ac)^4_{@11}}}{1^4}=b
\end{gather*}
and direct computations give us that $\treesection{\psi(c)}{1^4}=\treesection{\psi(b_{@1})}{1^4}=\treesection{\psi\bigl((ac)^4\bigr)}{1^4}=1$ and hence that $\psi(c)$, $\psi(b_{@1})$ and $\psi\bigl((ac)^4\bigr)$ all fixe pointwise $T_{1^4}$ as desired.
\end{proof}
We can now prove Proposition~\ref{Proposition:ClosureOfStabs}.
\begin{proof}[Proof of Proposition~\ref{Proposition:ClosureOfStabs}]
Since $C$ is finite, there exists a level $n_0$ such that every vertex $v$ of level at least $n_0$ belongs to at most one ray in $C$.
For the following, we will always look at levels $n\geq n_0$.

We begin by proving the equality for pointwise stabilizers.
Let $g$ be in $\Stab_{\overline{\Grig}}(C)$ and let $(h_n)_n$ be a sequence of elements of $G$ converging to $g$ such that $h_n$ is at distance at most $2^{-n}$ of $g$.
The condition on the distance is here to ensure that $\portrait(h_n)$ and $\portrait(g)$ coincide on $T_{\lfloor n-1\rfloor}$. In other words, the actions of $h_n$ and $g$ coincide on $T_{\lfloor n\rfloor}$.

We will now modify the $h_n$ to obtain a new sequence $(g_n)_n$ of elements of $\Stab_{\Grig}(C)$ (respectively $\SStab_{\Grig}(C)$) that still converges to $g$.
Let $\{v_1,\dots,v_{\abs C}\}$ be the intersection of $C$ with level $n$ and let $w_i$ be the unique vertex of level $n+4$ that is in $C$ with $v_i\leq w_i$. Then $\{w_1,\dots,w_{\abs C}\}$ is the intersection of $C$ with level $n+4$.
The element $h_{n+4}$ fixes all the $w_i$, and hence also all the $v_i$.
By Lemma~\ref{Lemma:StabilizerinT4}, we can replace the portrait of $h_{n+4}$ below $v_i$ by the portrait of some $f_{v_i}$ fixing pointwise $T_{w_i}$.
The element $g_n$ obtained in this way is in $\Grig$ (Proposition~\ref{Proposition:Portrait}), still at distance at most $2^{-n}$ of $g$ and the restriction of its portrait to the $T_{w_i}$'s is trivial.
Hence, $g_n$ pointwise stabilizes $C$.

We now show that the same strategy takes care of the setwise stabilizer.
Firstly, since $\Grig$ is finitary along rays, the orbits of $\Grig$ are contained in cofinality classes.\footnote{In fact, if $G\leq\Aut(T)$ is both finitary along rays and spherically transitive, then it's orbits are exactly the cofinality classes.}
In particular, if $C$ is finite and contained in one $\Grig$-orbit, then there exists $n_1$ such that the portrait of elements of $\SStab_{\Aut(T)}(C)$ contains only $1$ on vertices of level at least $n_1$ that are above~$C$.

Now, let $n\geq \max\{n_0,n_1\}$ and $g$ be in $\SStab_{\overline{\Grig}}(C)$.
As for the pointwise stabilizer, let $\{v_1,\dots,v_\abs{C}\}$ be the intersection of $C$ with level $n$ and $\{w_1,\dots,w_\abs{C}\}$ be the intersection of $C$ with level $n+4$.
By the above remark, the portrait of $g$ contains only $1$ on vertices between $v_i$ and $w_i$ included.
Hence, we can still apply Lemma~\ref{Lemma:StabilizerinT4} to obtain an element $g_n$ at distance at most $2^{-n}$ of $g$ such that the portrait of $g_n$ contains only $1$'s on vertices below the $w_i$.
Since the portraits of $g$ and $g_n$ coincide on the $n$\textsuperscript{th} level and that they both have only $1$ on vertices of level at least $n+1$ lying above elements of $C$, the action of $g_n$ on $C$ coincide with the action of $g$ on $C$.
In particular, $g_n$ belongs to $\SStab_G(C)$.
\end{proof}
We finally provide a full description of the sections of parabolic subgroups.
\begin{proposition}\label{Proposition:SectionsParabolic}
Let $W=\Stab_\Grig(\xi)$, with $\xi=(v_i)_{i\geq 0}$ and let $\sigma\colon\{0,1\}^\N\to\{0,1\}^\N$ be the shift operator: $\sigma(v_i)_{i\geq 0}=(v_{i+1})_{i\geq 0}$.
Then for all $j$ we have 
\[
	\pi_{v_j}(W)=\Stab_\Grig\bigl(\sigma^{j}(\xi)\bigr) \qquad \pi_{\bar v_j}(W)=J_{0,2}<_2\Grig
\]
where $\bar v_j$ is the only sibling of $v_j$.

Moreover, if $v\in T$ is not equal to any $v_j$ or $\bar v_j$, then $\pi_{v}(W)=\Grig$.
\end{proposition}
\begin{proof}
The description of $P$ immediately implies that for all $j$ we have 
\[
	\pi_{1^j}(P)=P \qquad \pi_{1^{j-1}0}(P)=\gen{B,(ad)^2,a}=J_{0,2}<_2 \Grig
\]
For the general case, let $\xi$ be any ray in $T$ and $W=\Stab_\Grig(\xi)$.
There exists $g\in\hat\Grig$ sending $\xi$ onto $1^\infty$.
Therefore, for all $j$, we have
\begin{align*}
	\overline{\pi_{\bar v_j}(W)}&=\pi_{\bar v_j}(\overline{W})\\
	&= \pi_{\bar v_j}\bigl(\Stab_{\hat\Grig}(\xi)\bigr)\\
	&=\pi_{1^{j-1}0}\bigl(\Stab_{\hat\Grig}(1^\infty)\bigr)^g\\
	&=\overline{J_{0,2}}^g=\overline{J_{0,2}}.
\end{align*}
Where the last equality follows from the fact that the closure preserves normality.
Finally we have $\pi_{\bar v_j}(W)=\overline{\pi_{\bar v_j}(W)}\cap G=J_{0,2}$.

For the last part, if $v$ is not equal to one off the $v_j$ or $\bar v_j$, then it is a descendant of some $\bar v_j$.
In this case, $\pi_v(W)$ is equal to the section $\pi_w(J_{0,2})$ for some vertex $w$ distinct from the root.
All these sections are equal to $\Grig$.
\end{proof}
\providecommand{\noopsort}[1]{} \def\cprime{$'$}


\begin{thebibliography}{BRLN16}

\bibitem[BG00]{MR1841750}
Laurent Bartholdi and Rostislav~I. Grigorchuk.
\newblock On the spectrum of {H}ecke type operators related to some fractal
  groups.
\newblock {\em Tr. Mat. Inst. Steklova}, 231(Din. Sist., Avtom. i Beskon.
  Gruppy):5--45, 2000.

\bibitem[BG02]{MR1899368}
Laurent Bartholdi and Rostislav~I. Grigorchuk.
\newblock On parabolic subgroups and {H}ecke algebras of some fractal groups.
\newblock {\em Serdica Math. J.}, 28(1):47--90, 2002.

\bibitem[BG{\v{S}}03]{MR2035113}
Laurent Bartholdi, Rostislav~I. Grigorchuk, and Zoran {\v{S}}uni{\'k}.
\newblock Branch groups.
\newblock In {\em Handbook of algebra, {V}ol. 3}, pages 989--1112.
  North-Holland, Amsterdam, 2003.
  
 \bibitem[BMB20]{2020arXiv200608677L}
Adrien {le Boudec} and Nicol{\'a}s {Matte Bon}.
\newblock {A commutator lemma for confined subgroups and applications to groups
  acting on rooted trees}.
\newblock {\em Trans. Amer. Math. Soc.}, 376(10):7187--7233, 2023.

\bibitem[Bon10]{MR2727305}
Ievgen~V. Bondarenko.
\newblock Finite generation of iterated wreath products.
\newblock {\em Arch. Math. (Basel)}, 95(4):301--308, 2010.

\bibitem[BRLN16]{MR3478865}
Khalid Bou-Rabee, Paul-Henry Leemann, and Tatiana Nagnibeda.
\newblock Weakly maximal subgroups in regular branch groups.
\newblock {\em J. Algebra}, 455:347--357, 2016.

\bibitem[BSZ12]{MR2891709}
Laurent Bartholdi, Olivier Siegenthaler, and Pavel Zalesskii.
\newblock The congruence subgroup problem for branch groups.
\newblock {\em Israel J. Math.}, 187:419--450, 2012.

\bibitem[CSST01]{MR1872621}
Tullio Ceccherini-Silberstein, Fabio Scarabotti, and Filippo Tolli.
\newblock The top of the lattice of normal subgroups of the {G}rigorchuk group.
\newblock {\em J. Algebra}, 246(1):292--310, 2001.

\bibitem[FAZR14]{MR3152720}
Gustavo~A. Fernández-Alcober and Amaia Zugadi-Reizabal.
\newblock {GGS}-groups: order of congruence quotients and {H}ausdorff
  dimension.
\newblock {\em Trans. Amer. Math. Soc.}, 366(4):1993--2017, 2014.

\bibitem[FG18]{MR3886188}
Dominik Francoeur and Alejandra Garrido.
\newblock Maximal subgroups of groups of intermediate growth.
\newblock {\em Adv. Math.}, 340:1067--1107, 2018.

\bibitem[FGLN24]{2024arXiv240215496F}
Dominik {Francoeur}, Rostislav {Grigorchuk}, Paul-Henry {Leemann}, and Tatiana
  {Nagnibeda}.
\newblock {On the structure of finitely generated subgroups of branch groups}.
\newblock {\em arXiv e-prints}, page arXiv:2402.15496, February 2024.

\bibitem[FL20]{FL2019}
Dominik Francoeur and Paul-Henry Leemann.
\newblock Subgroup induction property for branch groups.
\newblock {\em arXiv e-prints}, page arXiv:2011.13310, November 2020.

\bibitem[Fra19]{FrancoeurThese}
Dominik Francoeur.
\newblock {\em On maximal subgroups and other aspects of branch groups}.
\newblock PhD thesis, Université de Genève, 2019.

\bibitem[Gar16a]{MR3513107}
Alejandra Garrido.
\newblock Abstract commensurability and the {G}upta-{S}idki group.
\newblock {\em Groups Geom. Dyn.}, 10(2):523--543, 2016.

\bibitem[Gar16b]{MR3556961}
Alejandra Garrido.
\newblock On the congruence subgroup problem for branch groups.
\newblock {\em Israel J. Math.}, 216(1):1--13, 2016.

\bibitem[GLN21]{MR4344374}
Rostislav~I. Grigorchuk, P.-H. Leemann, and T.~V. Nagnibeda.
\newblock Finitely generated subgroups of branch groups and subdirect products
  of just infinite groups.
\newblock {\em Izv. Ross. Akad. Nauk Ser. Mat.}, 85(6):104--125, 2021.

\bibitem[Gri84]{MR764305}
Rostislav~I. Grigorchuk.
\newblock Degrees of growth of finitely generated groups and the theory of
  invariant means.
\newblock {\em Izv. Akad. Nauk SSSR Ser. Mat.}, 48(5):939--985, 1984.

\bibitem[Gri00]{MR1765119}
Rostislav~I. Grigorchuk.
\newblock Just infinite branch groups.
\newblock In {\em New horizons in pro-{$p$} groups}, volume 184 of {\em Progr.
  Math.}, pages 121--179. Birkhäuser Boston, Boston, MA, 2000.

\bibitem[Gri05]{MR2195454}
Rostislav~I. Grigorchuk.
\newblock Solved and unsolved problems around one group.
\newblock In {\em Infinite groups: geometric, combinatorial and dynamical
  aspects}, volume 248 of {\em Progr. Math.}, pages 117--218. Birkhäuser,
  Basel, 2005.

\bibitem[Gri11]{MR2893544}
Rostislav~I. Grigorchuk.
\newblock Some problems of the dynamics of group actions on rooted trees.
\newblock {\em Tr. Mat. Inst. Steklova}, 273(Sovremennye Problemy
  Matematiki):72--191, 2011.

\bibitem[GS83]{MR696534}
Narain Gupta and Saïd Sidki.
\newblock On the {B}urnside problem for periodic groups.
\newblock {\em Math. Z.}, 182(3):385--388, 1983.

\bibitem[GW03a]{MR2009443}
Rostislav~I. Grigorchuk and John~S. Wilson.
\newblock A structural property concerning abstract commensurability of
  subgroups.
\newblock {\em J. London Math. Soc. (2)}, 68(3):671--682, 2003.

\bibitem[GW03b]{MR2011117}
Rostislav~I. Grigorchuk and John~S. Wilson.
\newblock The uniqueness of the actions of certain branch groups on rooted
  trees.
\newblock {\em Geom. Dedicata}, 100:103--116, 2003.

\bibitem[Har00]{MR1786869}
Pierre de~la Harpe.
\newblock {\em Topics in geometric group theory}.
\newblock Chicago Lectures in Mathematics. University of Chicago Press,
  Chicago, IL, 2000.
  
\bibitem[KT18]{MR3834728}
Benjamin Klopsch and Anitha Thillaisundaram.
\newblock Maximal {S}ubgroups and {I}rreducible {R}epresentations of
  {G}eneralized {M}ulti-{E}dge {S}pinal {G}roups.
\newblock {\em Proc. Edinb. Math. Soc. (2)}, 61(3):673--703, 2018.

\bibitem[Lee16]{LeemannThese}
Paul-Henry Leemann.
\newblock {\em On Subgroups and Schreier Graphs of Finitely Generated Groups}.
\newblock PhD thesis, Université de Genève, 2016.

\bibitem[MW17]{MR3592603}
François le~Maître and Phillip Wesolek.
\newblock On strongly just infinite profinite branch groups.
\newblock {\em J. Group Theory}, 20(1):1--32, 2017.

\bibitem[Myr15]{2015arXiv150908090M}
Aglaia Myropolska.
\newblock {The class $\mathcal{MN}$ of groups in which all maximal subgroups
  are normal}.
\newblock {\em arXiv e-prints}, page arXiv:1509.08090, Sep 2015.

\bibitem[Per00]{MR1841763}
Ekaterina~L. Pervova.
\newblock Everywhere dense subgroups of a group of tree automorphisms.
\newblock {\em Tr. Mat. Inst. Steklova}, 231(Din. Sist., Avtom. i Beskon.
  Gruppy):356--367, 2000.

\bibitem[Per05]{MR2197824}
Ekaterina~L. Pervova.
\newblock Maximal subgroups of some non locally finite {$p$}-groups.
\newblock {\em Internat. J. Algebra Comput.}, 15(5-6):1129--1150, 2005.

\bibitem[Per07]{MR2308183}
Ekaterina~L. Pervova.
\newblock Profinite completions of some groups acting on trees.
\newblock {\em J. Algebra}, 310(2):858--879, 2007.

\bibitem[UA16]{Uria}
Jone Uria-Albizuri.
\newblock On the concept of fractality for groups of automorphisms of a regular
  rooted tree.
\newblock {\em Reports@ SCM}, 2(1):33--44, 2016.

\bibitem[Vov00]{MR1754681}
Taras Vovkivsky.
\newblock Infinite torsion groups arising as generalizations of the second
  {G}rigorchuk group.
\newblock In {\em Algebra ({M}oscow, 1998)}, pages 357--377. de Gruyter,
  Berlin, 2000.

\bibitem[Wil09]{MR2605182}
John~S. Wilson.
\newblock Structure theory for branch groups.
\newblock In {\em Geometric and cohomological methods in group theory}, volume
  358 of {\em London Math. Soc. Lecture Note Ser.}, pages 306--320. Cambridge
  Univ. Press, Cambridge, 2009.

\end{thebibliography}
\end{document}